\documentclass[a4paper,11pt,reqno]{amsart}
\usepackage{amsfonts}
\usepackage{amsmath}
\usepackage{logicproof}
\usepackage{amssymb, amsthm}
\usepackage{tabularx}
\usepackage{caption}
\usepackage{booktabs}
\usepackage{multirow,dsfont}
\usepackage{array}
\usepackage{floatrow}
\usepackage{float}
\usepackage{floatpag}
\usepackage{centernot}
\usepackage{enumitem}
\newcolumntype{L}[1]{>{\raggedright\let\newline\\\arraybackslash\hspace{0pt}}m{#1}}
\newcolumntype{C}[1]{>{\centering\let\newline\\\arraybackslash\hspace{0pt}}m{#1}}
\newcolumntype{R}[1]{>{\raggedleft\let\newline\\\arraybackslash\hspace{0pt}}m{#1}}
\usepackage[symbol]{footmisc}
\usepackage[labelfont=bf,format=plain,justification=raggedright,singlelinecheck=false]{caption}
\usepackage{mathrsfs}
\usepackage[]{mdframed}
\usepackage[colorlinks]{hyperref}
\usepackage{tcolorbox}
\usepackage{adjustbox}
\tcbuselibrary{theorems}
\newtcbtheorem[number within=section]{mytheo}{Note}
{colback=white!5,colframe=black!35!black,fonttitle=\bfseries}{th}

\numberwithin{equation}{section}

\usepackage{graphicx}
\usepackage{textgreek}

 {
      \theoremstyle{plain}
      \newtheorem{assumption}{Assumption}
  }
\newtheorem{theorem}{Theorem}[section]

\newtheorem{remark}[theorem]{Remark}

\newtheorem{lemma}[theorem]{Lemma}

\begin{document}
	
	\title[General Regularization for Functional Linear Regression]{Optimal rates for functional linear regression with General regularization}
\author[N. Gupta]{Naveen Gupta}
\address[N. Gupta]{Department of Mathematics, Indian Institute of Technology Delhi, India}
\email{ngupta.maths@gmail.com}
\author{S. Sivananthan}
\address[S. Sivananthan]{Department of Mathematics, Indian Institute of Technology Delhi, India}
\email{siva@maths.iitd.ac.in}
\author[B. K. Sriperumbudur]{Bharath K. Sriperumbudur}
\address[B. K. Sriperumbudur]{Department of Statistics, Pennsylvania State University, USA}
\email{bks18@psu.edu}

\begin{abstract}
Functional linear regression is one of the fundamental and well-studied methods in functional data analysis.  
   In this work, we investigate the functional linear regression model within the context of reproducing kernel Hilbert space by employing general spectral regularization to approximate 
   the slope function with certain smoothness assumptions.  
   We establish optimal convergence rates for estimation and prediction errors associated with the proposed method under H\"{o}lder type source condition, which generalizes and sharpens all the known results in the literature.
\end{abstract}

	\maketitle

\section{introduction}\label{sec:intro}
Functional data analysis (FDA) deals with analyzing and extracting information from curves, known as functional data, that include growth curves, weather data, health data, etc. The functional linear regression (FLR) model introduced by 
Ramsay and Dalzell \cite{ramsay1991some} has emerged as one of the mainstays for analyzing functional data \cite{ramsay2002afda,ramsaywhendataarefunctions,kokoszka2017,reiss2017,morris2015functional,wang2016functional}.
Formally, the FLR model is stated as 
\begin{equation}
\label{model}
    Y = \int_{S} X(t) \beta^*(t)\, dt + \epsilon,
\end{equation}
where $Y(\omega)$ is a real-valued random variable, $(X(\omega,t); t\in S )$ is a continuous time process, $\epsilon$ is a zero mean random noise, independent of $X$, with finite variance $\sigma^2$ and $\beta^*$ is an unknown slope function.
Throughout the paper, we assume that $X$ and $\beta^*$ are in $L^2(S)$, $\mathbb{E}[\|X\|_{L^2}^2]$ is finite, and $S$ is a compact subset of $\mathbb{R}^d$. 
From \eqref{model}, it is easy to note that for any $s\in S$, $\mathbb{E}[YX(s)]=\mathbb{E}[\langle X,\beta^*\rangle_{L^2} X(s)]+\mathbb{E}[\epsilon X(s)]$ which implies $\mathbb{E}[XY](s)=(C\beta^*)(s)$, where $C:=\mathbb{E}[X\otimes X]$ is the covariance operator. Therefore, the slope function $\beta^*$ satisfies the operator equation $$C\beta^*=\mathbb{E}[XY],$$ which can be alternately expressed as  
\begin{equation}
\label{true solution}
    \beta^* := \arg\min_{\beta \in L^2(S)}\mathbb{E}[Y-\langle X,\beta \rangle_{L^2} ]^2.
\end{equation}
The goal is to construct an estimator $\hat{\beta}$ to estimate the unknown slope function $\beta^*$ using observed empirical data $\{(X_1,Y_1),(X_2,Y_2),\ldots, (X_n,Y_n)\}$. Here, $X_i$'s are i.i.d.~copies of the random process $X(\omega,\cdot)$, and $Y_{i}$'s are scalar responses. A general study by taking $Y_{i}$'s to be functional responses has been carried out in \cite{tong2022functionalresponses,mollenhauer2022learning,benatia2017,cuevas2002}.

One of the initial methodologies developed for estimating the slope function is functional principal component analysis (FPCA) \cite{cardot1999, muller2005generalized, cai2006prediction,li2007rates, Zhu2014}, which relies on representing the estimator of $\beta^*$ with respect to a basis, for example, a B-spline basis as in \cite{cardot2003spline}, or more popularly the eigenfunctions of the covariance operator.
Under a smoothness condition on the slope function, \cite{hall2007methodology} provided an optimal convergence rate for the FPCA method using covariance eigenfunctions. Recently, Holzleitner and Pereverzyev \cite{polynomial2023regularization} explored the FPCA method under general regularization---the aforementioned works use spectral cut-off as the regularizer---, for the functional polynomial regression model with general source condition. We refer the reader to \cite{wang2016functional, multivariateFPCA2018} and references therein for a detailed account of the FPCA approach.

A different approach to address the functional linear regression (FLR) model is inspired by the approach of kernel methods for solving inverse problems. In learning theory, it is well established that kernel methods represent predictor functions using data-driven kernel functions, yielding favorable generalization performance. A comprehensive analysis of kernel methods in the context of the kernel ridge regression problem can be found in \cite{smale2002foundation,sergeionregularization,cuckerzhou2007learningtheory,AI2022sergei,lin2020distributed_stochastic,guo2017distributed}. The application of kernel methods for kernel ridge regression in a general Hilbert space has been investigated in \cite{lin2017multipass,lin2020optimalspectral,lin2020stochastic}, encompassing classical learning theory problems and bearing a close connection to the FLR model. The key idea in these works is that since the FLR model satisfies the operator equation $C\beta^*=\mathbb{E}[XY]$, an estimator of $\beta^*$ based on samples $(X_i,Y_i)^n_{i=1}$ can be constructed as $\check{\beta}=\frac{1}{n}\sum^n_{i=1}g_\lambda(\hat{C}_n)X_iY_i$, where $\hat{C}_n=\frac{1}{n}\sum^n_{i=1}X_i\otimes_{L^2(S)}X_i:L^2(S)\rightarrow L^2(S)$ is the empirical estimator of $C$ and $g_\lambda$ is the spectral regularizer. Let $\mathscr{H}(L^2(S))$ be a reproducing kernel Hilbert space (RKHS) defined on $L^2(S)$ with $k(x,y)=\langle x,y\rangle_{L^2(S)},\,x,y\in L^2(S)$, as the reproducing kernel, i.e., a linear kernel on $L^2(S)$. Since $\mathscr{H}(L^2(S))$ is isometrically isomorphic to $L^2(S)$, it is easy to see that this approach corresponds to learning a function in $\mathscr{H}(L^2(S))$. Therefore, exploiting the general kernel regression theory for learning in RKHS,  \cite{lin2017multipass,lin2020optimalspectral,lin2020stochastic} obtained convergence rates for $\check{\beta}$ and established their minimax optimality under the standard source condition involving the integral operator, which in this case is nothing but $C$ since $k$ is a linear kernel. 
It has to be noted that this approach also closely aligns with the classical FPCA method, where the unknown slope function is expressed via the eigenfunctions of $C$.

On the other hand, in the FDA community, the operator equation $C\beta^*=\mathbb{E}[XY]$ has been approximated in an alternate way (compared to the above) to obtain an estimator of $\beta^*$. Instead of estimating $\beta^*$ through a function in $\mathscr{H}(L^2(S))$ as discussed above, the FDA literature deals with estimating $\beta^*\in L^2(S)$ through a function in $\mathcal{H}\subset L^2(S)\equiv \mathscr{H}(L^2(S))$, where $\mathcal{H}$ is an RKHS defined on $S$ through the reproducing kernel, $K$ that is free for the user to choose.  
The details of the construction of the estimator $\hat{\beta}$ are provided in Section~\ref{preliminaries}. It has to be noted that $\mathscr{H}(L^2(S))$ is determined by the covariance operator, $C$ while $\mathcal{H}$ is usually independent of $C$ and is free to be chosen by the user. Therefore, the analysis of $\hat{\beta}$ depends on both $C$ and the integral operator $T:L^2(S)\rightarrow L^2(S),\,f\mapsto \int_S K(\cdot,x)f(x)\,dx$, and their interaction. This means the convergence rates of $\hat{\beta}$ to $\beta^*$ could exhibit different behavior depending on whether $C$ and $T$ commute or not, in contrast to $\check{\beta}$. 
Hence, the analysis in this approach is not a direct consequence of the classical learning theory analysis of kernel-based nonparametric regression in RKHS. We refer the reader to Section~\ref{Sec:similar} for a detailed comparison of these approaches.

The goal of this paper is to investigate the RKHS-based approach for the FLR model, obtain optimality results for both commutative and non-commutative settings, close the gaps in the literature thereby obtaining closure on this approach, and also explore the connection to the learning theory-based RKHS approach. Before discussing our contributions, we first establish the context by discussing the results in the FDA literature for this approach.
The above-mentioned alternate RKHS approach for FLR was first proposed and investigated by Yuan and Cai \cite{ARKHSFORFLR}, wherein they developed an estimator of $\beta^*$ based on Tikhonov regularization and established its optimality under the commutativity of the integral operator, $T$ associated with the reproducing kernel, $K$ and the covariance operator, $C$, assuming $\beta^*\in\mathcal{H}$. 
Under the same assumption of $\beta^*\in\mathcal{H}$, \cite{tonyyuan2012minimax,HTONGFLR} established optimal convergence rates when the integral and covariance operators do not commute. On the other hand, Zhang et al. \cite{ZhangFaster2020} considered a source condition of $\beta^*\in\mathcal{R}(T^{\frac{1}{2}}(T^{\frac{1}{2}}CT^{\frac{1}{2}})^s)$ and derived optimal convergence rates for the estimation and prediction errors associated with the estimator for $s\in(0,1]$ and $s\in(0,\frac{1}{2}]$, respectively, where $\mathcal{R}(A)$ denotes the range space of the operator $A$. 
Balasubramanian et al. \cite{balasubramanian2024functional} considered a non-linear extension to the FLR model, called the functional single-index model $(Y = g( \langle X, \beta^*  \rangle_{L^2}) + \epsilon )$ where $g$ is a non-linear function, and showed that if $X$ is a Gaussian process with covariance operator $C$, then under appropriate regularity condition on $g$, the slope function $\beta^*$ of the single-index model satisfies the operator equation, $\mathbb{E}[XY]=C\beta^*$, which is also satisfied by the slope function of the FLR model. Therefore, they analyzed the functional single-index model by analyzing the FLR model, wherein they considered a least squares approach with Tikhonov regularization (similar to \cite{ARKHSFORFLR}) to obtain an estimator of $\beta^*$. Under the source condition of $\beta^*\in \mathcal{R}(T^\alpha)\subset L^2(S)\backslash \mathcal{H},\,\alpha\in(0,\frac{1}{2})$---note that this is weaker than that used in \cite{ARKHSFORFLR} which uses $\alpha=\frac{1}{2}$, i.e., $\beta^*\in\mathcal{R}(T^{\frac{1}{2}})=\mathcal{H}$---,  \cite{balasubramanian2024functional} established convergence rates for estimation and prediction errors in the commutative setup, extending the results of \cite{ARKHSFORFLR}. \cite{balasubramanian2024functional} also obtained convergence rates for the estimation and prediction errors in the non-commutative setting under the assumption that $\beta^*\in \mathcal{R}(T^{\frac{1}{2}}(T^{\frac{1}{2}}CT^{\frac{1}{2}})^s),\,0\le s\le1$, which partially extends the results of \cite{tonyyuan2012minimax,HTONGFLR,ZhangFaster2020}. 

In the context of the above-mentioned results, many open questions need to be addressed: (i) the rates for a large range of smoothness, i.e., $\alpha>\frac{1}{2}$ (in the commutative setting) and $s>1$ (in the non-commutative setting) is not known, (ii) the optimality of the rates is not known in the commutative setting except for $\alpha=\frac{1}{2}$, and in the non-commutative setting for $s>1$, and (iii) the impact of qualification of a general regularization family on the convergence rate is not known since all the above-mentioned works employ Tikhonov regularization.  
The main contribution of this paper is to address these limitations to provide closure on the FDA-related RKHS approach for FLR while throwing light on the relationship between this approach and the learning theory-based approach. To this end, we undertake a comprehensive examination of this approach by using a general regularization scheme (subsuming Tikhonov regularization) that caters to a wide range of smoothness and obtains optimal convergence rates for the estimation and prediction errors in the commutative and non-commutative settings. The details of these contributions are elaborated below.

\subsection{Contributions}
    \emph{(i)} We propose an estimator of $\beta^*$ based on generalized regularization that subsumes Tikhonov regularization (which was already considered in the literature) so as to understand the impact of the source condition and the qualification (of the regularization) on the convergence rates.\vspace{1.5mm}

    \emph{(ii)} Under the source condition, $\beta^* \in \mathcal{R}(T^{\alpha})$, $\alpha >0$, in Theorems \ref{comm.estimation} and \ref{predictionbasic}, we present the convergence rates for the estimation and prediction errors, respectively, in the commutative setting (i.e., $T$ and $C$ commute). These results match the optimal results of \cite{ARKHSFORFLR} when $\alpha=\frac{1}{2}$ and improve upon the rates of \cite{balasubramanian2024functional} for $0<\alpha<\frac{1}{2}$ while providing convergence rates for $\alpha>\frac{1}{2}$ but up to a smoothness index determined by the qualification of the general regularization. If the qualification is infinity, then Theorems \ref{comm.estimation} and \ref{predictionbasic} provide convergence rates for all orders of smoothness, i.e., $\alpha>0$. We would like to highlight that until this work, the convergence rates are known only for $0<\alpha\le\frac{1}{2}$ and the optimality of these rates are known only for $\alpha=\frac{1}{2}$. In this work, we not only resolve the issue of catering to $\alpha>\frac{1}{2}$ but also establish the optimality of the rates in the commutative setting for $\alpha>0$ by deriving lower bounds in Theorem~\ref{lowercomm}. A summary of our results in relation to those in the literature is captured in Table~\ref{table:1}.\vspace{1.5mm}

\begin{center}
    \begin{tabular}[t]{|C{2.9cm}|C{2.9cm}|C{4.0cm}|C{3.6cm}|}
        \hline
           &  &  \multirow{2}{*}{} & \multirow{2}{*}{} \\
          Results $(\text{Regularization})$ & Assumption & Estimation error $\|\beta^*-\hat{\beta}\|_{L^2}$ & Prediction error $\|C^{\frac{1}{2}}(\beta^*-\hat{\beta})\|^2_{L^2}$  \\
          &  &  \multirow{2}{*}{} & \multirow{2}{*}{} \\
        \hline
         \multirow{2}{*}{Cai and Yuan\cite{ARKHSFORFLR} } &\multirow{2}{*}{$\beta^* \in \mathcal{R}(T^{\frac{1}{2}})$} & \multirow{2}{*}{$\asymp n^{-\frac{t}{2(1+t+c)}}$} & \multirow{2}{*}{$\asymp n^{-\frac{t+c}{t+c+1}}$} \\
         $(\text{Tikhonov})$ & & &  \\
          \hline
          \multirow{2}{*}{Balasubramanian } & \multirow{2}{*}{$\beta^* \in \mathcal{R}(T^{\alpha} )$} & \multirow{2}{*}{$\lesssim n^{-\frac{\alpha t}{1+c+2t(1-\alpha)}}$} & \multirow{2}{*}{$ \lesssim n^{-\frac{2 \alpha t + c}{1+c + 2t(1-\alpha)}} $} \\
          et al. \cite{balasubramanian2024functional} $(\text{Tikhonov})$ &$(0 < \alpha \leq \frac{1}{2})$ & &   \\
          \hline
          \multirow{2}{*}{This paper  }&\multirow{2}{*}{$\beta^* \in \mathcal{R}(T^{\alpha})$} &\multirow{2}{*}{$ \lesssim n^{-\frac{r t}{1+c+2r t}}$} &\multirow{2}{*}{$ \lesssim n^{-\frac{2 r t + c}{1+c+2r t}}$}\\
          Theorems~\ref{comm.estimation}, \ref{predictionbasic} $(\text{General})$ &   & $r=\alpha \wedge (\nu + \frac{c}{t}(\nu-\frac{1}{2}))$
          & $r=\alpha \wedge (\nu+\frac{1}{2}+ \nu \frac{c}{t})$\\
           \hline
           \multirow{2}{*}{ }&\multirow{2}{*}{} &\multirow{2}{*}{} &\multirow{2}{*}{}\\
         This paper & $\beta^* \in \mathcal{R}(T^{\alpha})$   & $\gtrsim n^{-\frac{\alpha t}{1+c+2\alpha t}}$  & $\gtrsim n^{-\frac{2 \alpha t + c}{1+c+2\alpha t}}$  \\
           Theorem \ref{lowercomm} &  $(0< \alpha )$ &  &  \\
          \hline
    \end{tabular}
    \centering
    \captionof{table}{Convergence rates for FLR in the RKHS framework when the kernel integral operator $T$ and the covariance operator $C$ commute. Here $\nu$ denotes the qualification of the regularization and $a\wedge b:= \min\{a,b\}.$}\vspace{0.2mm}
    \label{table:1}
\end{center}
     \emph{(iii)} We provide convergence rates for the estimation and prediction errors in the non-commutative setting (i.e., without assuming the commutativity of $T$ and $C$) in Theorems~\ref{NCestimation} and \ref{predictionNC}, respectively,  under the source condition $\beta^* \in \mathcal{R}(T^{\frac{1}{2}}(T^{\frac{1}{2}}CT^{\frac{1}{2}})^s)$ where $s>0$. These results generalize those of \cite{balasubramanian2024functional,tonyyuan2012minimax,ZhangFaster2020} by investigating rates for smoothness index, $s> 1$ since the prior results are known only for $0\le s\le 1$ for the estimation error and $0 \le s\leq \frac{1}{2}$ for the prediction error. Moreover, the optimality is known only for the prediction error for $0 \le s\leq \frac{1}{2}$. We address this issue by showing the rates of Theorems~\ref{NCestimation} and \ref{predictionNC} to be optimal by deriving lower bounds in Theorem~\ref{lowerNC}. A summary of our contributions in the non-commutative setting in contrast to the existing results in the literature is provided in Table~\ref{table:2}.\vspace{1mm}\\
     \indent \emph{(iv)} A key highlight of our work is that the proposed estimator is minimax optimal both in the commutative and non-commutative settings without using the knowledge of commutativity or non-commutativity in its construction. Moreover, in terms of the source condition, the commutative setting allows us to handle the misspecified case of the true regressor being in $L^2(S)\backslash \mathcal{H}$.
     \vspace{1mm}\\
    \indent \emph{(v)} Since the slope function of a single-index model can be estimated using the FLR estimator if $X$ is Gaussian (as shown in \cite{balasubramanian2024functional}), our results discussed in \emph{(i)}--\emph{(iii)} cater to the single-index model as well, thereby sharpening and generalizing the existing results in the literature (see Remark~\ref{rem}).\vspace{1mm}\\
    \indent \emph{(vi)} When $T$ and $C$ commute, we show that the rates derived in Theorem \ref{comm.estimation} and Theorem \ref{predictionbasic} under H\"{o}lder type source condition match the convergence rates for the estimation and the prediction error achieved in \cite{lin2020optimalspectral} for the FLR model (under the learning theory framework), i.e., both the FDA and learning theory approaches are similar except for the assumptions (see Remark~\ref{rem:compare}). However, in terms of the rates, the non-commutative scenario explored in our analysis is in general not comparable to the learning theory approach (see Remark~\ref{rem:compare1}). \vspace{-1.5mm}
    
\begin{center}
    \begin{tabular}[t]{|C{3.2cm}|C{3.8cm}|C{2.8cm}|C{3.8cm}|}
        \hline
         &  & \multirow{2}{*}{}  & \multirow{2}{*}{} \\
          Results $(\text{Regularization})$& Assumption& Estimation error $\|\beta^*-\hat{\beta}\|$ & Prediction error $\|C^{\frac{1}{2}}(\beta^*-\hat{\beta})\|^2_{L^2}$\\
         & &&  \\
        \hline
         \multirow{2}{*}{Yuan and Cai \cite{tonyyuan2012minimax}} &\multirow{2}{*}{$\beta^* \in \mathcal{R}(T^{\frac{1}{2}})$} & \multirow{2}{*}{N.A.} & \multirow{2}{*}{$\asymp n^{-\frac{b}{b+1}}$} \\
         $(\text{Tikhonov})$ & & &  \\
          \hline
          \multirow{2}{*}{Tong and Ng\cite{HTONGFLR}}& \multirow{2}{*}{$\beta^* \in \mathcal{R}(T^{\frac{1}{2}})$} & \multirow{2}{*}{N.A.} & \multirow{2}{*}{$\lesssim n^{-\frac{b}{b+1}}$}\\
           $(\text{Tikhonov})$& & & without Assumption \ref{as:5}\\
          \hline
          \multirow{3}{*}{Zhang et al.\cite{ZhangFaster2020}} & \multirow{3}{*}{$\beta^* \in \mathcal{R}(T^{\frac{1}{2}}(T^{\frac{1}{2}}CT^{\frac{1}{2}})^s )$}& \multirow{3}{*}{$\asymp n^{-\frac{sb}{1+b+2sb}}$} & \multirow{3}{*}{$\asymp n^{-\frac{(b+2sb)}{1+b+2sb}}$}  \\
           & & &  \\
           $(\text{Tikhonov})$ &$(0 < s \leq 1)$ &(RKHS norm) & $(0 \leq s \leq \frac{1}{2})$  \\
          \hline
          \multirow{2}{*}{Balasubramanian } & \multirow{2}{*}{$\beta^* \in \mathcal{R}(T^{\frac{1}{2}}(T^{\frac{1}{2}}CT^{\frac{1}{2}})^s )$} & \multirow{2}{*}{$\lesssim n^{-\frac{sb}{1+b+2sb}}$} & \multirow{2}{*}{$\lesssim n^{-\frac{b}{1+b}}~;\beta^* \in \mathcal{R}(T^{\frac{1}{2}}) $} \\
          et al.\cite{balasubramanian2024functional} &&&\\
         $(\text{Tikhonov})$ &$(0 < s \leq 1)$ &($L^2$ norm) &   \\
          \hline
          \multirow{3}{*}{ }&\multirow{3}{*}{} &\multirow{3}{*}{} &\multirow{3}{*}{}\\
          This paper Theorems \ref{NCestimation}, \ref{predictionNC} & $\beta^* \in \mathcal{R}(T^{\frac{1}{2}}(T^{\frac{1}{2}}CT^{\frac{1}{2}})^s )$ $(0 < s \leq \nu)$  & $\lesssim n^{-\frac{rb}{1+b+2rb}}$ (RKHS norm) $r= s\wedge \nu$ & $\lesssim n^{-\frac{(b+2rb)}{1+b+2rb}}$ $r= s\wedge (\nu-\frac{1}{2})$\\
          $(\text{General})$&&&\\
          \hline
        \multirow{3}{*}{ }&\multirow{3}{*}{} &\multirow{3}{*}{} &\multirow{3}{*}{}\\
        This paper Theorem \ref{lowerNC}& $\beta^* \in \mathcal{R}(T^{\frac{1}{2}}(T^{\frac{1}{2}}CT^{\frac{1}{2}})^s )$ $(s>0)$ & $\gtrsim n^{-\frac{sb}{1+b+2sb}}$ (RKHS norm) & $\gtrsim n^{-\frac{(b+2sb)}{1+b+2sb}}$ \\
          \hline
    \end{tabular}
    \centering
    \captionof{table}{Convergence rates for FLR in the RKHS framework when the kernel integral operator $T$ and the covariance operator $C$ do not commute. Here $\nu$ denotes the qualification of the regularization and $a\wedge b:= \min\{a,b\}.$}
    \label{table:2}
\end{center}     

\subsection{Organization}
The paper is organized as follows. Section~\ref{preliminaries} presents the necessary background for the regularization method and the FLR model in the RKHS setting and provides some elementary results that are necessary for analysis. In Section~\ref{main results}, we discuss our assumptions and provide convergence rates for the estimation and prediction errors in both commutative and non-commutative settings for the general regularization method. We provide matching lower bounds for these errors in Section~\ref{lower bounds}, which establish the optimality of the general regularization method in the FLR setting. In Section~\ref{Sec:similar}, we provide a detailed discussion comparing the RKHS-based approach for the FLR model with its learning theory counterpart and the non-parametric regression in RKHS. The proofs of all the results are presented in Section~\ref{proofs}, while supplementary results needed to prove the main results are relegated to Appendices \ref{supplements}--\ref{equivalentbounds}.
\subsection{Notations} $L^2(S)$ denotes the space of all real-valued square-integrable functions defined on $S$. For $f, g \in L^2(S)$, $L^2$ inner product and norm are defined as $\langle f, g \rangle_{L^2} = \int_{S}f(x)g(x)\, dx$ and $\|f\|^2_{L^2}= \langle f, f \rangle_{L^2}$. The inner product and norm associated with the RKHS, $\mathcal{H}$ are denoted as $\langle \cdot, \cdot \rangle_{\mathcal{H}}$ and $\|\cdot\|_{\mathcal{H}}$, respectively. For an operator $A$, $\mathcal{R}(A)$ denotes the range of operator $A$. For an operator $A : L^2 \to L^2$, the operator norm is defined as 
$\|A\|_{L^2 \to L^2} = \sup \{\|Af\|_{L^2} | f \in L^2, \|f\|_{L^2}=1\}.$
Similarly, for an operator $A : L^2 \to \mathcal{H}$, the operator norm is defined as 
$\|A\|_{L^2 \to \mathcal{H}} = \sup \{\|Af\|_{\mathcal{H}}: f \in L^2, \|f\|_{L^2}=1\}$. For two positive numbers $a$ and $b$, $a \lesssim b$ means $a \leq cb $ for some positive constant $c$.  For positive sequences $(a_{k})_k$ and $(b_{k})_k$, $a_{k} \asymp b_{k}$ means $a_k \lesssim b_k \lesssim a_k$ for all $k$.  For a random variable $W$ with law $P$ and a constant $b$, $W\lesssim_p b$ denotes that for
any $\delta > 0$, there exists a positive constant $c_\delta<\infty$  such that $P(W\le c_\delta b)\ge \delta$. For notational convenience, we define $\Lambda:= T^{\frac{1}{2}}CT^{\frac{1}{2}}$ and $\hat{\Lambda}_{n}:= T^{\frac{1}{2}}\hat{C}_{n}T^{\frac{1}{2}}$, where $\hat{C}_n$ is an empirical estimator of $C$ (see Section~\ref{preliminaries} for the definition of $\hat{C}_{n}$). We define $a\wedge b:=\min\{a,b\}$ for $a,b\in\mathbb{R}$.
\section{preliminaries}
\label{preliminaries}
A Hilbert space $\mathcal{H}$ of real-valued functions on $S$ is called an RKHS, if the point-wise evaluation functional is linear and continuous. There is a one-to-one correspondence between an RKHS  $\mathcal{H}$ and a reproducing kernel $k$ which is a symmetric and positive definite function from $S\times S$ to $\mathbb{R}$ such that $k(s,\cdot) \in \mathcal{H}$ and
 $f(s) = \langle k(s,\cdot),f \rangle_{\mathcal{H}}, \,\, \forall f \in \mathcal{H}.$ We assume that $k$ is measurable and $\sup_{x\in S}k(x,x)\leq \kappa $, where $\kappa $ is a positive constant. Then the elements in the associated RKHS $\mathcal{H}$ are measurable and bounded functions on $S.$
It can be easily seen that $\mathcal{H}$ is continuously embedded in $L^2(S)$ \cite{cuckerzhou2007learningtheory}. We refer the reader to  \cite{aronszajn1950rkhs,SVM2008steinwart,paulsen2016rkhs,AI2022sergei} for more details on RKHS.

We define the integral operator $T:= JJ^* : L^2(S) \to L^2(S)$, where $J: \mathcal{H} \to L^2(S) $ is the inclusion operator defined as $(Jf)(x) = \langle k(x,\cdot) ,f \rangle_{\mathcal{H}},\,\,\forall x\in S$ and $J^*:L^2(S) \to \mathcal{H}$ is the adjoint of the inclusion operator given as
$J^*g= \int_S k(\cdot,t)g(t)\, dt.$ It is easy to verify that $T^{\frac{1}{2}}$ is a partial isometry from $L^2(S)$ to $\mathcal{H}$, consequently $\mathcal{R}(T^{\frac{1}{2}}) = \mathcal{H}$. Moreover, for $\alpha > \frac{1}{2}$,  $\mathcal{R}(T^{\alpha}) \subset \mathcal{H}$. In other words, coefficients of a representation of a function, $f \in \mathcal{R}(T^{\alpha})$ in terms of eigenfunctions of the operator $T$, will have a faster decay for a larger value of $\alpha$.

Motivated by the observation that $\beta^*$ solves \eqref{true solution}, using the given data $\{(X_{i},Y_{i})\}_{i=1}^{n}$, we consider the following estimator of $\beta^*$:
\begin{equation}
\label{minimization equation}
    \begin{split}
        \hat{\beta} & = \arg\min_{\beta \in \mathcal{H}}\frac{1}{n}\sum_{i=1}^n[Y_i-\langle \beta, X_i \rangle_{L^2} ]^2
         = \arg\min_{\beta \in \mathcal{H}}\frac{1}{n}\sum_{i=1}^n[Y_i-\langle J\beta, X_i \rangle_{L^2} ]^2\\
        & = \arg\min_{\beta \in \mathcal{H}}\frac{1}{n}\sum_{i=1}^n[Y_i-\langle \beta, J^*X_i \rangle_{L^2} ]^2.\\
    \end{split}
\end{equation}
Let us define
$$\hat{C}_{n} := \frac{1}{n}\sum_{i=1}^{n} X_{i} \otimes X_{i} \quad \text{ and } \quad \hat{R} := \frac{1}{n} \sum_{i=1}^{n} Y_{i}X_{i}. $$
It can be easily verified that the solution of the equation $(\ref{minimization equation})$ can be given by solving the operator equation
\begin{equation}
    \label{estimator}
   J^*\hat{C}_{n}J \hat{\beta} = J^*\hat{R}.
\end{equation}
Equation $(\ref{estimator})$ represents a discretized version of the equation  $J^*CJ\beta^* = J^*\mathbb{E}[XY]$, which is an ill-conditioned problem. Regularization techniques have been well explored in the learning theory framework to solve such ill-posed problems \cite{sergeionregularization}. A detailed study on the Tikhonov regularization technique to solve $(\ref{estimator})$ has been carried out in \cite{tonyyuan2012minimax,ZhangFaster2020,balasubramanian2024functional,HTONGFLR,TONGHUBER}. In this paper, we use the general regularization techniques from the theory of ill-posed inverse problems to solve $(\ref{estimator})$, yielding the regularized estimator
\begin{equation*}
    \hat{\beta}_\lambda = g_{\lambda}(J^*\hat{C}_{n}J)J^*\hat{R},
\end{equation*}
where $\lambda>0$ is the regularization parameter. For simplicity, we remove the subscript $\lambda$ and write $\hat{\beta}_\lambda$ as $\hat{\beta}$. Here $g_{\lambda}:[0,\eta ] \to \mathbb{R},\, 0<\lambda \leq \eta$, is the regularization family satisfying the following conditions:
\begin{itemize}
    \item There exists a constant $A>0$ such that
    \begin{equation}
    \label{reg1}
        \sup_{0<\sigma \leq \eta} |\sigma g_{\lambda}(\sigma )| \leq A.
    \end{equation}
    \item There exists a constant $B>0$ such that
    \begin{equation}
    \label{reg2}
        \sup_{0<\sigma \leq \eta} |g_{\lambda}(\sigma )| \leq \frac{B}{\lambda}.
    \end{equation}
    \item There exists a constant $D >0 $ such that
    \begin{equation}
        \label{reg3}
        \sup_{0<\sigma \leq \eta} |1- g_{\lambda}(\sigma )\sigma| = \sup_{0<\sigma \leq \eta} |r_{\lambda}(\sigma)| \leq D.
    \end{equation}
    \item The maximal $p$ such that 
    \begin{equation}
    \label{qualification}
    \sup_{0<\sigma \leq \eta} |1- g_{\lambda}(\sigma )\sigma|\sigma^{p} = \sup_{0<\sigma \leq \eta}|r_{\lambda}(\sigma)| \sigma^{p} \leq \omega_{p} \lambda^{p},
    \end{equation}
    is called the qualification of the regularization family $g_{\lambda}$, where the constant $\omega_{p}$ does not depend on $\lambda$.
\end{itemize}
Examples of regularization families include Tikhonov ($g_\lambda(\sigma)=(\sigma+\lambda)^{-1}$), spectral cut-off ($g_\lambda(\sigma)=\sigma^{-1}\mathds{1}_{\sigma\ge \lambda}$), Showalter ($g_\lambda(\sigma)=\sigma^{-1}(1-e^{-\sigma/\lambda})\mathds{1}_{\sigma\ne 0}+\lambda^{-1}\mathds{1}_{\sigma=0}$), and Landweber iteration ($g_t(\sigma)=\sum^{t-1}_{i=1}(1-\sigma)^i$ where $\lambda$ is identified as $t^{-1}$, $t\in\mathbb{N}$), with qualification $1$ for Tikhonov and $\infty$ for the rest, where $\mathds{1}_{\Omega}(x)=1$ if $x\in \Omega$, and $0$, otherwise. We refer the reader to \cite{sergei2013,englmartinbook} for more details about the regularization method.

\section{Main Results}
\label{main results}
\noindent
In this section, we present the convergence rates of the general regularization method in the FLR model for both estimation and prediction errors. First, we begin with a list of assumptions required for the analysis.
\begin{assumption}\label{as:1}
  There exists a $h \in L^2(S)$ such that $\beta^* = T^{\alpha}h$, where $\alpha>0$, i.e., $\beta^*\in \mathcal{R}(T^\alpha)$.
\end{assumption}
\begin{assumption}\label{as:2}
    There exists a $h \in L^2(S)$ such that $\beta^* = T^{\frac{1}{2}}(T^{\frac{1}{2}}CT^{\frac{1}{2}})^s h$, where $s>0$, i.e., $\beta^*\in\mathcal{R}(T^{\frac{1}{2}}(T^{\frac{1}{2}}CT^{\frac{1}{2}})^s).$
\end{assumption}
Assumptions \ref{as:1} and \ref{as:2}, referred to as the source condition, impose the smoothness condition on $\beta^*$ in commutative (when $T$ and $C$ commute) and non-commutative (when $T$ and $C$ do not commute) settings, respectively. Since $\mathcal{R}(T^{\frac{1}{2}}) = \mathcal{H}\subset \mathcal{R}(T^\alpha),\,0<\alpha<\frac{1}{2}$ \cite{cuckerzhou2007learningtheory}, it follows that the latter imposes stricter conditions than the former. The weaker assumption considered for the commutative case allows us to explore the possibility of $\beta^*\in L^2(S)\backslash\mathcal{H}$. This assumption has been explored in \cite{balasubramanian2024functional} to derive sharper bounds in the commutative setting, which we improve in this paper by obtaining optimal convergence rates. Assumption~\ref{as:2}, which is more stringent than Assumption~\ref{as:1} has been employed by   \cite{ARKHSFORFLR,tonyyuan2012minimax} to derive optimal convergence rates when $s=0$ and by \cite{ZhangFaster2020, balasubramanian2024functional} to obtain convergence rates for $0 \leq s \leq \frac{1}{2}$ with the Tikhonov regularization. In this paper, we derive optimal convergence rates using the general regularization method under Assumption~\ref{as:2} for all positive values of $s$. We refer the reader to \cite[Section 3]{balasubramanian2024functional} for more insights into Assumption~\ref{as:2}.
We would like to highlight that while both the commutative and non-commutative cases can be analyzed in a unified manner under Assumption~\ref{as:2}, \cite{balasubramanian2024functional} demonstrated that sharper bounds can be obtained for the commutative case even under a weaker assumption (i.e., Assumption~\ref{as:1}). Hence, we analyze our estimator under two separate settings of commutative and non-commutative, using the appropriate assumptions.

\begin{assumption}\label{as:3}
    For some $t, c >1$,
\begin{equation*}
    i^{-t}\lesssim \mu_{i} \lesssim i^{-t} \quad \text{ and } \quad i^{-c}\lesssim \xi_{i} \lesssim i^{-c} \quad \forall  i \in \mathbb{N},
\end{equation*}
where $(\mu_{i}, \phi_{i})_{i \in \mathbb{N}}, (\xi_{i}, \phi_{i})_{i \in \mathbb{N}}$ are simple eigenvalue-eigenfunction pairs of operators $T$ and $C$, respectively.
\end{assumption}
\begin{assumption}\label{as:4}
     For some $b>1$,
\begin{equation*}
    i^{-b}\lesssim \tau_{i} \lesssim i^{-b} \quad \forall  i \in \mathbb{N},
\end{equation*}
where $(\tau_{i},e_{i})_{i\in \mathbb{N}}$ is the eigenvalue-eigenfunction pair of operator $T^{\frac{1}{2}}CT^{\frac{1}{2}}$.
\end{assumption}

Assumptions \ref{as:3} and \ref{as:4} capture the decay rate of eigenvalues of operators, $T$, $C$ and $T^{\frac{1}{2}}CT^{\frac{1}{2}}$, which are employed to obtain convergence rates in the commutative and non-commutative cases, respectively. Note that when $T$ and $C$ commute, then $b=t+c$.

\begin{assumption}\label{as:5}
    There exists a constant $d_1>0$ such that
\begin{equation*}
    \mathbb{E}[\langle X, f\rangle_{L^2}^4] \leq d_1[\mathbb{E}\langle X, f\rangle_{L^2}^2]^2,\,\,\, \forall\,f \in L^2(S).
\end{equation*}
\end{assumption}
In Assumption \ref{as:5}, it is posited that the fourth moment is constrained not to exceed the square of the second moment. An instance of data that adheres to this inequality is the one that follows a Gaussian distribution.
\subsection{Estimation error} 
We provide error bounds for estimating the slope function, $\beta^*$ under commutative and non-commutative
settings in Theorems \ref{comm.estimation} (proved in Sections~\ref{pcomm-est1} and \ref{pcomm-est2}) and \ref{NCestimation} (proved in Section~\ref{pcomm-nc}) under the source conditions mentioned in Assumptions~\ref{as:1} and \ref{as:2}, respectively.
 
\begin{theorem}[Estimation error--Commutative]
\label{comm.estimation}
Suppose $T$ and $C$ commute, and Assumptions~\ref{as:1}, \ref{as:3} and \ref{as:5} hold.
Let $\nu$ be the qualification of the regularization family and $\nu \geq \frac{1}{2}$. 
Then, for 
$ \lambda = 
    n^{-\frac{t+c}{1+c+2tr}}$, we have
\begin{enumerate}[label=(\alph*)]
    \item\label{a} $\|\beta^*-\hat{\beta}\|_{L^2} \lesssim_{p}
         n^{-\frac{r t}{1+c+2tr}},  \quad \text{ where } \quad  r = \min \{\alpha, \nu+ \frac{c}{t}(\nu -\frac{1}{2})\};$ 
    \item\label{b} $\|\beta^*-\hat{\beta}\|_{\mathcal{H}} \lesssim_{p}
        n^{-\frac{t(r-\frac{1}{2})}{1+c+2tr}}, \quad \mbox{ where} \quad r = \min \{\alpha, \nu+ \frac{c}{t}\nu +\frac{1}{2}\} \text{ for } \alpha \geq \frac{1}{2}$. 
\end{enumerate}
\end{theorem}
\noindent
\begin{remark}
(i) The assumption $\beta^* \in \mathcal{R}(T^{\alpha})$ enforces a smoothness condition that influences the properties of the functional space to which the slope function belongs. For $\alpha \geq\frac{1}{2}$, $\mathcal{R}(T^{\alpha})\subset\mathcal{H}$, while for $0 <\alpha <\frac{1}{2}$, $\mathcal{R}(T^{\alpha}) \subset L^2(S)\setminus \mathcal{H}$. \vspace{1.5mm}

(ii) The results on $L^2$ estimation error in the literature mostly deal with $\alpha=\frac{1}{2}$ \cite{ARKHSFORFLR} and $0 < \alpha \leq \frac{1}{2}$ \cite{balasubramanian2024functional}. When $\alpha =\frac{1}{2}$, our $L^2$-error rates align with the minimax rate established by \cite{ARKHSFORFLR} (also see Theorem~\ref{lowercomm}), while for $0 < \alpha < \frac{1}{2}$, our rates improve those provided in \cite{balasubramanian2024functional} 
from $n^{-\frac{\alpha t}{1+c+2t(1-\alpha)}}$ to $n^{-\frac{\alpha t}{1+c+2t\alpha}}$. In fact, in Theorem~\ref{lowercomm}, we show this improved rate to be minimax optimal. Moreover, the results presented in Theorem~\ref{comm.estimation} cater to smoothness index $\alpha>\frac{1}{2}$, which means faster convergence rates are obtained for smoother slope function but with the rate saturating at a level that depends on the qualification of the regularization method. \vspace{1.5mm}

(iii) As can be seen from Theorem \ref{comm.estimation}, we can allow $\alpha$ to go beyond the qualification $\nu$, and still the algorithm can achieve faster convergence rates. In particular, while using Tikhonov regularization for the commutative setting, we can get faster convergence rates for our algorithm even if $\alpha >1$. This happens since the regularization is applied to the operator $\Lambda$ while the source condition on $\beta^*$ is enforced through the operator $T$.
\end{remark}

\noindent
The following result (proved in Section~\ref{pcomm-est2}) presents the estimation error in the RKHS norm when $T$ and $C$ do not commute.
\begin{theorem}[Estimation error--Non-commutative]
\label{NCestimation}
Suppose Assumptions~\ref{as:2}, \ref{as:4} and \ref{as:5} hold.
Let $\nu$ be the qualification of the regularization family and $\nu \geq \frac{1}{2}$.
Then, for $ \lambda = n^{-\frac{b}{1+b+2rb}}$, we have
\begin{equation*}
    \|\beta^*-\hat{\beta}\|_{\mathcal{H}} \lesssim_p n^{-\frac{br}{1+b+2rb}}, \quad \text{ where } \quad r = \min \{s, \nu\}.
\end{equation*}
\end{theorem}
\noindent
The rate in Theorem~\ref{NCestimation} not only matches with the existing rates of \cite{balasubramanian2024functional,ZhangFaster2020} when $0< s \leq 1$, but also addresses the case of smoothness index $s$ beyond 1 (but with the rate saturating at $s=\nu$, the qualification of the regularization family). Moreover, in Theorem~\ref{lowerNC}, we show the bound of Theorem~\ref{NCestimation} to be minimax optimal.

\subsection{Prediction error} 
We present the prediction error bounds in commutative and non-commutative settings in Theorems~\ref{predictionbasic} (proved in Section~\ref{pcomm-pred1}) and \ref{predictionNC} (proved in Section~\ref{pcomm-pred2}), respectively.

\begin{theorem}[Prediction error--Commutative]
    \label{predictionbasic}
    Suppose $T$ and $C$ commute, and Assumptions~\ref{as:1}, \ref{as:3} and \ref{as:5} hold. Let $\nu$ be the qualification of the regularization family and $\nu \geq \frac{1}{2}$.
Then, for $ \lambda = n^{-\frac{t+c}{1+c+2tr}}$, we have
\begin{equation*}
    \|C^{\frac{1}{2}}(\beta^*-\hat{\beta})\|_{L^2}^2 \lesssim_p
       n^{-\frac{2r t+c}{1+c+2tr}} , \quad \mbox{ where } \quad r = \min\left\{\alpha, \nu + \frac{1}{2} + \nu \frac{c}{t} \right\}.
\end{equation*}
\end{theorem}
\noindent
The above result matches the results of \cite{balasubramanian2024functional,ARKHSFORFLR}, wherein \cite{ARKHSFORFLR} considers only $\alpha=\frac{1}{2}$, while \cite{balasubramanian2024functional} extended the result of \cite{ARKHSFORFLR} under Assumption \ref{as:1} for $0 < \alpha \leq \frac{1}{2}$. Moreover, Theorem \ref{predictionbasic} obtains a rate for $\alpha$ beyond the above-mentioned ranges, and we show this rate to be minimax optimal in Theorem~\ref{lowercomm}. 
\begin{remark}\label{rem:compare}
In the learning theory approach for FLR \cite{lin2017multipass, lin2020optimalspectral, lin2020distributed_stochastic}, the authors consider the source condition with respect to the covariance operator, specifically $\beta^* \in \mathcal{R}(C^{\gamma})$. Under infinite qualification, the convergence rates achieved for the FLR model in these articles are $\|C^{\frac{1}{2}}(\beta^*-\check{\beta})\|_{L^2}^2 \lesssim_{p} n^{-\frac{2 \gamma c + c}{1 + c + 2 \gamma c}}$ and $\|\beta^* - \check{\beta}\|_{L^2} \lesssim_{p} n^{-\frac{\gamma c}{1+c + 2 \gamma c}}$. We can observe that these rates can be produced by Theorems \ref{comm.estimation} and \ref{predictionbasic}, i.e., under the commutativity of $T$ and $C$, if $\gamma c=\alpha t$. Though both of these methods seem equivalent, we would like to highlight that the boundedness assumption over $\|X\|_{L^2}$ considered in the learning theory approach does not incorporate all stochastic processes while our approach works for a more general class of stochastic processes (See Remark~\ref{comparetobounded}).
\end{remark}

The following result, which considers the prediction error when $T$ and $C$ do not commute, is an extension of the results of \cite{balasubramanian2024functional,tonyyuan2012minimax,HTONGFLR,ZhangFaster2020} in terms of the smoothness condition. \cite{balasubramanian2024functional,tonyyuan2012minimax,HTONGFLR} have considered the  prediction error for non-commutative case with $\beta^*\in \mathcal{H}$. In \cite{ZhangFaster2020}, authors have extended these results with Assumption \ref{as:3} but for a restricted range of $0<s\leq \frac{1}{2}$. We extend all these results for all positive ranges of $s$.
\begin{theorem}[Prediction error--Non-commutative]
\label{predictionNC}
Suppose Assumptions~\ref{as:2}, \ref{as:4} and \ref{as:5} hold. 
Let $\nu$ be the qualification of the regularization family and $\nu \geq \frac{1}{2}$.
Then, for $ \lambda = n^{-\frac{b}{1+b+2rb}}$, we have
\begin{equation*}
    \|C^{\frac{1}{2}}(\beta^*-\hat{\beta})\|_{L^2}^2 \lesssim_p n^{-\frac{b(2r+1)}{1+b+2rb}}, \quad \text{where} \quad  r = \min \left\{s, \nu - \frac{1}{2} \right\}.
\end{equation*}
\end{theorem}

\begin{remark}\label{rem:compare1}
Based on Remark~\ref{rem:compare}, it is clear that the learning theory approach and the commutative setting of our approach are equivalent (except for the assumptions). Therefore, any comparison of the learning theory approach to the non-commutative setting can be made by comparing the commutative and non-commutative settings. As it can be seen the source conditions in these settings are completely different and they match as $s\rightarrow 0$ and $\alpha=\frac{1}{2}$, in which case the rates match. Otherwise, these settings are difficult to compare because of the interaction between $T$ and $T^{1/2}CT^{1/2}$ in the non-commutative setting.
\end{remark}

\begin{remark}\label{comparetobounded}
Note that all the results so far have been obtained under the fourth moment assumption as captured in Assumption \ref{as:5}. The key ingredients in the proofs of these results that exploit Assumption \ref{as:5} are
Lemmas~\ref{cmdifference} and \ref{ncmdifference}. In the learning theory literature, a bounded assumption $\sup_{\omega \in \Omega}\|X(\cdot, \omega)\|_{L^2} < \infty$ is used in place of Assumption \ref{as:5} to obtain the optimal rates. In Appendix~\ref{equivalentbounds}, we derive equivalent versions of Lemma \ref{cmdifference} and  \ref{ncmdifference} (denoted as Lemmas \ref{cmdifnew} and \ref{ncmdiffnew}) under the boundedness assumption, using which it can be shown that the same rates as we obtained under Assumption \ref{as:5} hold even under the boundedness assumption. 
The comparison between the two assumptions highlights their different applicability to stochastic processes. Assumption~\ref{as:5} is well-suited for non-degenerate Gaussian processes due to their well-controlled moments. On the other hand, the assumption of boundedness in \( \|X\|^2 \) will hold for bounded processes but will fail to hold for a non-degenerate Gaussian process.
\end{remark}

\begin{remark}\label{rem}
For an index function $g: \mathbb{R} \mapsto \mathbb{R}$, the single index model is defined as
 \begin{equation*}
        Y = g\left(\int_{S} X(t) \beta^*(t) dt\right) + \epsilon.
    \end{equation*}
A detailed study of estimating $\beta^*$ in a single-index model has been carried out in \cite{muller2005generalized,james2002generalized,chen2011single,shang2015single,wang2011single}. In \cite{balasubramanian2024functional}, the authors showed that if $X$ is sampled from a Gaussian process, then under some regularity conditions on $g$, $\beta^*$ in the single-index model satisfies the operator equation of the FLR model, i.e., $\mathbb{E}[XY]=C\beta^*$. This means the results obtained in this paper will address the question of estimating $\beta^*$ in a single-index model if $X$ is sampled from a Gaussian process, thereby sharpening and generalizing the results of \cite{balasubramanian2024functional}.
\end{remark}

In the following, we provide examples of kernels and covariance functions that commute, while the non-commutative setup was already discussed in \cite[Section 4]{balasubramanian2024functional}. Suppose $K(x,y)=\sum_i \lambda_i \phi_i(x)\phi_i(y)$ and $C(x,y)=\sum_i \mu_i \psi_i(x)\psi_i(y)$ where $(\phi_i)_i$ and $(\psi_j)_j$ are orthonormal systems in $L^2(S).$ Then $(\lambda_i,\phi_i)_i$ and $(\mu_i,\psi_i)_i$ are the eigensystems of $T$ and $C$, respectively. It is easy to verify that $T$ and $C$ commute if $\phi_i=\psi_i$ for all $i$. For example, choosing $\phi_i(s)=\psi_i(s)=\sin\left(\left(i-\frac{1}{2}\right)\pi s\right), s\in [0,1]$, we have
\begin{align*}C(s,t) &= \sum_{n \geq 1} \frac{2}{(n-\frac{1}{2})^2 \pi^2} \sin\left(\left(n-\frac{1}{2}\right)\pi s\right) \sin\left(\left(n-\frac{1}{2}\right)\pi t\right) \\
&= E_{1}\left(\frac{s+t}{2}\right) - E_{1}\left(\frac{|s-t|}{2}\right) = s\wedge t,\end{align*}
which is the Brownian motion covariance function 
and 
\begin{align*}K(s,t) &= \sum_{n \geq 1} \frac{2}{(n-\frac{1}{2})^4 \pi^4} \sin\left(\left(n-\frac{1}{2}\right)\pi s\right) \sin\left(\left(n-\frac{1}{2}\right)\pi t\right) \\ &= \frac{2}{3}\left(E_{3}\left(\frac{|s-t|}{2}\right) - E_{3}\left(\frac{s+t}{2}\right)\right)
 = \frac{1}{12}|s-t|^3- \frac{1}{12}(s+t)^3 + st,\end{align*}
where 
$$E_{2n-1}(x) = \frac{(-1)^n 4 (2n-1)!}{\pi^{2n}} \sum_{k = 0}^{\infty} \frac{\cos((2k+1)\pi x)}{(2k+1)^{2n}}, ~ \forall n \in \mathbb{N} \text{ and } x\in[0,1],$$
is the Euler polynomial of degree $2n-1$ (see \cite[Section 1.2.1] {euler1996representation}).

\cite[Section 4]{balasubramanian2024functional} discusses examples of kernels and covariance functions, where $\phi_i$'s are considered as Fourier basis and $\psi_i$'s are considered as Haar basis on $[0,1]$, resulting in a non-commutative system.

\section{Lower Bounds}
\label{lower bounds}
In this section, we establish lower bounds for each of the cases discussed in the previous section and show them to match the upper bounds, thereby confirming the optimality of the proposed regularized estimator of $\beta^*$. The idea of the proof is inspired by the work of Cai and Yuan \cite{tonyyuan2012minimax}.   
To derive the lower bounds, we need to discuss the divergence between two probability measures $P_1$ and $P_2$ defined on a measurable space $(\mathcal{X}, \mathcal{A})$.  In our analysis, we use \textbf{Kullback-Leibler divergence} which is defined as : 
\begin{equation*}
    \mathcal{K}(P_1, P_2) :=
    \begin{cases}
         \displaystyle \int\limits_\mathcal{X} \log\left(\frac{dP_1}{dP_2}\right) dP_1 & \mbox{ if } P_1 \ll P_2;\\
        +\infty & \mbox{  } \text{ otherwise},
        \end{cases}
\end{equation*}
where $P_1 \ll P_2$ denotes $P_1$ is absolutely continuous with respect to $P_2$.

The key concept to derive the lower bounds follows from \cite[Chapter 2]{tsyback2009lb}, where the major task is to find $N+1$ elements $\{\theta_{0}', \ldots, \theta_{N}'\}$ from a hypothesis space such that the distance between any two of them is at least $2r$ for some constant $r>0$. Furthermore, associated with  each $\theta_{j}'$, we have a distribution $P_{\theta_{j}'}$, and the mean of the Kullback-Leibler divergence between  distributions $P_{\theta_{j}'}$ and $P_{\theta_{0}'}$ should grow at most logarithmically. Consequently, we get that for large $N$, with high probability, the distance between  $\theta_{j}'$ and
any estimator $\hat{\theta}$ will be at least $r$. For easy reference, we state the key theorem \cite[Theorem 2.5]{tsyback2009lb} below.

Let  $\Theta'$ be a non-empty set. Suppose $d : \Theta' \times \Theta' \to [0, \infty)$ be a semi-distance function and $\mathcal{P} = \{P_{\theta} : \theta \in \Theta'\}$ be the set of probability measures on $(\mathcal{X}, \mathcal{A})$  indexed by $\Theta'$.
\begin{theorem}\cite[Theorem 2.5]{tsyback2009lb}
\label{tsyback theorem}
    Assume that $N >2$ and suppose that the hypothesis space $\Theta'$ contains elements $\theta_{0}', \theta_{1}', \ldots, \theta_{N}'$ such that:
    \begin{enumerate}
        \item $d(\theta_{j}', \theta_{k}' ) \geq 2r >0 \quad \forall~ 0\leq j< k \leq N$;
        \item $P_{\theta_{j}'} \ll P_{\theta_{0}'} \quad \forall j =1,2,\ldots,N,$ and
        $$\frac{1}{N} \sum_{j=1}^{N}\mathcal{K}(P_{\theta_{j}'},P_{\theta_{0}'} ) \leq u \log N$$
        for some $0 < u <\frac{1}{8}$. Then
    \end{enumerate}
    $$\inf_{\hat{\theta}}\sup_{\theta\in \Theta'} P_{\theta}(d(\theta,\hat{\theta})\geq r ) \geq \frac{\sqrt{N}}{1+\sqrt{N}}\left(1-2u-\sqrt{\frac{2u}{\log N}}\right),$$
    where the infimum is taken over all estimators  $\hat{\theta}$ of $\theta\in\Theta'$.
\end{theorem}

The next lemma, which ensures the existence of $N+1$ elements in $\{0,1\}^{M}$ with Hamming distance at least $\frac{M}{8}$, will be helpful to construct $N+1$ elements in a hypothesis space with the desired distance.
\begin{lemma}[Varshamov-Gilbert bound \cite{tsyback2009lb}] 
\label{VGbound}
Let $M \geq 8$. Then there exists a subset $\Theta = \{\theta^{(0)},\ldots,\theta^{(N)}\} \subset \{0,1\}^{M}$ such that $\theta^{(0)}=(0,\cdots,0)$,
\begin{equation*}
    H(\theta,\theta^{'}) > \frac{M}{8}, \quad \forall ~~ \theta \neq \theta^{'} \in \Theta,
\end{equation*}
where $\displaystyle H(\theta, \theta^{'}) = \sum_{i=1}^{M}(\theta_{i}-\theta^{'}_{i})^2 $ is the Hamming distance and $N \geq 2^{\frac{M}{8}}$.
\end{lemma}

\noindent
The following result (proved in Section~\ref{pcomm-low}) provides lower bounds associated with Theorems~\ref{comm.estimation} and \ref{predictionbasic}, i.e., estimation and prediction errors, respectively in the commutative setting. 

\begin{theorem}[Lower bound--Commutative]
\label{lowercomm}
    Suppose Assumptions \ref{as:1} and \ref{as:3} hold. Then for $\alpha \geq \frac{1}{2}$, we have
    \begin{equation*}
        \lim_{a \to 0} \lim_{n \to \infty} \inf_{\hat{\beta}} \sup_{\beta^* \in \mathcal{R}(T^{\alpha})} \mathbb{P}\left\{\|\hat{\beta}-\beta^*\|_{\mathcal{H}} \geq a n^{-\frac{t(\alpha-\frac{1}{2})}{1+c+2 \alpha t}}\right\} =1.
    \end{equation*}
Further for $\alpha >0$, we have
\begin{equation*}
        \lim_{a \to 0} \lim_{n \to \infty} \inf_{\hat{\beta}} \sup_{\beta^* \in \mathcal{R}(T^{\alpha})} \mathbb{P}\left\{\|\hat{\beta}-\beta^*\|_{L^2} \geq a n^{-\frac{\alpha t}{1+c+2 \alpha t}}\right\} =1
    \end{equation*}
and
\begin{equation*}
    \lim_{a \to 0} \lim_{n \to \infty} \inf_{\hat{\beta}} \sup_{\beta^* \in \mathcal{R}(T^{\alpha})} \mathbb{P}\left\{\mathbb{E}\langle\hat{\beta}-\beta^*, X\rangle_{L^2}^2 \geq a n^{-\frac{2\alpha t +c}{1+c+2 \alpha t}}\right\} =1,
\end{equation*}
where the infimum is taken over all possible estimators.
\end{theorem}

\noindent
The following result (proved in Section~\ref{pcomm-lower2}) provides lower bounds on the estimation and prediction errors in the non-commutative setting. Combined with Theorems \ref{NCestimation} and \ref{predictionNC}, Theorem \ref{lowerNC} ensures that the convergence rates achieved in the non-commutative setting are optimal.
\begin{theorem}[Lower bound--Non-commutative]
\label{lowerNC}
Suppose Assumptions \ref{as:2} and \ref{as:4} hold. Then for any $s>0$, we have
\begin{equation*}
        \lim_{a \to 0} \lim_{n \to \infty} \inf_{\hat{\beta}} \sup_{\beta^* \in \mathcal{R}(T^{\frac{1}{2}}\Lambda^s)} \mathbb{P}\left\{\|\hat{\beta}-\beta^*\|_{\mathcal{H}} \geq a n^{-\frac{sb}{1+b+2 sb}}\right\} =1
    \end{equation*}
and 
\begin{equation*}
    \lim_{a \to 0} \lim_{n \to \infty} \inf_{\hat{\beta}} \sup_{\beta^* \in \mathcal{R}(T^{\frac{1}{2}}\Lambda^s)} \mathbb{P}\left\{\mathbb{E}\langle\hat{\beta}-\beta^*, X\rangle_{L^2}^2 \geq a n^{-\frac{b(2s+1)}{1+b+2 sb}}\right\} =1,
\end{equation*}
where the infimum is taken over all possible estimators.
\end{theorem}

\textcolor{black}{
\section{Comparison to the Learning Theory Approach} \label{Sec:similar}
In this section, we provide a detailed discussion of the similarities and differences between $\mathbf{(I)}$ the learning theory approach for FLR, $\mathbf{(II)}$ the RKHS-based approach for FLR (assuming $\beta^*\in \mathcal{H})$, and $\mathbf{(III)}$ nonparametric regression in RKHS. }

\textcolor{black}{Consider the setting of $\mathbf{(II)}$ which satisfies the normal equation $J^*\mathbb{E}[X\otimes_{L^2}X]J\beta^*=J^*\mathbb{E}[XY]$ where $\beta^*$ is the minimizer of the following problem:
\begin{equation}\label{Eq:II}
\beta^*=\arg\inf_{\beta\in\mathcal{H}} \mathbb{E}(Y-\langle X,\beta\rangle_{L^2(S)})^2.
\end{equation}
Note that replacing the minimization over $\mathcal{H}$ by $L^2(S)$ precisely yields the normal equation $\mathbb{E}[X\otimes_{L^2}X]\beta^*=\mathbb{E}[XY]$ which corresponds to that of $\mathbf{(I)}$. Since $\beta\in\mathcal{H}$ is equivalent to $\beta\in \mathcal{R}(T^{1/2})$, there exists $\zeta\in L^2(S)$ such that $T^{1/2}\zeta=J\beta$ so that \eqref{Eq:II} can be alternately written as
\begin{equation*}
\zeta^*=\arg\inf_{\zeta\in L^2(S)} \mathbb{E}(Y-\langle X,T^{1/2}\zeta\rangle_{L^2(S)})^2=\arg\inf_{\zeta\in L^2(S)} \mathbb{E}(Y-\langle T^{1/2}X,\zeta\rangle_{L^2(S)})^2,
\end{equation*}
where $J\beta^*=T^{1/2}\zeta^*$. This is equivalent to solving an FLR problem of type $\mathbf{(I)}$ with $X$ being replaced by $T^{1/2}X$ so that the resulting normal equation is given by $\mathbb{E}[T^{1/2}X\otimes_{L^2}T^{1/2}X]\zeta^*=\mathbb{E}[T^{1/2}XY]$, i.e., $T^{1/2}CT^{1/2}\zeta^*=T^{1/2}\mathbb{E}[XY]$.} 

\textcolor{black}{Based on these normal equations, we can define the estimators of $\beta^*$ in the settings of $\mathbf{(I)}$ and $\mathbf{(II)}$ as 
$g_\lambda(\hat{C})\hat{\mathbb{E}}[XY]$ and $T^{1/2}g_\lambda(T^{1/2}\hat{C}T^{1/2})T^{1/2}\hat{\mathbb{E}}[XY]$, respectively where $\hat{C}$ is the empirical estimator of $C$ and $\hat{\mathbb{E}}$ is the empirical expectation based on $n$ samples. Based on the general convergence theory developed for $\mathbf{(III)}$---of which $\mathbf{(I)}$ is a special case, and $\mathbf{(II)}$ being formally similar to $\mathbf{(I)}$---it follows that convergence rates for these estimators depend on the source conditions, $\beta^*\in\mathcal{R}(C^\gamma),\,\gamma>0$ and $\zeta^*\in \mathcal{R}((T^{1/2}CT^{1/2})^s),\,s>0$, i.e., $\beta^*\in \mathcal{R}(T^{1/2}(T^{1/2}CT^{1/2})^s),\,s>0$, respectively, and the eigenvalue-decay rate of $C$ and $T^{1/2}CT^{1/2}$, respectively. This observation about $\mathbf{(II)}$ matches with Assumptions~\ref{as:2} and \ref{as:4} using which we derived convergence rates of our estimator in the non-commutative setting.} 

\textcolor{black}{Given the discussion so far, it is reasonable to expect that the convergence rate analysis of $\mathbf{(I)}$ and $\mathbf{(II)}$ would follow the same steps, i.e., the proof of non-commutative case would match that of the learning theory approach, which in turn matches with that of $\mathbf{(III)}$ since as we highlighted in Section~\ref{sec:intro} that $\mathbf{(I)}$ is equivalent to non-parametric regression in RKHS with a linear kernel on $L^2(S)$, which is a special case of $\mathbf{(III)}$. However, there are some key differences that we highlight below.
\begin{itemize}
\item Though the proof strategy of the non-commutative case can be similar to that of $\mathbf{(I)}$ and $\mathbf{(III)}$, the corresponding rates are not comparable as highlighted in Remark~\ref{rem:compare1} because of the interaction between $T$ and $C$, which does not occur in $\mathbf{(I)}$.\vspace{1mm}
\item While the same proof strategy of the non-commutative case can be used in the commutative setting (i.e., $T$ and $C$ commute), we would like to highlight that the resulting rates will not be minimax optimal. Our modified proof strategy deeply exploits the commutative assumption to achieve optimal rates of convergence. But interestingly, we show in Remark~\ref{rem:compare} that the convergence rates of $\mathbf{(I)}$ and that of the commutative setting of $\mathbf{(II)}$ are equivalent. In addition, employing the proof strategy of the non-commutative case to the commutative case can only handle the source condition $\beta^*\in \mathcal{R}(T^{1/2})$, i.e., with $s=0$ while the commutative analysis deals with general source condition of $\beta^*\in\mathcal{R}(T^{\alpha}),\,\alpha>0$.\vspace{1mm}
\item Another key difference between our approach and $\mathbf{(I)}$ is Assumption~\ref{as:5}, i.e., the assumption about the stochastic process generating the input data, which is discussed in detail in Remark~\ref{comparetobounded}.
As mentioned in Remark~\ref{comparetobounded}, $\mathbf{(I)}$ assumes boundedness of the stochastic process, an assumption that is not valid for the non-degenerate Gaussian processes that are popularly considered in the statistical community. On the other hand, our analysis covers both cases when the process is (i) bounded as mentioned in Remark \ref{comparetobounded}, or (ii) having controlled moments such as non-degenerate Gaussian process via Assumption~\ref{as:5}.
 \vspace{1mm}
\item \textcolor{black}{Since $\mathbf{(I)}$ is a special case of $\mathbf{(III)}$, the lower bounds of $\mathbf{(I)}$ were provided in the literature by directly specializing the corresponding lower bounds for $\mathbf{(III)}$. In contrast, the lower bounds of $\mathbf{(II)}$ do not follow directly from $\mathbf{(III)}$ because of the interaction between $T$ and $C$. Hence, we highlight that the lower bound results provided in Theorems~\ref{lowercomm} and \ref{lowerNC} are new.
}
\end{itemize}
}

\section{Proofs}
\label{proofs}
The proofs of the main results of Sections \ref{main results} and \ref{lower bounds} are provided in this section.

\subsection{Proof of part \ref{a} in Theorem \ref{comm.estimation}} \label{pcomm-est1}
The proof involves decomposing $\|\hat{\beta}-\beta^*\|_{L^2}$ into several terms and bounding each of the terms separately. We start by considering the error term
	\begin{equation*}
        \begin{split}
	        \| \hat{\beta} - \beta^* \|_{L^2}  = & \| J \hat{\beta} -\beta^* \|_{L^2} 
            = \|Jg_{\lambda}(J^*\hat{C}_{n}J)J^*\hat{R}-\beta^*\|_{L^2} \\
	        = & \| T^{\frac{1}{2}}g_{\lambda}(\hat{\Lambda}_{n})T^{\frac{1}{2}}\hat{R}  -\beta^* \|_{L^2} \quad (\text{by Lemma~\ref{J T}}) \\
          \leq & \underbrace{\| T^{\frac{1}{2}}g_{\lambda}(\hat{\Lambda}_{n})(T^{\frac{1}{2}}\hat{R} - T^{\frac{1}{2}}\hat{C}_{n}\beta^*)\|_{L^2}}_{\textit{Term-1}} + \underbrace{\|T^{\frac{1}{2}}g_{\lambda}(\hat{\Lambda}_{n})T^{\frac{1}{2}}\hat{C}_{n}\beta^*-\beta^*\|_{L^2}}_{\textit{Term-2}}. 
         \end{split}
	\end{equation*}
 \textit{Bounding Term-1:}
 \begin{equation*}
 \begin{split}
     & \| T^{\frac{1}{2}}g_{\lambda}(\hat{\Lambda}_{n})(T^{\frac{1}{2}}\hat{R} - T^{\frac{1}{2}}\hat{C}_{n}\beta^*)\|_{L^2} \\
      \leq & \| T^{\frac{1}{2}}g_{\lambda}(\hat{\Lambda}_{n})(\hat{\Lambda}_{n}+\lambda I)^{\frac{1}{2}}\|_{L^2 \to L^2}\|(\hat{\Lambda}_{n}+\lambda I)^{-\frac{1}{2}}(\Lambda+\lambda I)^{\frac{1}{2}}\|_{L^2 \to L^2}\\
     & \times\|(\Lambda+\lambda I)^{-\frac{1}{2}}(T^{\frac{1}{2}}\hat{R} -T^{\frac{1}{2}}\hat{C}_{n}\beta^*)\|_{L^2}\\
     \lesssim & \| T^{\frac{1}{2}}(\Lambda+\lambda I)^{-\frac{1}{2}}\|_{L^2 \to L^2}\|(\Lambda+\lambda I)^{\frac{1}{2}}(\hat{\Lambda}_{n}+\lambda I)^{-\frac{1}{2}}\|_{L^2 \to L^2}\\
     & \times \|(\hat{\Lambda}_{n}+\lambda I)^{\frac{1}{2}}g_{\lambda}(\hat{\Lambda}_{n})(\hat{\Lambda}_{n}+\lambda I)^{\frac{1}{2}}\|_{L^2 \to L^2} \|(\Lambda+\lambda I)^{-\frac{1}{2}}(T^{\frac{1}{2}}\hat{R} -T^{\frac{1}{2}}\hat{C}_{n}\beta^*)\|_{L^2}\\
     \stackrel{(*)}{\lesssim_{p}} &  \| T^{\frac{1}{2}}(\Lambda+\lambda I)^{-\frac{1}{2}}\|_{L^2 \to L^2} \|(\Lambda+\lambda I)^{-\frac{1}{2}}(T^{\frac{1}{2}}\hat{R} -T^{\frac{1}{2}}\hat{C}_{n}\beta^*)\|_{L^2},
 \end{split}
 \end{equation*}
 where we used Lemma \ref{Covariance estimation} and the regularization properties $(\ref{reg1}), (\ref{reg2})$ in $(*)$. 
 Next, we apply Lemma \ref{Empirical bound} to conclude that
 \begin{equation}
     \begin{split}
           \| T^{\frac{1}{2}}g_{\lambda}(\hat{\Lambda}_{n})(T^{\frac{1}{2}}\hat{R} - T^{\frac{1}{2}}\hat{C}_{n}\beta^*)\|_{L^2}
          \lesssim_{p} & \sqrt{\frac{ \sigma^2 \mathcal{N}(\lambda)}{n}} \| T^{\frac{1}{2}}(\Lambda+\lambda I)^{-\frac{1}{2}}\|_{L^2 \to L^2} \nonumber\\
          \leq & \lambda^{-\frac{c}{2(t+c)}} \sqrt{\frac{ \sigma^2 \mathcal{N}(\lambda)}{n}},\nonumber
     \end{split}
 \end{equation}
where the last inequality follows from the fact that
 \begin{equation}
 \label{testimate}
 \begin{split}
     \| T^{\frac{1}{2}}(\Lambda+\lambda I)^{-\frac{1}{2}}\|_{L^2 \to L^2} & = \left(\sup_{i} \left|\frac{\mu_{i}}{\mu_{i} \xi_{i}+\lambda}\right|\right)^{\frac{1}{2}} \quad (\text{$T$ and $C$ commute})\\
     & \lesssim \left(\sup_{i} \left|\frac{i^{-t}}{i^{-(t+c)}+\lambda}\right|\right)^{\frac{1}{2}}\quad (\text{by Assumption }\ref{as:3}) \\
     & \lesssim \lambda^{\frac{t-(t+c)}{2(t+c)}} = \lambda^{-\frac{c}{2(t+c)}} \quad (\text{by Lemma~\ref{supbound}}).
 \end{split}
 \end{equation}
 \textit{Bounding Term-2:}
 As the bound on this term will depend on the source condition, the estimation has been considered under two cases over the range of $\alpha$.\\
\textbf{Case-1}: Suppose $\alpha\in(0,\frac{1}{2}]$. Consider
 \begin{equation*}
     \begin{split}
         & \|T^{\frac{1}{2}}g_{\lambda}(\hat{\Lambda}_{n})T^{\frac{1}{2}}\hat{C}_{n}\beta^*-\beta^*\|_{L^2}\\
         \leq & \| T^{\frac{1}{2}}(g_{\lambda}(\hat{\Lambda}_{n})-(\hat{\Lambda}_{n}+ \lambda I)^{-1})T^{\frac{1}{2}}\hat{C}_{n}\beta^*\|_{L^2} + \|T^{\frac{1}{2}}(\hat{\Lambda}_{n} + \lambda I)^{-1}T^{\frac{1}{2}}\hat{C}_{n}\beta^* -\beta^* \|_{L^2}\\
         \leq &  \underbrace{\|T^{\frac{1}{2}}((\hat{\Lambda}_{n}+ \lambda I)^{-1}T^{\frac{1}{2}}\hat{C}_{n}\beta^* - (\Lambda+ \lambda I)^{-1}T^{\frac{1}{2}}C\beta^*)\|_{L^2}}_{\textit{Term-2a}}\\
          & + \underbrace{\| T^{\frac{1}{2}}(g_{\lambda}(\hat{\Lambda}_{n})-(\hat{\Lambda}_{n}+ \lambda I)^{-1})T^{\frac{1}{2}}\hat{C}_{n}\beta^*\|_{L^2}}_{\textit{Term-2b}} + \underbrace{\|T^{\frac{1}{2}}(\Lambda + \lambda I)^{-1}T^{\frac{1}{2}}C\beta^* -\beta^*\|_{L^2}}_{\textit{Term-2c}}. 
     \end{split}
 \end{equation*}

\noindent
\textit{Bounding Term-2a:}
Define $\Tilde{\beta}:= (\Lambda + \lambda I)^{-1}T^{\frac{1}{2}}C\beta^* $ and $\Upsilon := \frac{t(1-2 \alpha)}{2(t+c)}$. Then
\begin{equation*}
    \begin{split}
        & \|T^{\frac{1}{2}}((\hat{\Lambda}_{n} + \lambda I)^{-1}T^{\frac{1}{2}}\hat{C}_{n}\beta^* -(\Lambda + \lambda I)^{-1}T^{\frac{1}{2}}C\beta^*)\|_{L^2} \\
        = & \|T^{\frac{1}{2}}(\hat{\Lambda}_{n} + \lambda I)^{-1}(T^{\frac{1}{2}}\hat{C}_{n}\beta^* - \hat{\Lambda}_{n}\Tilde{\beta}- \lambda \Tilde{\beta} )\|_{L^2} \\
        = & \|T^{\frac{1}{2}}(\hat{\Lambda}_{n} + \lambda I)^{-1}(T^{\frac{1}{2}} (C -\hat{C}_{n})(T^{\frac{1}{2}}\Tilde{\beta} -  \beta^*))\|_{L^2} \quad (\text{since } \lambda \tilde{\beta} = T^{\frac{1}{2}}C\beta^* - \Lambda \tilde{\beta} )\\
        \leq & \|T^{\frac{1}{2}}(\hat{\Lambda}_{n} + \lambda I)^{-\frac{1}{2}}\|_{L^2 \to L^2}\|(\hat{\Lambda}_{n} + \lambda I)^{-\frac{1}{2}}T^{\frac{1}{2}} (C -\hat{C}_{n})(T^{\frac{1}{2}}\Tilde{\beta} -  \beta^*)\|_{L^2}.
        \end{split}
        \end{equation*}
        From equation $(\ref{testimate})$ and the observation that $T^{\frac{1}{2}}\Tilde{\beta} - \beta^* = - \lambda (\Lambda + \lambda I)^{-1}\beta^*$, we have
        \begin{equation}
        \label{2b}
        \begin{split}
        & \|T^{\frac{1}{2}}((\hat{\Lambda}_{n} + \lambda I)^{-1}T^{\frac{1}{2}}\hat{C}_{n}\beta^* -(\Lambda + \lambda I)^{-1}T^{\frac{1}{2}}C\beta^*)\|_{L^2} \\
        \lesssim & \lambda \lambda^{-\frac{c}{2(t+c)}} \|(\hat{\Lambda}_{n} + \lambda I)^{-\frac{1}{2}}T^{\frac{1}{2}}(C -\hat{C}_{n})(\Lambda+ \lambda I)^{-1}\beta^*\|_{L^2}\\
        \lesssim & \sqrt{\lambda}\lambda^{-\frac{c}{2(t+c)}} \|(\Lambda + \lambda I)^{-\frac{1}{2}}T^{\frac{1}{2}}(C -\hat{C}_{n})(\Lambda+ \lambda I)^{-\frac{1}{2}}T^{\alpha}\|_{L^2 \to L^2}\|h\|_{L^2}\\
        \lesssim & \sqrt{\lambda}\lambda^{-\frac{c}{2(t+c)}} \|T^{\frac{1}{2}-\alpha}(\Lambda + \lambda I)^{-\Upsilon}T^{\alpha}(\Lambda + \lambda I)^{-(\frac{1}{2}-\Upsilon)}(C -\hat{C}_{n})(\Lambda+\lambda I)^{-\frac{1}{2}}T^{\alpha}\|_{L^2 \to L^2}\\
        \leq & \sqrt{\lambda}\lambda^{-\frac{c}{2(t+c)}} \lambda^{-\Upsilon} \|T^{\alpha}(\Lambda + \lambda I)^{-(\frac{1}{2}-\Upsilon)}(C-\hat{C}_{n})(\Lambda + \lambda I)^{-(\frac{1}{2}-\Upsilon)}T^{\alpha}\|_{L^2 \to L^2}.
    \end{split}
\end{equation}
By using $p= (\frac{1}{2}-\Upsilon)$ in Lemma \ref{cmdifference}, we get

\begin{equation*}
\label{empirical-real}
\|T^{\frac{1}{2}}(\hat{\Lambda}_{n} + \lambda I)^{-1}T^{\frac{1}{2}}\hat{C}_{n}\beta^* -T^{\frac{1}{2}}(\Lambda + \lambda I)^{-1}T^{\frac{1}{2}}C\beta^*\|_{L^2} \lesssim_{p}
         \frac{\lambda^{-\frac{1- \alpha t}{t+c}}}{\sqrt{n}}.
\end{equation*}

\noindent
\textit{Bound of Term-2b:} Consider
\begin{equation}
    \begin{split}
    \label{power3/2}
        & \| T^{\frac{1}{2}}(g_{\lambda}(\hat{\Lambda}_{n})-(\hat{\Lambda}_{n}+ \lambda I)^{-1})T^{\frac{1}{2}}\hat{C}_{n}\beta^*\|_{L^2}\\
        = & \| T^{\frac{1}{2}}(\hat{\Lambda}_{n}+ \lambda I)^{-1}((\hat{\Lambda}_{n}+ \lambda I)g_{\lambda}(\hat{\Lambda}_{n})-I)T^{\frac{1}{2}}\hat{C}_{n}\beta^*\|_{L^2}\\
        \leq & \| T^{\frac{1}{2}}(\hat{\Lambda}_{n}+ \lambda I)^{-1}((\hat{\Lambda}_{n}g_{\lambda}(\hat{\Lambda}_{n})-I) + \lambda g_{\lambda}(\hat{\Lambda}_{n}) )T^{\frac{1}{2}}\hat{C}_{n}\beta^*\|_{L^2} \\
        \leq & \| T^{\frac{1}{2}}(\Lambda+ \lambda I)^{-\frac{1}{2}}\|_{L^2 \to L^2} \|(\Lambda+ \lambda I)^{\frac{1}{2}}(\hat{\Lambda}_{n}+ \lambda I)^{-\frac{1}{2}}\|_{L^2 \to L^2} \|(\hat{\Lambda}_{n}+ \lambda I)^{-\frac{3}{2}}T^{\frac{1}{2}}\hat{C}_{n}\beta^*\|_{L^2} \\
        & \times \|(\hat{\Lambda}_{n}+ \lambda I)^{-\frac{1}{2}}(r_{\lambda}(\hat{\Lambda}_{n})+\lambda g_{\lambda}(\hat{\Lambda}_{n}))(\hat{\Lambda}_{n}+ \lambda I)^{\frac{3}{2}}\|_{L^2 \to L^2}.
    \end{split}
\end{equation}
Using equation $(\ref{testimate})$, Lemma \ref{Covariance estimation} and the regularization properties $(\ref{reg1}),(\ref{reg2}),(\ref{reg3})$ and $(\ref{qualification})$ in \eqref{power3/2} yields
\begin{equation}
\label{3/2powerbound}
    \begin{split}
    & \| T^{\frac{1}{2}}(g_{\lambda}(\hat{\Lambda}_{n})-(\hat{\Lambda}_{n}+ \lambda I)^{-1})T^{\frac{1}{2}}\hat{C}_{n}\beta^*\|_{L^2}\\
       \lesssim_{p} & \lambda \lambda^{-\frac{c}{2(t+c)}} \|(\hat{\Lambda}_{n}+ \lambda I)^{-\frac{3}{2}}T^{\frac{1}{2}}\hat{C}_{n}\beta^*\|_{L^2} \\
        \leq & \lambda \lambda^{-\frac{c}{2(t+c)}} \|(\hat{\Lambda}_{n}+ \lambda I)^{-\frac{1}{2}}(\Lambda+ \lambda I)^{\frac{1}{2}}\|_{L^2 \to L^2}(\|(\Lambda+ \lambda I)^{-\frac{3}{2}}T^{\frac{1}{2}}C\beta^*\|_{L^2}\\
         &  + \|(\Lambda+ \lambda I)^{-\frac{1}{2}}((\hat{\Lambda}_{n}+ \lambda I)^{-1}T^{\frac{1}{2}}\hat{C}_{n}\beta^*-(\Lambda+ \lambda I)^{-1}T^{\frac{1}{2}}C\beta^*)\|_{L^2}).
    \end{split}
\end{equation}
Similar to $(\ref{testimate})$, we can see that
\begin{equation*}
    \begin{split}
        \|(\Lambda+ \lambda I)^{-\frac{3}{2}}T^{\frac{1}{2}}C\beta^*\|_{L^2}  \leq & \|(\Lambda+ \lambda I)^{-\frac{3}{2}}T^{\frac{1}{2}}CT^{\alpha}\|_{L^2 \to L^2} \|h\|_{L^2}\\
        \lesssim &  \left(\sup_{i} \left|\frac{i^{-\frac{2}{3}(t(\alpha +\frac{1}{2})+c)}}{i^{-(t+c)}+\lambda}\right|\right)^{\frac{3}{2}} 
        \lesssim  \lambda^{\frac{2\alpha t - 2t - c}{2(t+c)}},
    \end{split}
\end{equation*}
thereby yielding
\begin{equation}
\label{2abound}
    \begin{split}
        & \| T^{\frac{1}{2}}(g_{\lambda}(\hat{\Lambda}_{n})-(\hat{\Lambda}_{n}+ \lambda I)^{-1})T^{\frac{1}{2}}\hat{C}_{n}\beta^*\|_{L^2}\\
        \lesssim_{p} & \sqrt{\lambda} \lambda^{-\frac{c}{2(t+c)}} \|(\hat{\Lambda}_{n}+ \lambda I)^{-1}T^{\frac{1}{2}}\hat{C}_{n}\beta^*-(\Lambda+ \lambda I)^{-1}T^{\frac{1}{2}}C\beta^*\|_{L^2} + \lambda^{\frac{\alpha t}{t+c}}.
    \end{split}
\end{equation}
Using similar ideas from $(\ref{2b})$ to bound the norm term in the r.h.s. of \eqref{2abound}, we get
\begin{equation*}
\begin{split}
         & \| T^{\frac{1}{2}}g_{\lambda}(\hat{\Lambda}_{n})T^{\frac{1}{2}}\hat{C}_{n}\beta^*-T^{\frac{1}{2}}(\hat{\Lambda}_{n}+ \lambda I)^{-1}T^{\frac{1}{2}}\hat{C}_{n}\beta^*\|_{L^2} 
         \lesssim_{p}  \textcolor{black}{\frac{\lambda^{-\frac{1-\alpha t}{(t+c)}}}{\sqrt{n}} } +\lambda^{\frac{\alpha t}{t+c}}.
    \end{split}
\end{equation*}
\textit{Bounding Term-2c:} Consider
\begin{equation*}
    \begin{split}
        & \|T^{\frac{1}{2}}(\Lambda + \lambda I)^{-1}T^{\frac{1}{2}}C\beta^* -\beta^* \|_{L^2}  = \|T^{\frac{1}{2}}(\Lambda + \lambda I)^{-1}T^{\frac{1}{2}}CT^{\alpha}h -T^{\alpha}h \|_{L^2} \\
        = & \left[\sum_{i}(\frac{\mu_{i}^{1+\alpha}\xi_{i}}{\mu_{i}\xi_{i}+\lambda}-\mu_{i}^{\alpha})^2\langle\phi_{i},h\rangle^2\right]^{\frac{1}{2}} =  \left[\sum_{i}(\frac{\lambda \mu_{i}^{\alpha}}{\mu_{i}\xi_{i}+\lambda})^2\langle\phi_{i},h\rangle^2\right]^{\frac{1}{2}} \\
        \leq & \lambda \sup_{i} \left[\frac{\mu_{i}^{\alpha}}{\mu_{i}\xi_{i}+\lambda}\right] \|h\|_{L^2}
        \lesssim  \lambda \sup_{i}\left[\frac{i^{-\alpha t}}{i^{-(t+c)}+\lambda}\right] \|h\|_{L^2}
        \lesssim  \lambda^{\frac{\alpha t}{t+c}}.
    \end{split}
\end{equation*}

\textbf{Case-2 }: Suppose $\alpha \geq \frac{1}{2}$. Let $\nu$ be the qualification of the regularization family and $\nu'$ be the greatest integer less than or equal to $\nu$. Consider
\begin{equation*}
    \begin{split}
        & \| T^{\frac{1}{2}}g_{\lambda}(\hat{\Lambda}_{n})T^{\frac{1}{2}}\hat{C}_{n}\beta^*-\beta^*\|_{L^2} 
        =  \| T^{\frac{1}{2}}g_{\lambda}(\hat{\Lambda}_{n})T^{\frac{1}{2}}\hat{C}_{n}T^{\frac{1}{2}}T^{\alpha-\frac{1}{2}}h-T^{\frac{1}{2}}T^{\alpha-\frac{1}{2}}h\|_{L^2} \\
        = & \|T^{\frac{1}{2}}r_{\lambda}(\hat{\Lambda}_{n})T^{\alpha-\frac{1}{2}}h\|_{L^2}\\
        = & \|T^{\frac{1}{2}}r_{\lambda}(\hat{\Lambda}_{n})(\hat{\Lambda}_{n}+ \lambda I)^{\nu}(\hat{\Lambda}_{n}+ \lambda I)^{-(\nu-\nu')}(\Lambda + \lambda I)^{\nu-\nu'} \\
        &  \qquad \qquad \qquad \times  (\Lambda + \lambda I)^{-(\nu-\nu')}(\hat{\Lambda}_{n}+ \lambda I)^{-\nu'}T^{\alpha-\frac{1}{2}}h\|_{L^2}\\
        \leq & \|T^{\frac{1}{2}}r_{\lambda}(\hat{\Lambda}_{n})(\hat{\Lambda}_{n}+ \lambda I)^{\nu}(\hat{\Lambda}_{n}+ \lambda I)^{-(\nu-\nu')}(\Lambda + \lambda I)^{\nu-\nu'}  \\
        &  \times ( (\Lambda + \lambda I)^{-(\nu-\nu')}(\hat{\Lambda}_{n}+ \lambda I)^{-\nu'}T^{\alpha-\frac{1}{2}}h - (\Lambda+ \lambda I)^{-\nu}T^{\alpha-\frac{1}{2}}h + (\Lambda+ \lambda I)^{-\nu}T^{\alpha-\frac{1}{2}}h )\|_{L^2}\\
        \lesssim_{p} &  \underbrace{\lambda^{\nu} \|(\Lambda + \lambda I)^{-(\nu-\nu')}(\hat{\Lambda}_{n}+ \lambda I)^{-\nu'}T^{\alpha-\frac{1}{2}}h - (\Lambda+ \lambda I)^{-\nu}T^{\alpha-\frac{1}{2}}h\|_{L^2}}_{\textit{Term-2d}}\\
        & + \underbrace{\|T^{\frac{1}{2}}r_{\lambda}(\hat{\Lambda}_{n})(\hat{\Lambda}_{n}+ \lambda I)^{\nu}(\hat{\Lambda}_{n}+ \lambda I)^{-(\nu-\nu')}(\Lambda + \lambda I)^{\nu-\nu'}(\Lambda+ \lambda I)^{-\nu}T^{\alpha-\frac{1}{2}}h\|_{L^2}}_{\textit{Term-2e}},
    \end{split}
    \end{equation*}
    where the last inequality is obtained by using 
Lemma \ref{Covariance estimation} and the qualification property $(\ref{qualification})$ of the regularization family.\vspace{2mm}\\
\noindent
\textit{Bounding Term-2d:}
Note that
\begin{equation}
\label{t4}
    \begin{split}
        & \lambda^{\nu} \|(\Lambda + \lambda I)^{-(\nu-\nu')}(\hat{\Lambda}_{n}+ \lambda I)^{-\nu'}T^{\alpha-\frac{1}{2}}h - (\Lambda+ \lambda I)^{-\nu}T^{\alpha-\frac{1}{2}}h \|_{L^2}\\
        \leq & \lambda^{\nu'} \|(\hat{\Lambda}_{n}+ \lambda I)^{-\nu'}T^{\alpha-\frac{1}{2}}h - (\Lambda+ \lambda I)^{-\nu'}T^{\alpha-\frac{1}{2}}h\|_{L^2}.
    \end{split}
\end{equation}
Using Lemma \ref{reducing}, we get
\begin{equation*}
    \begin{split}
        & \|(\hat{\Lambda}_{n}+ \lambda I)^{-\nu'}T^{\alpha-\frac{1}{2}}h - (\Lambda+ \lambda I)^{-\nu'}T^{\alpha-\frac{1}{2}}h\|_{L^2} \\
        \leq & \|(\hat{\Lambda}_{n}+\lambda I)^{-(\nu'  -1  )}[(\hat{\Lambda}_{n}+\lambda I)^{-1}T^{\alpha-\frac{1}{2}}h-(\Lambda+\lambda I)^{-1}T^{\alpha-\frac{1}{2}}h]\|_{L^2}\\
        & + \|\sum_{i=1}^{\nu'-1}(\hat{\Lambda}_{n}+\lambda I )^{-i}(\Lambda-\hat{\Lambda}_{n})(\Lambda+\lambda I )^{-( \nu'  +1-i)}T^{\alpha-\frac{1}{2}}h\|_{L^2}\\
        \lesssim & \frac{1}{\lambda^{\nu'-1}} \|(\hat{\Lambda}_{n}+\lambda I)^{-1}T^{\alpha-\frac{1}{2}}-(\Lambda+\lambda I)^{-1}T^{\alpha-\frac{1}{2}}\|_{L^2 \to L^2}\\
        & + \sum_{i=1}^{\nu'-1}\frac{1}{\lambda^{i-\frac{1}{2}}} \frac{1}{\lambda^{\nu'+1-i-\frac{1}{2}}}\|(\hat{\Lambda}_{n}+\lambda I)^{-\frac{1}{2}}(\Lambda - \hat{\Lambda}_{n})(\Lambda+\lambda I)^{-\frac{1}{2}}T^{\alpha-\frac{1}{2}}\|_{L^2 \to L^2}
        \end{split}
        \end{equation*}
        \begin{equation*}
        \begin{split}
        \lesssim & \frac{1}{\lambda^{\nu'}} \|(\hat{\Lambda}_{n}+\lambda I)^{-\frac{1}{2}}(\Lambda - \hat{\Lambda}_{n})(\Lambda+\lambda I)^{-\frac{1}{2}}\|_{L^2 \to L^2}.
            \end{split}
\end{equation*}
By using $p=\frac{1}{2}$ and $\alpha = \frac{1}{2}$ in Lemma \ref{cmdifference}, we get
\begin{equation*}
\|(\hat{\Lambda}_{n}+ \lambda I)^{-\nu'}T^{\alpha-\frac{1}{2}}h - (\Lambda+ \lambda I)^{-\nu'}T^{\alpha-\frac{1}{2}}h\|_{L^2} 
        \lesssim_{p}  \frac{1}{\lambda^{\nu'}} \textcolor{black}{\frac{\lambda^{-\frac{1}{(t+c)}}}{\sqrt{n}}}.
\end{equation*}
\noindent
Using this bound in $(\ref{t4})$ yields
\begin{equation}
\label{empiricaltocovariance}
       \lambda^{\nu} \|(\Lambda + \lambda I)^{-(\nu-\nu')}(\hat{\Lambda}_{n}+ \lambda I)^{-\nu'}T^{\alpha-\frac{1}{2}}h - (\Lambda+ \lambda I)^{-\nu}T^{\alpha-\frac{1}{2}}h \|_{L^2}
       \lesssim_{p} \textcolor{black}{\frac{\lambda^{-\frac{1}{(t+c)}}}{\sqrt{n}}.}
\end{equation}\\
\noindent
\textit{Bounding Term-2e:}
\begin{equation*}
\begin{split}
    & \|T^{\frac{1}{2}}r_{\lambda}(\hat{\Lambda}_{n})(\hat{\Lambda}_{n}+ \lambda I)^{\nu}(\hat{\Lambda}_{n}+ \lambda I)^{-(\nu-\nu')}(\Lambda + \lambda I)^{\nu-\nu'}(\Lambda+ \lambda I)^{-\nu}T^{\alpha-\frac{1}{2}}h\|_{L^2}\\
    \leq & \|T^{\frac{1}{2}}(\Lambda+ \lambda I)^{-\frac{1}{2}}\|_{L^2 \to L^2} \|(\Lambda+ \lambda I)^{\frac{1}{2}} (\hat{\Lambda}_{n}+ \lambda I)^{-\frac{1}{2}}\|_{L^2 \to L^2}\\
    & \times \|(\hat{\Lambda}_{n}+ \lambda I)^{\frac{1}{2}}r_{\lambda}(\hat{\Lambda}_{n})(\hat{\Lambda}_{n}+ \lambda I)^{\nu-\frac{1}{2}}\|_{L^2 \to L^2}
    \|(\Lambda+ \lambda I)^{-(\nu-\frac{1}{2})}T^{\alpha-\frac{1}{2}}h\|_{L^2}\\
    & \times \|(\hat{\Lambda}_{n}+ \lambda I)^{\frac{1}{2}}(\hat{\Lambda}_{n}+ \lambda I)^{-(\nu-\nu')}(\Lambda + \lambda I)^{\nu-\nu'}(\Lambda + \lambda I)^{-\frac{1}{2}}\|_{L^2 \to L^2}.
\end{split}
\end{equation*}
Using $(\ref{testimate})$, Lemma \ref{Covariance estimation}, Lemma \ref{Empirical bound} and the qualification property $(\ref{qualification})$ of regularization family, we get 

\begin{equation*}
    \begin{split}
        & \|T^{\frac{1}{2}}r_{\lambda}(\hat{\Lambda}_{n})(\hat{\Lambda}_{n}+ \lambda I)^{\nu}(\hat{\Lambda}_{n}+ \lambda I)^{-(\nu-\nu')}(\Lambda + \lambda I)^{\nu-\nu'}(\Lambda+ \lambda I)^{-\nu}T^{\alpha-\frac{1}{2}}h\|_{L^2} \\ 
        &  \qquad \qquad \lesssim_{p} \lambda^{-\frac{c}{2(t+c)}} \lambda^{\nu}  \|(\Lambda+ \lambda I)^{-(\nu-\frac{1}{2})}T^{\alpha-\frac{1}{2}}h\|_{L^2},
    \end{split}
\end{equation*}
where
\begin{equation*}
    \begin{split}
        &\|(\Lambda+ \lambda I)^{-(\nu-\frac{1}{2})}T^{\alpha-\frac{1}{2}}h\|_{L^2}  \leq \|(\Lambda+ \lambda I)^{-(\nu-\frac{1}{2})}T^{\alpha-\frac{1}{2}}\|_{L^2 \to L^2}\|h\|_{L^2}\\
        & \leq \sup_{i}\left|\frac{\mu_{i}^{\alpha-\frac{1}{2}}}{(\mu_{i} \xi_{i}+\lambda)^{(\nu-\frac{1}{2})}}\right| \|h\|_{L^2}
         \lesssim \sup_{i} \left|\frac{i^{-t(\alpha-\frac{1}{2})}}{(i^{-t-c}+\lambda)^{(\nu-\frac{1}{2})}}\right| 
         \lesssim  \lambda^{\frac{r t-\frac{t}{2}-(t+c)(\nu-\frac{1}{2})} {t+c}},
    \end{split}
\end{equation*}
for  $r = \min\{\alpha , \nu + \frac{c}{t} (\nu -\frac{1}{2}) \}$,
where the last inequality follows from Lemma \ref{supbound} with $(a,m,p,q,l) = (\frac{1}{2}, t, \alpha,t+c, \nu-\frac{1}{2})$,
resulting in
\begin{equation*}
    \begin{split}
        & \|T^{\frac{1}{2}}r_{\lambda}(\hat{\Lambda}_{n})(\hat{\Lambda}_{n}+ \lambda I)^{\nu}(\hat{\Lambda}_{n}+ \lambda I)^{-(\nu-\nu')}(\Lambda + \lambda I)^{\nu-\nu'}(\Lambda+ \lambda I)^{-\nu}T^{\alpha-\frac{1}{2}}h\|_{L^2}\\
        & \qquad \qquad \qquad \qquad \qquad \qquad \qquad \qquad \lesssim_{p} \lambda^{\frac{r t}{t+c}}.
    \end{split}
\end{equation*}
Combining all the bounds we get
\begin{equation*}
\label{fb}
    \begin{split}
        \|\beta^*-\hat{\beta}\|_{L^2} \lesssim_{p}
        \begin{cases}
            \lambda^{-\frac{c}{2(t+c)}} \sqrt{\frac{ \sigma^2 \mathcal{N}(\lambda)}{n }} + \lambda^{\frac{\alpha t}{t+c}} + \textcolor{black}{\frac{\lambda^{-\frac{1-\alpha t}{(t+c)}}}{\sqrt{n}} }& \mbox{ if } \quad 0 < \alpha \leq \frac{1}{2} \\
            \lambda^{-\frac{c}{2(t+c)}} \sqrt{\frac{ \sigma^2 \mathcal{N}(\lambda)}{n }} + \lambda^{\frac{\alpha t}{t+c}}+ \textcolor{black}{\frac{\lambda^{-\frac{1}{ (t+c)}}}{\sqrt{n}} }&  \mbox{ if } \quad \frac{1}{2} \leq \alpha \leq \nu + \frac{c}{t} (\nu -\frac{1}{2})\\
            \lambda^{-\frac{c}{2(t+c)}} \sqrt{\frac{ \sigma^2 \mathcal{N}(\lambda)}{n }} + \lambda^{\frac{(\nu + \frac{c}{t} (\nu -\frac{1}{2})) t}{t+c}}+ \textcolor{black}{\frac{\lambda^{-\frac{1}{(t+c)}}}{\sqrt{n}} }&  \mbox{ if } \quad \alpha \geq \nu + \frac{c}{t} (\nu -\frac{1}{2})
        \end{cases}, 
    \end{split}
\end{equation*}
where the effective dimension $\mathcal{N}(\lambda)$ can be bounded as
\begin{equation*}
    \begin{split}
        \mathcal{N}(\lambda ) & = \text{trace}(\Lambda (\Lambda+\lambda I)^{-1})
         = \sum_{i}\frac{\mu_{i}\xi_{i}}{\mu_{i}\xi_{i}+\lambda} \lesssim \sum_{i} \frac{i^{-(t+c)}}{i^{-(t+c)}+\lambda} \lesssim \lambda^{-\frac{1}{t+c}}.
    \end{split}
\end{equation*}
Therefore,
\begin{align*}
    \|\hat{\beta}-\beta^*\|_{L^2} &\lesssim_{p} 
    \frac{\lambda^{-\frac{(c+1)}{2(t+c)}}}{\sqrt{n}}+ \lambda^{\frac{r t}{t+c}}  
    \lesssim n^{-\frac{r t }{1+c+2 r t}} 
\end{align*}
for $\lambda = n^{-\frac{t+c}{1+c+2 r t}}$ and the result follows.

\subsection{Proof of part \ref{b} in Theorem \ref{comm.estimation}}\label{pcomm-est2}
    We start by considering the error term in the RKHS norm and applying Lemma \ref{J T}. 
    \begin{equation*}
    \begin{split}
        \|\hat{\beta}-\beta^*\|_{\mathcal{H}}  = & \|J \hat{\beta}-\beta^*\|_{\mathcal{H}}= \|Jg_{\lambda}(J^*\hat{C}_{n}J)J^*\hat{R}-\beta^*\|_{\mathcal{H}}\\
         = & \|T^{\frac{1}{2}}g_{\lambda}(\hat{\Lambda}_{n})T^{\frac{1}{2}}\hat{R}-\beta^*\|_{\mathcal{H}}
         =  \|T^{\frac{1}{2}}g_{\lambda}(\hat{\Lambda}_{n})T^{\frac{1}{2}}\hat{R}-T^{\alpha}h\|_{\mathcal{H}}\\
         \leq & \|T^{\frac{1}{2}}\|_{L^2 \to \mathcal{H}} \|g_{\lambda}(\hat{\Lambda}_{n})T^{\frac{1}{2}}\hat{R}-T^{\alpha-\frac{1}{2}}h\|_{L^2} \quad (\|T^{\frac{1}{2}}\|_{L^2 \to \mathcal{H}} = 1)\\
         \leq & \|g_{\lambda}(\hat{\Lambda}_{n})(T^{\frac{1}{2}}\hat{R} - T^{\frac{1}{2}}\hat{C}_{n}\beta^*)\|_{L^2} + \|g_{\lambda}(\hat{\Lambda}_{n})T^{\frac{1}{2}}\hat{C}_{n}\beta^*-T^{\alpha-\frac{1}{2}}h\|_{L^2}\\
         \leq & ( \|g_{\lambda}(\hat{\Lambda}_{n})(\hat{\Lambda}_{n}+\lambda I)^{\frac{1}{2}}\|_{L^2 \to L^2} \|(\hat{\Lambda}_{n}+\lambda I)^{-\frac{1}{2}}(\Lambda+\lambda I)^{\frac{1}{2}}\|_{L^2 \to L^2}\\
        & \times \|(\Lambda+\lambda I)^{-\frac{1}{2}}(T^{\frac{1}{2}}\hat{R} - T^{\frac{1}{2}}\hat{C}_{n}\beta^*)\|_{L^2}) + \|g_{\lambda}(\hat{\Lambda}_{n})T^{\frac{1}{2}}\hat{C}_{n}\beta^*-T^{\alpha-\frac{1}{2}}h\|_{L^2}.
         \end{split}
         \end{equation*}
    From Lemma~\ref{Empirical bound}, Lemma \ref{Covariance estimation} and regularization property $(\ref{reg2})$, we get
         \begin{equation*}
         \|\hat{\beta}-\beta^*\|_{\mathcal{H}}
          \lesssim_{p} \frac{1}{\sqrt{\lambda}} \frac{\lambda^{-\frac{1}{2(t+c)}}}{\sqrt{n}} + \|g_{\lambda}(\hat{\Lambda}_{n})T^{\frac{1}{2}}\hat{C}_{n}\beta^*-T^{\alpha-\frac{1}{2}}h\|_{L^2}.
    \end{equation*}
    Now we bound the second term in the above inequality as
    \begin{equation*}
        \begin{split}
& \|g_{\lambda}(\hat{\Lambda}_{n})T^{\frac{1}{2}}\hat{C}_{n}\beta^*-T^{\alpha-\frac{1}{2}}h\|_{L^2}  = \|r_{\lambda}(\hat{\Lambda}_{n})T^{\alpha-\frac{1}{2}}h\|_{L^2} \\
= &\|r_{\lambda}(\hat{\Lambda}_{n})(\hat{\Lambda}_{n}+\lambda I)^{\nu'}(\hat{\Lambda}_{n}+\lambda I)^{-\nu'}T^{\alpha-\frac{1}{2}}h\|_{L^2}\\
= & \|r_{\lambda}(\hat{\Lambda}_{n})(\hat{\Lambda}_{n}+ \lambda I)^{\nu}(\hat{\Lambda}_{n}+ \lambda I)^{-(\nu-\nu')}(\Lambda + \lambda I)^{\nu-\nu'}  \\
        &  \qquad \qquad \qquad \times  (\Lambda + \lambda I)^{-(\nu-\nu')}(\hat{\Lambda}_{n}+ \lambda I)^{-\nu'}T^{\alpha-\frac{1}{2}}h\|_{L^2}\\
        = & \|r_{\lambda}(\hat{\Lambda}_{n})(\hat{\Lambda}_{n}+ \lambda I)^{\nu}(\hat{\Lambda}_{n}+ \lambda I)^{-(\nu-\nu')}(\Lambda + \lambda I)^{\nu-\nu'}  \\
        &  \times  ( (\Lambda + \lambda I)^{-(\nu-\nu')}(\hat{\Lambda}_{n}+ \lambda I)^{-\nu'}T^{\alpha-\frac{1}{2}}h - (\Lambda+ \lambda I)^{-\nu}T^{\alpha-\frac{1}{2}}h + (\Lambda+ \lambda I)^{-\nu}T^{\alpha-\frac{1}{2}}h)\|_{L^2}\\
        \lesssim_{p} & \lambda^{\nu}\underbrace{\|(\Lambda + \lambda I)^{-(\nu-\nu')}(\hat{\Lambda}_{n}+ \lambda I)^{-\nu'}T^{\alpha-\frac{1}{2}}h - (\Lambda+ \lambda I)^{-\nu}T^{\alpha-\frac{1}{2}}h\|_{L^2}}_{\textit{Term-1}} \\
        & + \lambda^{\nu} \underbrace{\|(\Lambda+ \lambda I)^{-\nu}T^{\alpha-\frac{1}{2}}h\|_{L^2}}_{\textit{Term-2}}.
\end{split}
\end{equation*}
The last inequality follows from Lemma \ref{Covariance estimation} and the qualification property $(\ref{qualification})$ of the regularization family.\\

\noindent
\textit{Bounding Term-1:}
\begin{equation*}
    \begin{split}
        & \|(\Lambda + \lambda I)^{-(\nu-\nu')}(\hat{\Lambda}_{n}+ \lambda I)^{-\nu'}T^{\alpha-\frac{1}{2}}h - (\Lambda+ \lambda I)^{-\nu}T^{\alpha-\frac{1}{2}}h\|_{L^2} \\
        \leq & \frac{1}{\lambda^{(\nu-\nu')}} \|(\hat{\Lambda}_{n}+ \lambda I)^{-\nu'}T^{\alpha-\frac{1}{2}}h - (\Lambda+ \lambda I)^{-\nu'}T^{\alpha-\frac{1}{2}}h\|_{L^2}.
    \end{split}
\end{equation*}
From $(\ref{empiricaltocovariance})$, we get
\begin{equation*}
    \begin{split}
        & \|(\Lambda + \lambda I)^{-(\nu-\nu')}(\hat{\Lambda}_{n}+ \lambda I)^{-\nu'}T^{\alpha-\frac{1}{2}}h - (\Lambda+ \lambda I)^{-\nu}T^{\alpha-\frac{1}{2}}h\|_{L^2} \\ 
        \lesssim_{p} & \frac{1}{\lambda^{(\nu-\nu')}} \frac{1}{\lambda^{\nu'}} \textcolor{black}{\frac{\lambda^{-\frac{1}{(t+c)}}}{\sqrt{n}}} = {\lambda^{-\nu}} \textcolor{black}{\frac{\lambda^{-\frac{1}{(t+c)}}}{\sqrt{n}}.}
    \end{split}
\end{equation*}
\textit{Bounding Term-2:} We apply Lemma \ref{supbound} with $(a,m,p,q,l) = (\frac{1}{2},t,\alpha,t+c,\nu)$ to get
\begin{equation*}
    \begin{split}
        \|(\Lambda+ \lambda I)^{-\nu}T^{\alpha-\frac{1}{2}}h\|_{L^2} & \lesssim \sup_{i} \left|\frac{\mu_{i}^{\alpha-\frac{1}{2}}}{(\mu_{i}\xi_{i}+\lambda)^{\nu}}\right| 
         \lesssim  \sup_{i} \left|\frac{i^{-t(\alpha -\frac{1}{2})}}{(i^{-(t+c)}+\lambda)^{\nu}}\right|
         \lesssim  \lambda^{\frac{t(r-\frac{1}{2})-(t+c)\nu}{t+c}},
    \end{split}
\end{equation*}
where $r = \min \{\alpha, \nu + \frac{c}{t}\nu +\frac{1}{2}\}$.\\

\noindent
Combining the bounds on Term-1 and Term-2, we have
\begin{equation*}
    \begin{split}
        \|g_{\lambda}(\hat{\Lambda}_{n})T^{\frac{1}{2}}\hat{C}_{n}\beta^*-T^{\alpha-\frac{1}{2}}h\|_{L^2}  = \|r_{\lambda}(\hat{\Lambda}_{n})T^{\alpha-\frac{1}{2}}h\|_{L^2} \lesssim_{p} \textcolor{black}{\frac{\lambda^{-\frac{1}{(t+c)}}}{\sqrt{n}} }+ \lambda^{\frac{t(r -\frac{1}{2})}{t+c}}.
    \end{split}
\end{equation*}
resulting in
\begin{equation*}
    \begin{split}
        \|\hat{\beta}-\beta^*\|_{\mathcal{H}} \lesssim_{p} \lambda^{\frac{t(r -\frac{1}{2})}{t+c}} + \frac{1}{\sqrt{\lambda}} \frac{\lambda^{-\frac{1}{2(t+c)}}}{\sqrt{n}}.
    \end{split}
\end{equation*}
The result follows by choosing 
$\lambda = n^{-\frac{t+c}{1+c+ 2 r t}}$.

\subsection{Proof of Theorem \ref{NCestimation}}\label{pcomm-nc}
    Consider  
	\begin{equation*}
        \begin{split}
	       & \| \hat{\beta} - \beta^* \|_{\mathcal{H}}  =  \| J \hat{\beta} -\beta^* \|_{\mathcal{H}} 
            =  \|Jg_{\lambda}(J^*\hat{C}_{n}J)J^*\hat{R}-\beta^*\|_{\mathcal{H}}\\
            = & \| T^{\frac{1}{2}}g_{\lambda}(\hat{\Lambda}_{n})T^{\frac{1}{2}}\hat{R}  -\beta^* \|_{\mathcal{H}} \\
          \leq & \| T^{\frac{1}{2}}g_{\lambda}(\hat{\Lambda}_{n})(T^{\frac{1}{2}}\hat{R} - T^{\frac{1}{2}}\hat{C}_{n}\beta^*)\|_{\mathcal{H}}
         + \| T^{\frac{1}{2}}g_{\lambda}(\hat{\Lambda}_{n})T^{\frac{1}{2}}\hat{C}_{n}\beta^*-\beta^*\|_{\mathcal{H}}.
        \end{split}
	\end{equation*}
The first term on the r.h.s. can be bounded as
\begin{equation*}
    \begin{split}
        & \| T^{\frac{1}{2}}g_{\lambda}(\hat{\Lambda}_{n})(T^{\frac{1}{2}}\hat{R} - T^{\frac{1}{2}}\hat{C}_{n}\beta^*)\|_{\mathcal{H}}\\
        \leq & \|T^{\frac{1}{2}}\|_{L^2 \to \mathcal{H}}\|g_{\lambda}(\hat{\Lambda}_{n})(\hat{\Lambda}_{n}+\lambda I)^{\frac{1}{2}}\|_{L^2 \to L^2} \|(\hat{\Lambda}_{n}+\lambda I)^{-\frac{1}{2}}(\Lambda +\lambda I)^{\frac{1}{2}}\| _{L^2 \to L^2}\\
        & \quad \times \|(\Lambda +\lambda I)^{-\frac{1}{2}}(T^{\frac{1}{2}}\hat{R} - T^{\frac{1}{2}}\hat{C}_{n}\beta^*)\|_{L^2}.
    \end{split}
\end{equation*}
Using regularization properties $(\ref{reg2})$, Lemma \ref{Covariance estimation} and Lemma \ref{Empirical bound}, we get
\begin{equation*}
         \| T^{\frac{1}{2}}g_{\lambda}(\hat{\Lambda}_{n})(T^{\frac{1}{2}}\hat{R} - T^{\frac{1}{2}}\hat{C}_{n}\beta^*)\|_{\mathcal{H}}
        \lesssim_{p}  \frac{1}{\sqrt{\lambda}} \sqrt{\frac{ \sigma^2 \mathcal{N}(\lambda)}{n}}.
\end{equation*}
Now we consider
\begin{equation*}
    \begin{split}
        & \| T^{\frac{1}{2}}g_{\lambda}(\hat{\Lambda}_{n})T^{\frac{1}{2}}\hat{C}_{n}\beta^*-\beta^*\|_{\mathcal{H}} = \| T^{\frac{1}{2}}g_{\lambda}(\hat{\Lambda}_{n})T^{\frac{1}{2}}\hat{C}_{n}T^{\frac{1}{2}} \Lambda^s h  -T^{\frac{1}{2}} \Lambda^s h\|_{\mathcal{H}} \\
        = & \|T^{\frac{1}{2}} (g_{\lambda}(\hat{\Lambda}_{n})\hat{\Lambda}_{n}-I)\Lambda^s h\|_{\mathcal{H}} = \|T^{\frac{1}{2}}r_{\lambda}(\hat{\Lambda}_{n})\Lambda^s h\|_{\mathcal{H}} \\
        = & \|T^{\frac{1}{2}}r_{\lambda}(\hat{\Lambda}_{n})(\hat{\Lambda}_{n}+ \lambda I)^{\nu}(\hat{\Lambda}_{n}+ \lambda I)^{-(\nu-\nu')}(\hat{\Lambda}_{n}+ \lambda I)^{-\nu'}\Lambda^s h\|_{\mathcal{H}}\\
        = & \|T^{\frac{1}{2}}r_{\lambda}(\hat{\Lambda}_{n})(\hat{\Lambda}_{n}+ \lambda I)^{\nu}(\hat{\Lambda}_{n}+ \lambda I)^{-(\nu-\nu')}(\Lambda+ \lambda I)^{(\nu-\nu')}\\
        & \times [(\Lambda+ \lambda I)^{-(\nu-\nu')}(\hat{\Lambda}_{n}+ \lambda I)^{-\nu'}\Lambda^s h- (\Lambda+ \lambda I)^{-\nu}\Lambda^s h+ (\Lambda+ \lambda I)^{-\nu}\Lambda^s h ]\|_{\mathcal{H}}\\
        \lesssim_{p} & \underbrace{\lambda^{\nu} \|(\Lambda+ \lambda I)^{-(\nu-\nu')}(\hat{\Lambda}_{n}+ \lambda I)^{-\nu'}\Lambda^s h - (\Lambda+ \lambda I)^{-\nu}\Lambda^s h\|_{L^2}}_{Term-3} \\
        & \qquad \qquad + \underbrace{\lambda^{\nu}\| (\Lambda+ \lambda I)^{-\nu}\Lambda^s h\|_{L^2}}_{Term-4},
    \end{split}
 \end{equation*}
where we used Lemma \ref{Covariance estimation} and the qualification properties $(\ref{qualification})$ of the regularization family in the last equality.\\

\noindent
\textit{Bounding Term-3:}
\begin{equation*}
    \begin{split}
    & \lambda^{\nu} \|(\Lambda+ \lambda I)^{-(\nu-\nu')}(\hat{\Lambda}_{n}+ \lambda I)^{-\nu'}\Lambda^s h - (\Lambda+ \lambda I)^{-\nu}\Lambda^s h\|_{L^2}\\
        \leq & \lambda^{\nu} \|(\Lambda+ \lambda I)^{-(\nu-\nu')}((\hat{\Lambda}_{n}+ \lambda I)^{-\nu'}\Lambda^s h - (\Lambda+ \lambda I)^{-\nu'}\Lambda^s h)\|_{L^2} \\
        \leq & \lambda^{\nu'} \|(\hat{\Lambda}_{n}+ \lambda I)^{-\nu'}\Lambda^s h - (\Lambda+ \lambda I)^{-\nu'}\Lambda^s h\|_{L^2}.\\
    \end{split}
\end{equation*}
Using Lemma \ref{reducing}, we get
\begin{equation*}
    \begin{split}
        & \|(\hat{\Lambda}_{n}+ \lambda I)^{-\nu'}\Lambda^s h - (\Lambda+ \lambda I)^{-\nu'}\Lambda^s h\|_{L^2}\\
        \leq & \lambda^{-(\nu'-1)} \|(\hat{\Lambda}_{n}+ \lambda I)^{-1}\Lambda^s h - (\Lambda+ \lambda I)^{-1}\Lambda^s h\|_{L^2} \\
        & + \left\|\sum_{i=1}^{\nu'-1}(\hat{\Lambda}_{n}+ \lambda I)^{-i}(\Lambda -\hat{\Lambda}_{n})(\Lambda+ \lambda I)^{-(\nu' + 1 -i )}\Lambda^s h\right\|_{L^2} \\
        \lesssim & \lambda^{-(\nu'-1)} \|(\hat{\Lambda}_{n}+ \lambda I)^{-1} (\Lambda - \hat{\Lambda}_{n})(\Lambda+ \lambda I)^{-1}\Lambda^s \|_{L^2 \to L^2}\\
        & + \sum_{i=1}^{\nu'-1}\frac{1}{\lambda^{i-\frac{1}{2}}}\frac{1}{\lambda^{\nu'+1-i-\frac{1}{2}}}\|(\hat{\Lambda}_{n}+ \lambda I)^{-\frac{1}{2}} (\Lambda - \hat{\Lambda}_{n})(\Lambda+ \lambda I)^{-\frac{1}{2}}\|_{L^2 \to L^2}\\
        \lesssim & \lambda^{-\nu'} \|(\hat{\Lambda}_{n}+ \lambda I)^{-\frac{1}{2}} (\Lambda - \hat{\Lambda}_{n})(\Lambda+ \lambda I)^{-\frac{1}{2}}\|_{L^2 \to L^2}\lesssim_{p} \lambda^{-\nu'} \textcolor{black}{\frac{\lambda^{-\frac{1}{b}}}{\sqrt{n}}, }
    \end{split}
\end{equation*}
where the last inequality follows from Lemma \ref{ncmdifference}.\vspace{1mm}\\
\noindent
\textit{Bounding Term-4:} We apply Lemma \ref{supbound} with $(a,m,p,q,l) = (0,b,s,b,\nu)$ to get
\begin{equation*}
    \begin{split}
       &\lambda^{\nu} \|(\Lambda+\lambda I)^{-\nu}\Lambda^s h\|_{L^2} \leq  \lambda^{\nu} \|(\Lambda+\lambda I)^{-\nu}\Lambda^s\|_{L^2 \to L^2} \|h\|_{L^2} \\
        \lesssim & \lambda^{\nu} \sup_{i} \left|\frac{\tau_{i}^{s}}{(\tau_{i}+\lambda)^{\nu}}\right|
        \lesssim \lambda^{\nu}\sup_{i}\left|\frac{i^{-sb}}{(i^{-b}+\lambda)^{\nu}}\right|
        \leq  \lambda^{\nu} \lambda^{\frac{rb-b \nu}{b}} = \lambda^{r},
    \end{split}
\end{equation*}
where $r = \min \{s, \nu\}$.\\
\indent Combing all the bounds, we get
\begin{equation*}
    \begin{split}
        \|\beta^*-\hat{\beta}\|_{\mathcal{H}} \lesssim_{p} \frac{1}{\sqrt{\lambda}} \sqrt{\frac{ \sigma^2 \mathcal{N}(\lambda)}{n}} + \textcolor{black}{\frac{\lambda^{-\frac{1}{b}}}{\sqrt{n}} }+ \lambda^r,
    \end{split}
\end{equation*}
where
\begin{equation*}
    \begin{split}
        \mathcal{N}(\lambda) = & \text{trace}(\Lambda (\Lambda + \lambda I)^{-1}) 
        = \sum_{i} \left[\frac{\tau_{i}}{\tau_{i}+\lambda}\right] \lesssim \sum_{i} \left[\frac{i^{-b}}{i^{-b}+\lambda}\right] \lesssim \lambda^{-\frac{1}{b}}   
    \end{split}
\end{equation*}
resulting in
\begin{equation*}
    \begin{split}
        \|\hat{\beta} - \beta^* \|_{\mathcal{H}} \lesssim_{p} &  \frac{\lambda^{-(\frac{1}{2b}+\frac{1}{2})}}{\sqrt{n}} + \lambda^r . 
    \end{split}
\end{equation*}
The result follows for 
$\lambda = n^{-\frac{b}{1+b+2rb}}$.

\subsection{Proof of Theorem \ref{predictionbasic}}\label{pcomm-pred1}
    Consider
    \begin{equation*}
        \begin{split}
           \|C^{\frac{1}{2}}(\hat{\beta}-\beta^*)\|_{L^2} = &
           \|C^{\frac{1}{2}}(J \hat{\beta}-\beta^*)\|_{L^2} 
           =  \|C^{\frac{1}{2}}(Jg_{\lambda}(J^*\hat{C}_{n}J)J^*\hat{R}-\beta^*)\|_{L^2}.
        \end{split}
    \end{equation*}
Using Lemma \ref{J T}, we get
\begin{equation*}
    \begin{split}
       & \|C^{\frac{1}{2}}(\hat{\beta}-\beta^*)\|_{L^2} = \|C^{\frac{1}{2}}(T^{\frac{1}{2}}g_{\lambda}(\hat{\Lambda}_{n})T^{\frac{1}{2}}\hat{R}-\beta^*)\|_{L^2} \\
       \leq & \|C^{\frac{1}{2}}(T^{\frac{1}{2}}g_{\lambda}(\hat{\Lambda}_{n})T^{\frac{1}{2}}\hat{R}- T^{\frac{1}{2}}g_{\lambda}(\hat{\Lambda}_{n})T^{\frac{1}{2}}\hat{C}_{n}\beta^*)\|_{L^2}\nonumber\\
       &\qquad\qquad\qquad+\|C^{\frac{1}{2}}(T^{\frac{1}{2}}g_{\lambda}(\hat{\Lambda}_{n})T^{\frac{1}{2}}\hat{C}_{n}\beta^*-\beta^*)\|_{L^2} \\
       = & \|C^{\frac{1}{2}}T^{\frac{1}{2}}g_{\lambda}(\hat{\Lambda}_{n})(T^{\frac{1}{2}}\hat{R}-T^{\frac{1}{2}}\hat{C}_{n}\beta^*)\|_{L^2}+\|C^{\frac{1}{2}}(T^{\frac{1}{2}}g_{\lambda}(\hat{\Lambda}_{n})T^{\frac{1}{2}}\hat{C}_{n}\beta^*-\beta^*)\|_{L^2} \\
       = & \|\Lambda^{\frac{1}{2}}g_{\lambda}(\hat{\Lambda}_{n})(T^{\frac{1}{2}}\hat{R}-T^{\frac{1}{2}}\hat{C}_{n}\beta^*)\|_{L^2}+\|C^{\frac{1}{2}}(T^{\frac{1}{2}}g_{\lambda}(\hat{\Lambda}_{n})T^{\frac{1}{2}}\hat{C}_{n}\beta^*-\beta^*)\|_{L^2} \\
       \leq & \|\Lambda^{\frac{1}{2}} (\Lambda + \lambda I)^{-\frac{1}{2}}\|_{L^2 \to L^2} \|(\Lambda + \lambda I)^{\frac{1}{2}} (\hat{\Lambda}_{n} + \lambda I)^{-\frac{1}{2}}\|_{L^2 \to L^2} \\
       &\times \|(\hat{\Lambda}_{n} + \lambda I)^{\frac{1}{2}}g_{\lambda}(\hat{\Lambda}_{n})(\hat{\Lambda}_{n} + \lambda I)^{\frac{1}{2}}\|_{L^2 \to L^2}\\
       & \times \|(\hat{\Lambda}_{n} + \lambda I)^{-\frac{1}{2}}(\Lambda + \lambda I)^{\frac{1}{2}}\|_{L^2 \to L^2} \|(\Lambda + \lambda I)^{-\frac{1}{2}}(T^{\frac{1}{2}}\hat{R}-T^{\frac{1}{2}}\hat{C}_{n}\beta^*)\|_{L^2} \\
       &  + \|C^{\frac{1}{2}}(T^{\frac{1}{2}}g_{\lambda}(\hat{\Lambda}_{n})T^{\frac{1}{2}}\hat{C}_{n}\beta^*-\beta^*)\|_{L^2}.
    \end{split}
\end{equation*}
Using Lemma \ref{Covariance estimation}, Lemma \ref{Empirical bound} and regularization properties $(\ref{reg1}),(\ref{reg2})$, we get
\begin{equation}
\label{nbound}
    \|C^{\frac{1}{2}}(\hat{\beta}-\beta^*)\|_{L^2} 
        \lesssim_{p}  \frac{\lambda^{-\frac{1}{2(t+c)}}}{\sqrt{n}} + \|C^{\frac{1}{2}}(T^{\frac{1}{2}}g_{\lambda}(\hat{\Lambda}_{n})T^{\frac{1}{2}}\hat{C}_{n}\beta^*-\beta^*)\|_{L^2}.
\end{equation}
Since $\beta^*$ appearing in the second term on the r.h.s. of \eqref{nbound} depends on the source condition, we split the analysis into two cases depending on the value of $\alpha$.\\

\noindent
\textbf{Case-1 :} Let $\alpha\in(0,\frac{1}{2}]$.
\begin{equation*}
    \begin{split}
       & \|C^{\frac{1}{2}}(T^{\frac{1}{2}}g_{\lambda}(\hat{\Lambda}_{n})T^{\frac{1}{2}}\hat{C}_{n}\beta^*-\beta^*)\|_{L^2}\\
       \leq &  \underbrace{\|C^{\frac{1}{2}}T^{\frac{1}{2}}(g_{\lambda}(\hat{\Lambda}_{n})-(\hat{\Lambda}_{n}+\lambda I)^{-1})T^{\frac{1}{2}}\hat{C}_{n}\beta^*\|_{L^2}}_{\textit{Term-5}} +\underbrace{\|C^{\frac{1}{2}}(T^{\frac{1}{2}}(\Lambda+\lambda I)^{-1}T^{\frac{1}{2}}C\beta^*-\beta^*)\|_{L^2}}_{\textit{Term-6}}\\
       & + \underbrace{\|C^{\frac{1}{2}}T^{\frac{1}{2}}((\hat{\Lambda}_{n}+\lambda I)^{-1}T^{\frac{1}{2}}\hat{C}_{n}\beta^*-(\Lambda+\lambda I)^{-1}T^{\frac{1}{2}}C\beta^*)\|_{L^2}}_{\textit{Term-7}}.
    \end{split}
\end{equation*}
\noindent
\textit{Bounding Term-5:}
\begin{equation*}
    \begin{split}
         & \|C^{\frac{1}{2}}T^{\frac{1}{2}}(g_{\lambda}(\hat{\Lambda}_{n})-(\hat{\Lambda}_{n}+\lambda I)^{-1})T^{\frac{1}{2}}\hat{C}_{n}\beta^*\|_{L^2} \\
         \leq & \|\Lambda^{\frac{1}{2}}(\hat{\Lambda}_{n}+\lambda I)^{-1}r_{\lambda}(\hat{\Lambda}_{n})T^{\frac{1}{2}}\hat{C}_{n}\beta^*\|_{L^2}+ \lambda \|\Lambda^{\frac{1}{2}}(\hat{\Lambda}_{n}+\lambda I)^{-1}g_{\lambda}(\hat{\Lambda}_{n})T^{\frac{1}{2}}\hat{C}_{n}\beta^*\|_{L^2} \\
         \lesssim & \left[\|(\hat{\Lambda}_{n}+\lambda I)^{-\frac{1}{2}}r_{\lambda}(\hat{\Lambda}_{n})(\hat{\Lambda}_{n}+\lambda I)^{\frac{3}{2}}\|_{L^2 \to L^2} + \lambda \|(\hat{\Lambda}_{n}+\lambda I)^{-\frac{1}{2}}g_{\lambda}(\hat{\Lambda}_{n})(\hat{\Lambda}_{n}+\lambda I)^{\frac{3}{2}}\|_{L^2 \to L^2}\right]\\
         & \times  \|\Lambda^{\frac{1}{2}}(\Lambda+\lambda I)^{-\frac{1}{2}}\|_{L^2 \to L^2} \|(\Lambda+\lambda I)^{\frac{1}{2}}(\hat{\Lambda}_{n}+\lambda I)^{-\frac{1}{2}}\|_{L^2 \to L^2} \|(\hat{\Lambda}_{n}+\lambda I)^{-\frac{3}{2}}T^{\frac{1}{2}}\hat{C}_{n}\beta^*\|_{L^2}.
        \end{split} 
\end{equation*}
\noindent
From Lemma \ref{Covariance estimation} and regularization properties $(\ref{reg1}), (\ref{reg2})$ and $(\ref{qualification})$, we get
\begin{equation*}
\begin{split}
     \|C^{\frac{1}{2}}T^{\frac{1}{2}}(g_{\lambda}(\hat{\Lambda}_{n})-(\hat{\Lambda}_{n}+\lambda I)^{-1})T^{\frac{1}{2}}\hat{C}_{n}\beta^*\|_{L^2}
         \lesssim_{p} &  \lambda \|(\hat{\Lambda}_{n}+\lambda I)^{-\frac{3}{2}}T^{\frac{1}{2}}\hat{C}_{n}\beta^*\|_{L^2}.
         \end{split}
\end{equation*}
Using similar steps of $(\ref{3/2powerbound})$, we get
\begin{equation*}
    \|C^{\frac{1}{2}}T^{\frac{1}{2}}(g_{\lambda}(\hat{\Lambda}_{n})-(\hat{\Lambda}_{n}+\lambda I)^{-1})T^{\frac{1}{2}}\hat{C}_{n}\beta^*\|_{L^2} \lesssim_{p} \textcolor{black}{\frac{\lambda^{-\frac{1-\frac{c}{2}-t\alpha}{(t+c)}}}{\sqrt{n}} } + \lambda^{\frac{\alpha t +\frac{c}{2}}{t+c}}.
\end{equation*}
\textit{Bounding Term-6:}
\begin{equation*}
    \begin{split}
         & \|C^{\frac{1}{2}}(T^{\frac{1}{2}}(\Lambda+\lambda I)^{-1}T^{\frac{1}{2}}C\beta^*-\beta^*)\|_{L^2} 
         =  \|C^{\frac{1}{2}}(T^{\frac{1}{2}}(\Lambda+\lambda I)^{-1}T^{\frac{1}{2}}CT^{\alpha}h-T^{\alpha}h)\|_{L^2}\\
         = & \left[\sum_{i}(\frac{\mu_{i}^{1+\alpha}\xi_{i}^{\frac{3}{2}}}{\mu_{i}\xi_{i}+\lambda}-\mu_{i}^{\alpha}\xi_{i}^{\frac{1}{2}})^2 \langle\phi_{i},h\rangle^2\right]^{\frac{1}{2}} = \left[\sum_{i}(\frac{\lambda \mu_{i}^{\alpha}\xi_{i}^{\frac{1}{2}}}{\mu_{i}\xi_{i}+\lambda})^2 \langle\phi_{i},h\rangle^2\right]^{\frac{1}{2}}\\
         \leq & \lambda \sup_{i} \left[\frac{\mu_{i}^{\alpha}\xi_{i}^{\frac{1}{2}}}{\mu_{i}\xi_{i}+\lambda}\right] \|h\|_{L^2} \lesssim \lambda \sup_{i} \left[\frac{i^{-\alpha t -\frac{c}{2}}}{i^{-(t+c)}+\lambda}\right]\|h\|_{L^2} \lesssim \lambda^{\frac{\alpha t +\frac{c}{2}}{t+c}}.
    \end{split}
\end{equation*}
\textit{Bounding Term-7:}
\begin{equation*}
    \begin{split}
         & \|C^{\frac{1}{2}}T^{\frac{1}{2}}((\hat{\Lambda}_{n}+\lambda I)^{-1}T^{\frac{1}{2}}\hat{C}_{n}\beta^*-(\Lambda+\lambda I)^{-1}T^{\frac{1}{2}}C\beta^*)\|_{L^2} \\
         = & \|\Lambda^{\frac{1}{2}}((\hat{\Lambda}_{n}+\lambda I)^{-1}T^{\frac{1}{2}}\hat{C}_{n}\beta^*-(\Lambda+\lambda I)^{-1}T^{\frac{1}{2}}C\beta^*)\|_{L^2}.
    \end{split}
\end{equation*}
Using the similar steps of $(\ref{2b})$, we get
\begin{equation*}
   \|C^{\frac{1}{2}}T^{\frac{1}{2}}((\hat{\Lambda}_{n}+\lambda I)^{-1}T^{\frac{1}{2}}\hat{C}_{n}\beta^*-(\Lambda+\lambda I)^{-1}T^{\frac{1}{2}}C\beta^*)\|_{L^2} \lesssim_{p} \textcolor{black}{\frac{\lambda^{-\frac{1-\frac{c}{2}-t\alpha}{(t+c)}}}{\sqrt{n}}.}
\end{equation*}
Combining all the bounds, we get
\begin{equation*}
        \|C^{\frac{1}{2}}(T^{\frac{1}{2}}g_{\lambda}(\hat{\Lambda}_{n})T^{\frac{1}{2}}\hat{C}_{n}\beta^*-\beta^*)\|_{L^2} \lesssim_{p} 
            \textcolor{black}{ \frac{\lambda^{-\frac{1-\frac{c}{2}-t\alpha}{(t+c)}}}{\sqrt{n}} }+\lambda^{\frac{\alpha t+ \frac{c}{2}}{t+c}}.
\end{equation*}
\textbf{Case-2:} Suppose $\alpha \geq \frac{1}{2}$. Let $\nu$ be the qualification of the regularization family.
\begin{equation*}
    \begin{split}
        &\|C^{\frac{1}{2}}(T^{\frac{1}{2}}g_{\lambda}(\hat{\Lambda}_{n})T^{\frac{1}{2}}\hat{C}_{n}\beta^*-\beta^*)\|_{L^2} \\
        = & \|C^{\frac{1}{2}}(T^{\frac{1}{2}}g_{\lambda}(\hat{\Lambda}_{n})T^{\frac{1}{2}}\hat{C}_{n}T^{\frac{1}{2}}T^{\alpha-\frac{1}{2}}h-T^{\frac{1}{2}}T^{\alpha-\frac{1}{2}}h)\|_{L^2} \\
         = &  \|C^{\frac{1}{2}}T^{\frac{1}{2}}r_{\lambda}(\hat{\Lambda}_{n})T^{\alpha-\frac{1}{2}}h\|_{L^2}
         =  \|\Lambda^{\frac{1}{2}}r_{\lambda}(\hat{\Lambda}_{n})T^{\alpha-\frac{1}{2}}h\|_{L^2}.
    \end{split}
\end{equation*}
Following the steps of \textbf{Case-2} in Theorem \ref{comm.estimation}, we get 
\begin{equation*}
    \|C^{\frac{1}{2}}(\hat{\beta}-\beta^*)\|_{L^2}^2 \lesssim_{p} \frac{\lambda^{-\frac{1}{t+c}}}{n} + \lambda^{\frac{2 r t+c}{t+c}},
\end{equation*}
where $r = \min \{\alpha , \nu +\frac{1}{2}+\nu \frac{c}{t}\}$.\\
\indent Combining all the bounds, we get
\begin{equation*}
    \|C^{\frac{1}{2}}(\hat{\beta}-\beta^*)\|_{L^2}^2 \lesssim_{p}
    \frac{\lambda^{-\frac{1}{t+c}}}{n}+ \lambda^{\frac{2 r t+ c}{t+c}} ,
\end{equation*}
and the result follows by choosing 
$\lambda = n^{-\frac{t+c}{1+c+ 2 r t}}$.

\subsection{Proof of Theorem \ref{predictionNC}}\label{pcomm-pred2}
    Consider
    \begin{equation*}
        \begin{split}
            \|C^{\frac{1}{2}}(\hat{\beta}-\beta^*)\|_{L^2} = & \|C^{\frac{1}{2}}(J\hat{\beta}-\beta^*)\|_{L^2}
            =  \|C^{\frac{1}{2}}(Jg_{\lambda}(J^*\hat{C}_{n}J)J^*\hat{R}-\beta^*)\|_{L^2}.
        \end{split}
    \end{equation*}
Following the steps of Theorem \ref{predictionbasic}, we get
\begin{equation*}
    \begin{split}
        \|C^{\frac{1}{2}}(\hat{\beta}-\beta^*)\|_{L^2} \lesssim_{p} & \|(\Lambda + \lambda I)^{-\frac{1}{2}}(T^{\frac{1}{2}}\hat{R}-T^{\frac{1}{2}}\hat{C}_{n}\beta^*)\|_{L^2} \\
       & \quad \quad + \|C^{\frac{1}{2}}(T^{\frac{1}{2}}g_{\lambda}(\hat{\Lambda}_{n})T^{\frac{1}{2}}\hat{C}_{n}\beta^*-\beta^*)\|_{L^2}.
    \end{split}
\end{equation*}
Using Lemma \ref{Empirical bound} yields
\begin{equation*}
    \|C^{\frac{1}{2}}(\hat{\beta}-\beta^*)\|_{L^2} \lesssim_{p} \frac{\lambda^{-\frac{1}{2b}}}{\sqrt{n}}+  \|C^{\frac{1}{2}}(T^{\frac{1}{2}}g_{\lambda}(\hat{\Lambda}_{n})T^{\frac{1}{2}}\hat{C}_{n}\beta^*-\beta^*)\|_{L^2}.
\end{equation*} 
Now consider
\begin{equation*}
    \begin{split}
        \|C^{\frac{1}{2}}(T^{\frac{1}{2}}g_{\lambda}(\hat{\Lambda}_{n})T^{\frac{1}{2}}\hat{C}_{n}\beta^*-\beta^*)\|_{L^2} = \|\Lambda^{\frac{1}{2}}r_{\lambda}(\hat{\Lambda}_{n})\Lambda^sh \|_{L^2}.
    \end{split}
\end{equation*}
\textcolor{black}{
Let $\mu = \nu -\frac{1}{2}$. Then
\begin{equation*}
\begin{split}
    \|\Lambda^{\frac{1}{2}}r_{\lambda}(\hat{\Lambda}_{n})\Lambda^sh \|_{L^2} \leq \|(\hat{\Lambda}+\lambda I)^{\frac{1}{2}}r_{\lambda}(\hat{\Lambda}_{n})(\hat{\Lambda}+\lambda I)^{\mu} (\hat{\Lambda}+\lambda I)^{-\mu}\Lambda^sh\|_{L^2}. 
\end{split}
\end{equation*}
Following steps similar to those of Theorem \ref{NCestimation} with $\mu$ in place of $\nu$, we obtain
}
\begin{equation*}
    \begin{split}
        \|C^{\frac{1}{2}}(T^{\frac{1}{2}}g_{\lambda}(\hat{\Lambda}_{n})T^{\frac{1}{2}}\hat{C}_{n}\beta^*-\beta^*)\|_{L^2} \lesssim_{p} & \textcolor{black}{\frac{\lambda^{-\frac{1}{b} + \frac{1}{2}}}{\sqrt{n}}} + \lambda^{\nu} \|(\Lambda + \lambda I)^{-(\nu-\frac{1}{2})}\Lambda^s h\|_{L^2}\\
        \lesssim & \textcolor{black}{\frac{\lambda^{-\frac{1}{b}+ \frac{1}{2}}}{\sqrt{n}}} +\lambda^{r+\frac{1}{2}},
    \end{split}
\end{equation*}
where $r = \min\{s, \nu -\frac{1}{2}\}$. Combining all the bounds, we get
\begin{equation*}
    \|C^{\frac{1}{2}}(\hat{\beta}-\beta^*)\|_{L^2}^{2} \lesssim_{p} \frac{\lambda^{-\frac{1}{b}}}{n} +\lambda^{2r+1}
\end{equation*}
and the result follows by choosing 
$\lambda = n^{-\frac{b}{1+b+2rb}}$.

\subsection{Proof of Theorem \ref{lowercomm}}\label{pcomm-low}
To establish the lower bound, without loss of generality, we assume $\epsilon \sim N(0, \sigma^2)$ since the lower bound for a special case guarantees the lower bound for the general case. Let $M$ be the smallest integer greater than $ c_{0} n^{\frac{1}{q}}$ for some constant $c_{0}>0$ and $q>0$, which will be specified later. For $\theta \in 
    \{0,1\}^M$,  
    define $$f_{\theta}:= \sum_{k=1}^{M}\theta_kM^{-\frac{1}{2}}T^{\alpha}\phi_{k+M}.$$ Let  $P_{\theta}$ denote the $n$-fold product of joint probability distributions for $\{(X_{i}, Y_{i}) : 1 \leq i \leq n\}$ with $\beta^* = f_{\theta}$. Then the Kullback-Leibler divergence between $P_{\theta}$ and $P_{\theta^{'}}$ is $$\mathcal{K}(P_{\theta}, P_{\theta^{'}})  = \frac{n}{2 \sigma^2}\|C^{\frac{1}{2}}(f_{\theta}-f_{\theta^{'}})\|_{L_2}^2$$  for $\theta, \theta^{'}\in\{0,1\}^M.$ \textcolor{black}{We refer the reader to Appendix~\ref{subsec:kl} for the derivation of the above formula.} Using the representation of $f_\theta$ and $f_{\theta^{'}}$, we obtain
    
\begin{equation*}
    \begin{split}        \mathcal{K}(P_{\theta}, P_{\theta^{'}}) 
        & = \frac{n}{2 \sigma^2}\left\|\sum_{k=1}^{M}(\theta_k-\theta_k^{'})M^{-\frac{1}{2}}C^{\frac{1}{2}}T^{\alpha}\phi_{k+M}\right\|_{L_2}^2\\
        & = \frac{n}{2 \sigma^2} \sum_{k=1}^{M}(\theta_k-\theta_k^{'})^2M^{-1}\mu_{k+M}^{2 \alpha}\xi_{k+M} 
         \leq \frac{n}{\sigma^2} M^{-1} \mu_M^{2 \alpha} \xi_M H(\theta,\theta^{'})\\
        &  \leq \frac{d}{\sigma^2} n M^{-1} M^{-2 \alpha t} M^{-c}M
         \leq \frac{d}{\sigma^2 c_0^{q}} M^{q-2 \alpha t-c}\\
         & \textcolor{black}{\leq \frac{d}{\sigma^2 c_{0}^q}M \leq \frac{8d}{\sigma^2 c_{0}^q \log 2} \log N   ,}
    \end{split}
\end{equation*}
\textcolor{black}{where we have used that $c_0n^{\frac{1}{q}}\le M$, $q = 1+2 \alpha t+c$ and $N \geq 2^{\frac{M}{8}}$.}
\textcolor{black}{Let $c_{0} = c'_{0} u^{-\frac{1}{q}} $, where $c'_{0}$ is a large constant such that $\frac{8 d}{\sigma^2 ((c'_{0})^q \log 2)} \leq 1$ then $\mathcal{K}(P_{\theta},P_{\theta^{'}}) \leq u \log N$.}

Next, for $\alpha \geq \frac{1}{2}$, by considering the error term in the RKHS norm, we have
\begin{equation*}
    \begin{split}
        &\|f_{\theta}-f_{\theta^{'}}\|_{\mathcal{H}}^2  = \left\|\sum_{k= 1}^{M}(\theta_k -\theta^{'}_k)M^{-\frac{1}{2}}T^{\alpha}\phi_{k+M}\right\|_{\mathcal{H}}^2 
         = \frac{1}{M} \left\|\sum_{k= 1}^{M} (\theta_k -\theta^{'}_k) T^{\alpha -\frac{1}{2}}\phi_{k+M}\right\|_{L_2}^2 \\
    & = \frac{1}{M} \sum_{k= 1}^{M} (\theta_k -\theta^{'}_k)^2 \mu_{k+M}^{2(\alpha -\frac{1}{2})} 
         \gtrsim M^{-1} H(\theta , \theta^{'}) \mu_{2M}^{2(\alpha -\frac{1}{2})} 
    \stackrel{(*)}{\gtrsim} M^{-1} M M^{-2t(\alpha -\frac{1}{2})} \\
        & = M^{-2t(\alpha -\frac{1}{2})}
         \stackrel{(\dagger)}{\gtrsim} 
          u^{\frac{t(2\alpha -1)}{q}}n^{-\frac{t(2\alpha -1)}{q}},
    \end{split}
\end{equation*}
where we used Lemma~\ref{VGbound} in $(*)$ and the fact that $\frac{M}{2}<c_0 n^{\frac{1}{q}}$ in $(\dagger)$.
Hence,
\begin{equation}
    \label{rkhses}
    \|f_{\theta}-f_{\theta^{'}}\|_{\mathcal{H}} \geq 2\tilde{c_{1}} u^{\frac{t(\alpha -\frac{1}{2})}{1+c+2 \alpha t}} n^{-\frac{t(\alpha -\frac{1}{2})}{1+c+2 \alpha t}},
\end{equation}
for some constant $\tilde{c_{1}}$.\\
Now applying Theorem \ref{tsyback theorem}, we get
\textcolor{black}{$$\inf_{\hat{\beta}} \sup_{\beta^* \in \mathcal{R}(T^{\alpha})} \mathbb{P}\left\{\|\beta^*-\hat{\beta}\|_{\mathcal{H}} \geq \tilde{c_{1}} u^{\frac{t(\alpha -\frac{1}{2})}{1+c+2 \alpha t}} n^{-\frac{t(\alpha -\frac{1}{2})}{1+c+2 \alpha t}}\right\} \geq \frac{\sqrt{N}}{1+ \sqrt{N}}\left(1-2u-\sqrt{\frac{2u}{\log N}}\right).$$
Then for $a = \tilde{c_{1}} u^{\frac{t(\alpha -\frac{1}{2})}{1+c+2 \alpha t}}$ and $\zeta = \frac{t(\alpha -\frac{1}{2})}{1+c+2 \alpha t}$, we have
\begin{equation*}
    \inf_{\hat{\beta}} \sup_{\beta^* \in \mathcal{R}(T^{\alpha})} \mathbb{P}\left\{\|\beta^*-\hat{\beta}\|_{\mathcal{H}} \geq a n^{-\frac{t(\alpha -\frac{1}{2})}{1+c+2 \alpha t}}\right\} \geq \frac{\sqrt{N}}{1+ \sqrt{N}}\left(1-2({a}/{\tilde{c_1}})^{\frac{1}{\zeta}}-\sqrt{\frac{2({a}/{\tilde{c_1}})^{\frac{1}{\zeta}}}{\log N}}\right).
\end{equation*}
Using the fact that $n \to \infty$ implies $N \to \infty$ yields the desired result.\\}

Next, we consider the estimation error term in the $L^2$ norm as
\begin{equation*}
    \begin{split}
       & \|f_{\theta}-f_{\theta^{'}}\|_{L_2}^2  = \left\|\sum_{k=1}^{M}(\theta_k-\theta_k^{'})M^{-\frac{1}{2}}T^{\alpha}\phi_{k+M}\right\|_{L_2}^2
         = M^{-1}\sum_{k=1}^{M}(\theta_k-\theta_k^{'})^2\mu_{k+M}^{2 \alpha} \\
        & \gtrsim M^{-1} H(\theta, \theta^{'}) \mu_{2M}^{2 \alpha}
        \gtrsim M^{-1} M M^{-2 \alpha t}
        = M^{-2 \alpha t} \gtrsim u^{\frac{2 \alpha t}{q}}n^{-\frac{2 \alpha t}{q}}.
    \end{split}
\end{equation*}
Therefore, we get
\begin{equation}
    \label{l2es}
    \|f_{\theta}-f_{\theta^{'}}\|_{L_2} \geq 2 \tilde{c_2} u^{\frac{ \alpha t}{1+c+2 \alpha t}} n^{-\frac{\alpha t}{1+c+2\alpha t}}.
\end{equation}
Now applying Theorem \ref{tsyback theorem}, we get
\textcolor{black}{$$\inf_{\hat{\beta}} \sup_{\beta^* \in \mathcal{R}(T^{\alpha})} \mathbb{P}\left\{\|\beta^*-\hat{\beta}\|_{L^2} \geq \tilde{c_2}  u^{\frac{\alpha t}{1+c+2\alpha t}} n^{-\frac{\alpha t}{1+c+2\alpha t}}\right\} \geq \frac{\sqrt{N}}{1+ \sqrt{N}}\left(1-2u-\sqrt{\frac{2u}{\log N}}\right),$$
resulting in the desired result for $a = \tilde{c_2} u^{\frac{\alpha t}{1+c+2\alpha t}}$.\\}

Next, we consider the error term for the 
 prediction error:
\begin{equation*}
    \begin{split}
        &\|C^{\frac{1}{2}}(f_{\theta}-f_{\theta^{'}})\|_{L^2}^2  = \left\|\sum_{k=1}^{M}(\theta_k-\theta_k^{'})M^{-\frac{1}{2}}C^{\frac{1}{2}}T^{\alpha}\phi_{k+M}\right\|_{L_2}^2\\
        & = M^{-1}\sum_{k=1}^{M}(\theta_k-\theta_k^{'})^2\mu_{k+M}^{2 \alpha} \xi_{k+M}
         \gtrsim M^{-1}(2M)^{-2 \alpha t}(2M)^{-c}H(\theta,\theta^{'})
         \gtrsim M^{-2 \alpha t-c}\\
&     \geq u^{\frac{2 \alpha t+c}{q}}n^{-\frac{2 \alpha t+c}{q}}.
    \end{split}
\end{equation*}
Hence,
\begin{equation}
    \label{pred}
    \|C^{\frac{1}{2}}(f_{\theta}-f_{\theta^{'}})\|_{L^2}^2 \geq 2 \tilde{c_3} u^{\frac{c+2 \alpha t}{1+c+2 \alpha t}} n^{-\frac{c+2 \alpha t}{1+c+2 \alpha t}}.
\end{equation}
Applying Theorem \ref{tsyback theorem}, we get
\textcolor{black}{$$\inf_{\hat{\beta}} \sup_{\beta^* \in \mathcal{R}(T^{\alpha})} \mathbb{P}\left\{\mathbb{E}\langle\hat{\beta}-\beta^*, X\rangle_{L^2}^2 \geq   \tilde{c_3} u^{\frac{c+ 2 \alpha t}{1+c+2\alpha t}} n^{-\frac{c+2 \alpha t}{1+c+2 \alpha t}}\right\} \geq \frac{\sqrt{N}}{1+ \sqrt{N}}\left(1-2u-\sqrt{\frac{2u}{\log N}}\right),$$
which for $a= \tilde{c_3} u^{\frac{c+ 2 \alpha t}{1+c+2\alpha t}}$, provides the desired result.
}

\subsection{Proof of Theorem \ref{lowerNC}}\label{pcomm-lower2}
Let $M$ be the smallest integer greater than $ c_{0} n^{\frac{1}{q}}$ for some constant $c_{0}>0$ and $q > 1+b$. For $\theta \in \{0,1\}^M $, we define $$f_{\theta}:= \sum_{k=1}^{M}\theta_kM^{-\frac{1}{2}}T^{\frac{1}{2}}(T^{\frac{1}{2}}CT^{\frac{1}{2}})^{s}\phi_{k+M}.$$
Similar to the proof of Theorem \ref{lowercomm}, consider the Kullback-Leibler divergence between joint probability distributions $P_{\theta}$ and $P_{\theta^{'}}$:
\begin{equation*}
    \begin{split}
        \mathcal{K}(P_{\theta},P_{\theta^{'}}) & = \frac{n}{2 \sigma^2}\|C^{\frac{1}{2}}(f_{\theta}-f_{\theta^{'}})\|_{L_2}^2  = \frac{n}{2 \sigma^2}\left\|\sum_{k=1}^{M}(\theta_k-\theta_k^{'})M^{-\frac{1}{2}}C^{\frac{1}{2}}T^{\frac{1}{2}}\Lambda^{s}\phi_{k+M}\right\|_{L_2}^2\\
        & = \frac{n M^{-1}}{2 \sigma^2}\sum_{k=1}^{M}(\theta_k-\theta_k^{'})^2 \tau_{k+M}^{1+2s} \leq \frac{d'}{\sigma^2} n M^{-b(1+2s)} = \frac{d'}{\sigma^2} n M^{-b(1+2s)}\\
        & \leq \frac{d'}{\sigma^2 c_{0}^q} M^{-b(1+2s)+q}.
    \end{split}
\end{equation*}
In the last step, we used that $c_{0}n^{\frac{1}{q}} \leq M$.
To satisfy the condition on Kullback divergence, we take $q= 1 + b(1+2s)$ and $c_{0} = c_{0}^{'} u^{-\frac{1}{q}} $, where $c_{0}^{'}$ is a large enough constant.\\
\noindent
Next, we consider the error term in the RKHS norm:
\begin{equation*}
    \begin{split}
        \|f_{\theta}-f_{\theta^{'}}\|_{\mathcal{H}}^2 & =  M^{-1}\sum_{k=1}^{M}(\theta_k-\theta_k^{'})^2 \tau_{k+M}^{2s} \gtrsim M^{-1}\tau_{2M}^{2s}H(\theta, \theta^{'})\\
        & \stackrel{(*)}{\gtrsim} M^{-1} M^{-b(2s)}M = M^{-b(2s)} \stackrel{(\dagger)}{\geq} 2\tilde{c_4} u^{\frac{2sb}{q}} n^{-\frac{2sb}{q}},
    \end{split}
\end{equation*}
where we used Lemma~\ref{VGbound} in $(*)$ and the fact that $\frac{M}{2}<c_0 n^{\frac{1}{q}}$ in $(\dagger)$. Applying Theorem \ref{tsyback theorem}, yields
\textcolor{black}{$$\inf_{\hat{\beta}} \sup_{\beta^* \in \mathcal{R}(T^{\frac{1}{2}}(\Lambda)^s)} \mathbb{P}\left\{\|\hat{\beta}-\beta^*\|_{\mathcal{H}} \geq  \tilde{c_4} u^{\frac{sb}{1+b+2sb}} n^{-\frac{bs}{1+b+2 sb}}\right\} \geq \frac{\sqrt{N}}{1+ \sqrt{N}}\left(1-2u-\sqrt{\frac{2u}{\log N}}\right).$$
Let $a= \tilde{c_4} u^{\frac{sb}{1+b+2sb}}$, then we have}
\begin{equation*}
        \lim_{a \to 0} \lim_{n \to \infty} \inf_{\hat{\beta}} \sup_{\beta^* \in \mathcal{R}(T^{\frac{1}{2}}(\Lambda)^s)} \mathbb{P}\left\{\|\hat{\beta}-\beta^*\|_{\mathcal{H}} \geq a n^{-\frac{sb}{1+b+2 sb}}\right\} =1.
    \end{equation*}

The error term for prediction is given by
\begin{equation*}
    \begin{split}
        \|C^{\frac{1}{2}}(f_{\theta}-f_{\theta^{'}})\|_{L^2}^2 & \gtrsim M^{-1}H(\theta,\theta^{'}) \tau_{2M}^{1+2 s} \\
        & \gtrsim M^{-b(1+2 s)} \geq 2 \tilde{c_5} u^{\frac{b(1+2 s)}{q}} n^{-\frac{b(1+2 s)}{q}},
    \end{split}
\end{equation*}
where applying Theorem \ref{tsyback theorem} yields
\textcolor{black}{\begin{align*}&\inf_{\hat{\beta}} \sup_{\beta^* \in \mathcal{R}(T^{\frac{1}{2}}(\Lambda)^s)} \mathbb{P}\left\{\mathbb{E}\langle\hat{\beta}-\beta^*, X\rangle_{L^2}^2 \geq \tilde{c_5} u^{\frac{b(2s+1)}{1+b+2sb}} n^{-\frac{b(2s+1)}{1+b+2 sb}}\right\} \\
&\qquad\qquad\geq \frac{\sqrt{N}}{1+ \sqrt{N}}\left(1-2u-\sqrt{\frac{2u}{\log N}}\right).\end{align*}
Let $a= \tilde{c_5} u^{\frac{b(2s+1)}{1+b+2sb}}$, then we have}
\begin{equation*}
    \lim_{a \to 0} \lim_{n \to \infty} \inf_{\hat{\beta}} \sup_{\beta^* \in \mathcal{R}(T^{\frac{1}{2}}(\Lambda)^s)} \mathbb{P}\left\{\mathbb{E}\langle\hat{\beta}-\beta^*, X\rangle_{L^2}^2 \geq a n^{-\frac{b(2s+1)}{1+b+2 sb}}\right\} =1.
\end{equation*}

\textcolor{black}{\section*{Acknowledgments}
The authors profusely thank the referees for their valuable comments and suggestions, which helped to significantly improve the manuscript. We are particularly grateful to the reviewer whose insightful feedback resulted in Section~\ref{Sec:similar}. We would also like to thank the reviewer who pointed out an error in the proof of Lemmas~\ref{cmdifference} and \ref{ncmdifference}. BKS is partially supported by the National Science Foundation (NSF) CAREER award DMS-1945396.}
\bibliographystyle{acm}
\bibliography{123}
\appendix
\numberwithin{equation}{section}
\section{Supplementary Results}
\label{supplements}
In this appendix, we present technical results that are used to prove the main results of the paper.

\begin{lemma}
\label{J T}
    For any bounded linear positive operator $\mathcal{A}:L^2(S)\to L^2(S)$, we have
\begin{equation*}
Jg_{\lambda}(J^*\mathcal{A}J)J^* = T^{1/2}g_{\lambda}(T^{\frac{1}{2}}\mathcal{A}T^{\frac{1}{2}})T^{1/2}.
\end{equation*}
\end{lemma}
\begin{proof}
Let $(m_i, \phi_i )_i$ be the eigensystem of the operator $J^*\mathcal{A}J:\mathcal{H} \to \mathcal{H}$. It is easy to verify that $(m_i,\frac{T^{1/2}\mathcal{A}J\phi_i}{m_i} )_i$ is the eigensystem of $T^{1/2}\mathcal{A}T^{1/2}$ since $$T^{1/2}\mathcal{A}T^{1/2} T^{1/2}\mathcal{A}J\phi_i=T^{1/2}\mathcal{A} T \mathcal{A}J\phi_i=T^{1/2}\mathcal{A} JJ^*\mathcal{A} J\phi_i=m_i T^{1/2}\mathcal{A} J\phi_i$$
and for any $i, j \in \mathbb{N}$,
    \begin{equation*}
			    \begin{split}
				&\left\langle \frac{T^{1/2}\mathcal{A}J\phi_i}{m_i}, \frac{T^{1/2}\mathcal{A}J\phi_j}{m_j} \right\rangle_{L^2}  = \left\langle \frac{\mathcal{A}J\phi_i}{m_i}, \frac{T\mathcal{A}J\phi_j}{m_j} \right\rangle_{L^2} 
				 = \left\langle \frac{\mathcal{A}J\phi_i}{m_i}, \frac{JJ^{*}\mathcal{A}J\phi_j}{m_j} \right\rangle_{L^2} \\
				& = \left\langle \frac{J^{*}\mathcal{A}J\phi_i}{m_i}, \frac{J^{*}\mathcal{A}J\phi_j}{m_j} \right\rangle_{\mathcal{H}} 
				 = \left\langle \frac{m_i\phi_i}{m_i}, \frac{m_j\phi_j}{m_j} \right\rangle_{\mathcal{H}} 
                 = \langle \phi_i, \phi_j \rangle_{\mathcal{H}} 
				 = \delta_{ij}.
                \end{split}
	\end{equation*}           
Therefore, using the spectral representation of $J^*\mathcal{A}J$, it follows that for all $f \in L^2(S)$,
  \begin{equation*}
      \begin{split}
&T^{1/2}g_{\lambda}(T^{\frac{1}{2}}\mathcal{A}T^{\frac{1}{2}})T^{1/2} f  = \sum_{i} g_{\lambda}(m_i) \left\langle T^{1/2}f , \frac{T^{1/2}\mathcal{A}J\phi_i}{m_i} \right\rangle_{L^2} T^{1/2} \left(\frac{T^{1/2}\mathcal{A}J\phi_i}{m_i}\right)\\
		& = \sum_{i} g_{\lambda}(m_i) \left\langle f , \frac{T\mathcal{A}J\phi_i}{m_i} \right\rangle_{L^2}  \frac{T\mathcal{A}J\phi_i}{m_i}
		 = \sum_{i} g_{\lambda}(m_i) \left\langle f , \frac{JJ^{*}\mathcal{A}J\phi_i}{m_i} \right\rangle_{L^2}  \frac{JJ^{*}\mathcal{A}J\phi_i}{m_i} \\
		& = \sum_{i} g_{\lambda}(m_i) \langle f , J\phi_i \rangle_{L^2}  J\phi_i 
		 = \sum_{i} g_{\lambda}(m_i) \langle J^{*}f , \phi_i \rangle_{\mathcal{H}}  J\phi_i 
		 = Jg_{\lambda}(J^*\mathcal{A}J)J^*f
      \end{split}
  \end{equation*}
  and the result follows.
\end{proof}

\begin{lemma}
\label{Empirical bound}
For any $\delta >0$, with at least probability $1-\delta$, we have that
\begin{equation*}
\|(\Lambda+\lambda I)^{-\frac{1}{2}}T^{\frac{1}{2}}(\hat{R}-\hat{C}_{n}\beta^*)\|_{L^2} \leq \sqrt{\frac{ \sigma^2 \mathcal{N}(\lambda)}{n \delta}}.
\end{equation*}
\end{lemma}
\begin{proof}
For $i= 1,2, \cdots, n$, define $Z_{i}:= (\Lambda+\lambda I)^{-\frac{1}{2}}T^{\frac{1}{2}}[Y_{i}X_{i}-(X_{i}\otimes X_{i})\beta^*]$. Since the slope function $\beta^*$ satisfies the operator equation $C\beta^*= \mathbb{E}[YX]$, the mean of the random variable $Z_{i}$ is zero, i.e.,
$$\mathbb{E}[Z_{i}]=(\Lambda+\lambda I)^{-\frac{1}{2}}T^{\frac{1}{2}}[\mathbb{E}[YX]-C\beta^*] =0.$$
By Markov's inequality, for any $t>0$
$$ \mathbb{P}\left(\left\|\frac{1}{n}\sum_{i=1}^{n}Z_{i}\right\|_{L^2} \geq t\right)\leq \frac{\mathbb{E}\|\frac{1}{n}\sum_{i=1}^{n}Z_{i}\|_{L^2}^{2}}{t^2}.$$
Note that 
$$\mathbb{E}\left\|\frac{1}{n}\sum_{i=1}^{n}Z_{i}\right\|_{L^2}^{2} = \frac{1}{n^2}\sum_{i,j=1}^{n}\mathbb{E}\langle Z_{i},Z_{j}\rangle_{L^2}= \frac{1}{n^2}\sum_{i\neq j}^{n}\mathbb{E}\langle Z_{i},Z_{j}\rangle_{L^2}+ \frac{1}{n^2}\sum_{i=1}^{n}\mathbb{E}\|Z_{i}\|_{L^2}^{2}= \frac{\mathbb{E}\|Z_{1}\|_{L^2}^{2}}{n} $$
and by taking $t = \sqrt{\frac{\mathbb{E}\|Z_{1}\|_{L^2}^{2}}{n\delta}}$, with at least probability $1-\delta$, we have 
\begin{equation}
\label{chebyshevbound}
\|(\Lambda+\lambda I)^{-\frac{1}{2}}T^{\frac{1}{2}}(\hat{R}-\hat{C}_{n}\beta^*)\|_{L^2} \leq \sqrt{\frac{\mathbb{E}\|(\Lambda+\lambda I)^{-\frac{1}{2}}T^{\frac{1}{2}}(YX-(X \otimes X)\beta^*)\|_{L^2}^{2}}{n\delta}}.
\end{equation}
Consider
\begin{equation*}
    \begin{split}
        & \mathbb{E}[\|(\Lambda+\lambda I)^{-\frac{1}{2}}T^{\frac{1}{2}}(YX-(X \otimes X)\beta^*)\|_{L^2}^{2}]\\
        = & \mathbb{E}[\|(\Lambda+\lambda I)^{-\frac{1}{2}}T^{\frac{1}{2}}(Y-\langle X, \beta^* \rangle_{L^2})X\|_{L^2}^{2}]\\
        = & \mathbb{E}[(Y-\langle X, \beta^* \rangle_{L^2})^2\|(\Lambda+\lambda I)^{-\frac{1}{2}}T^{\frac{1}{2}}X\|_{L^2}^{2}]\\
        = & \mathbb{E}[\epsilon^{2} \langle (\Lambda+\lambda I)^{-\frac{1}{2}}T^{\frac{1}{2}}X, (\Lambda+\lambda I)^{-\frac{1}{2}}T^{\frac{1}{2}}X\rangle_{L^2}]\\
        = & \mathbb{E}[\epsilon^{2}\text{trace}((\Lambda+\lambda I)^{-1}T^{\frac{1}{2}}(X \otimes X)T^{\frac{1}{2}})] \\
        = & \mathbb{E}[\epsilon^{2} ] \text{trace}((\Lambda+\lambda I)^{-1}T^{\frac{1}{2}}CT^{\frac{1}{2}}) \\
        = & \sigma^2 \text{trace}((\Lambda+\lambda I)^{-1}\Lambda) = \sigma^2 \mathcal{N}(\lambda).
    \end{split}
\end{equation*}
The result therefore follows from \eqref{chebyshevbound}.
\end{proof}

\begin{lemma}
  \label{reducing}
  For any $n \geq 1$, we have
    \begin{equation*}
    \begin{split}
        (\hat{\Lambda}_{n}+\lambda I )^{-n} - (\Lambda+\lambda I )^{-n} =  &(\hat{\Lambda}_{n}+\lambda I )^{-(n-1)}[(\hat{\Lambda}_{n}+\lambda I )^{-1} - (\Lambda+\lambda I )^{-1}]\\
        & + \sum_{i=1}^{n-1}(\hat{\Lambda}_{n}+\lambda I )^{-i}(\Lambda-\hat{\Lambda}_{n})(\Lambda+\lambda I )^{-(n+1-i)}.
    \end{split}
\end{equation*}
\end{lemma}  
\begin{proof}
Define
\begin{equation*}
   \Lambda_{\lambda}:= (\Lambda+\lambda I)^{-1} \quad \text{ and } \quad \hat{\Lambda}_{\lambda}:= (\hat{\Lambda}_{n}+\lambda I)^{-1}.
\end{equation*}
Consider
\begin{equation*}
    \begin{split}
        \hat{\Lambda}_{\lambda}^{n} -\Lambda_{\lambda}^{n}  = & \hat{\Lambda}_{\lambda} (\hat{\Lambda}_{\lambda}^{n-1}-\hat{\Lambda}\Lambda_{\lambda}^{n} - \lambda\Lambda_{\lambda}^{n} ) 
        =  \hat{\Lambda}_{\lambda} (\hat{\Lambda}_{\lambda}^{n-1}-\hat{\Lambda}\Lambda_{\lambda}^{n} - \Lambda_{\lambda}^{n-1} + \Lambda\Lambda_{\lambda}^{n} ).
    \end{split}
\end{equation*}
In the last step, we have used the fact that $\lambda \Lambda_{\lambda}^{n} =\Lambda_{\lambda}^{n-1} - \Lambda\Lambda_{\lambda}^{n}, \quad \forall n \in \mathbb{N} $. Therefore,   
\begin{equation*}
    \begin{split}
        \hat{\Lambda}_{\lambda}^{n} -\Lambda_{\lambda}^{n}  = & \hat{\Lambda}_{\lambda} (\Lambda - \hat{\Lambda}_{n})\Lambda_{\lambda}^{n} + \hat{\Lambda}_{\lambda} (\hat{\Lambda}_{\lambda}^{n-1}-\Lambda_{\lambda}^{n-1}).
    \end{split}
\end{equation*}
Doing similar step $n-1$ times, we get
\begin{equation*}
    \begin{split}
        \hat{\Lambda}_{\lambda}^{n} -\Lambda_{\lambda}^{n} = & \hat{\Lambda}_{\lambda}^{n-1}[\hat{\Lambda}_{\lambda} -\Lambda_{\lambda}] + \sum_{i=1}^{n-1} \hat{\Lambda}_{\lambda}^{i}(\Lambda - \hat{\Lambda}_{n})\Lambda_{\lambda}^{n+1-i}
    \end{split}
\end{equation*}
and the result follows.
\end{proof}

\begin{lemma}
\label{cmdifference}
    Suppose $T$ and $C$ commute and Assumptions \ref{as:1}, \ref{as:3} and \ref{as:5} hold. Then for all choices of $p$ such that $2p (t+c) \geq 2 \alpha t +c$, we have
    \textcolor{black}{
    \begin{equation*}
    \|(\Lambda + \lambda I)^{-p}T^{\alpha}(C-\hat{C}_{n})T^{\alpha}(\Lambda + \lambda I)^{-p}\|_{L^2 \to L^2} \lesssim_{p} \frac{1}{\sqrt{n}} \lambda^{-\frac{1+2p(t+c)-(2 \alpha t +c)}{t+c}}.
\end{equation*}}
\end{lemma}
\begin{proof}
Consider
\begin{equation*}
    \begin{split}
        &\|(\Lambda + \lambda I)^{-p}T^{\alpha}(C-\hat{C}_{n})T^{\alpha}(\Lambda + \lambda I)^{-p}\|_{L^2 \to L^2} \\
        = & \sup_{h: \|h\|_{L^2}=1} \langle h,(\Lambda + \lambda I)^{-p}T^{\alpha}(C-\hat{C}_{n})T^{\alpha}(\Lambda + \lambda I)^{-p}h \rangle_{L^2}\\
        = & \sup_{h: \|h\|_{L^2}=1} \langle T^{\alpha}(\Lambda + \lambda I)^{-p} h,(C-\hat{C}_{n})T^{\alpha}(\Lambda + \lambda I)^{-p}h \rangle_{L^2}.
    \end{split}
\end{equation*}
Using the spectral representation of $T$ and $C$, we get
\begin{equation*}
    \begin{split}
        & \langle T^{\alpha}(\Lambda + \lambda I)^{-p} h,(C-\hat{C}_{n})T^{\alpha}(\Lambda + \lambda I)^{-p}h \rangle_{L^2} \\
        = &  \left\langle \sum_{j \geq 1} \frac{\mu_{j}^{\alpha}\langle h, \phi_{j} \rangle_{L^2}}{(\mu_{j}\xi_{j}+\lambda)^{p}}\phi_{j} , \sum_{k \geq 1} \frac{\mu_{k}^{\alpha}\langle h, \phi_{k} \rangle_{L^2}}{(\mu_{k}\xi_{k}+\lambda)^{p}}(C-\hat{C}_{n})\phi_{k} \right\rangle_{L^2}\\
        = & \sum_{j,k \geq 1} \frac{\mu_{j}^{\alpha}\mu_{k}^{\alpha}\langle h, \phi_{j} \rangle_{L^2}\langle h, \phi_{k} \rangle_{L^2}}{(\mu_{j}\xi_{j}+\lambda)^{p}(\mu_{k}\xi_{k}+\lambda)^{p}} \langle \phi_{j}, (C-\hat{C}_{n})\phi_{k}\rangle_{L^2}\\
        \le &\left(\sum_{j,k \geq 1} \frac{\mu_{j}^{2\alpha}\mu_{k}^{2\alpha}}{(\mu_{j}\xi_{j}+\lambda)^{2p}(\mu_{k}\xi_{k}+\lambda)^{2p}} \langle \phi_{j}, (C-\hat{C}_{n})\phi_{k}\rangle^2_{L^2}\right)^{\frac{1}{2}}\Vert h\Vert_{L^2},
    \end{split}
\end{equation*}
where we applied Cauchy-Schwartz inequality in the last inequality. Therefore, 
\begin{equation*}
    \begin{split}
        & \mathbb{E}\|(\Lambda + \lambda I)^{-p}T^{\alpha}(C-\hat{C}_{n})T^{\alpha}(\Lambda + \lambda I)^{-p}\|_{L^2 \to L^2} \\
        \leq & \left(\sum_{j,k \geq 1} \frac{\mu_{j}^{2\alpha}\mu_{k}^{2\alpha}}{(\mu_{j}\xi_{j}+\lambda)^{2p}(\mu_{k}\xi_{k}+\lambda)^{2p}}\mathbb{E}\langle\phi_{j}, (C-\hat{C}_{n})\phi_{k}\rangle^2_{L^2}\right)^{\frac{1}{2}},
    \end{split}
\end{equation*}
which follows through an application of Jensen's inequality.
Now we consider 
\begin{equation*}
    \begin{split}      &\mathbb{E}\langle\phi_{j}, (C-\hat{C}_{n})\phi_{k}\rangle^2_{L^2} =   \mathbb{E}\langle\phi_{j},\frac{1}{n}\sum_{i=1}^{n}(C-X_{i} \otimes X_{i} )\phi_{k}\rangle^2_{L^2}\\
        = & \mathbb{E}\left[\frac{1}{n}\sum_{i=1}^{n}\langle \phi_{j}, (C- X_{i} \otimes X_{i})\phi_{k}\rangle_{L^2}\right]^2
        \le \frac{1}{n} \mathbb{E} \langle \phi_{j}, (X \otimes X)\phi_{k} \rangle_{L^2}^2\\
         =&  \frac{1}{n} \mathbb{E} \langle X, (\phi_{j} \otimes \phi_{k})X \rangle_{L^2}^2=\textcolor{black}{\mathbb{E}[\langle X, \phi_{j} \rangle^2_{L^2} \langle X, \phi_{k} \rangle^2_{L^2}] = \mathbb{E}[x_{j}^2 x_{k}^2],}
    \end{split}
\end{equation*}
\textcolor{black}{where we used the representation from  Theorem~\ref{KLrep} that $X = \sum x_{i}\phi_{i}$ with $\mathbb{E}[x_{i}] = 0 $ and $\mathbb{E}[x_{i}x_{j}] = \delta_{ij}\xi_{i}$.}
\textcolor{black}{Using Cauchy-Schwartz inequality and Assumption \ref{as:5}, we have
\begin{equation*}
\mathbb{E} [x^2_jx^2_k] \leq (\mathbb{E}[x_{j}^4] \mathbb{E}[x_{k}^4])^{\frac{1}{2}} \lesssim \xi_{j} \xi_{k}.
\end{equation*}
}
\textcolor{black}{Therefore,
\begin{equation*}
    \begin{split}
&\mathbb{E}\|(\Lambda + \lambda I)^{-p}T^{\alpha}(C-\hat{C}_{n})T^{\alpha}(\Lambda + \lambda I)^{-p}\|_{L^2 \to L^2} \\ \lesssim & \frac{1}{\sqrt{n}} \left(\sum_{j,k \geq 1}\frac{\mu_{j}^{2 \alpha}\mu_{k}^{2 \alpha} \xi_{j} \xi_{k}}{(\mu_{j}\xi_{j}+\lambda)^{2p}(\mu_{k}\xi_{k}+\lambda)^{2p}}\right)^{\frac{1}{2}} 
        =  \frac{1}{\sqrt{n}} \sum_{j \geq 1}\frac{\mu_{j}^{2 \alpha} \xi_{j}}{(\mu_{j}\xi_{j}+\lambda)^{2p}}\\
         \lesssim & \frac{1}{\sqrt{n}} \sum_{j \geq 1} \frac{j^{-2 \alpha t} j^{-c}}{(j^{-(t+c)}+\lambda)^{2p}}\lesssim \frac{1}{\sqrt{n}} \lambda^{-\frac{1+2p(t+c)-(2 \alpha t +c)}{t+c}},
    \end{split}
\end{equation*}
where the last inequality follows from Lemma \ref{seriessum} for $2p(t+c) \geq 2 \alpha t +c$.
}
The result follows from an application of Markov's inequality.
\end{proof}

\begin{lemma}
\label{ncmdifference}
    Suppose Assumptions \ref{as:2},
    \ref{as:4} and \ref{as:5} hold. Then for $p\geq \frac{1}{2}$, we have
    \textcolor{black}{
    \begin{equation*}
    \|(\Lambda+\lambda I)^{-p}(\Lambda-\hat{\Lambda}_{n})(\Lambda+\lambda I)^{-p}\|_{L^2 \to L^2} \lesssim_{p} \frac{1}{\sqrt{n}} \lambda^{-\frac{1+2bp-b}{b}}.
\end{equation*}
}
\end{lemma}
\begin{proof}
Following similar steps as in Lemma \ref{cmdifference}, we have
    \begin{equation*}
        \begin{split}
         & \mathbb{E}\|(\Lambda+\lambda I)^{-p}(\Lambda-\hat{\Lambda}_{n})(\Lambda+\lambda I)^{-p}\|_{L^2 \to L^2}\\
         \lesssim & \left(\sum_{j,k \geq 1} \frac{1}{(\tau_{j}+\lambda)^{2p}(\tau_{k}+\lambda)^{2p}}\mathbb{E}\langle e_{j}, (\Lambda-\hat{\Lambda}_{n})e_{k}\rangle^2_{L^2}\right)^{\frac{1}{2}},
        \end{split}
    \end{equation*}
    where $(\tau_j,e_j)_j$ is the eigensystem of $\Lambda$. 
Note that \begin{equation*}
        \begin{split}
\mathbb{E}\langle e_{j}, (\Lambda-\hat{\Lambda}_{n})e_{k}\rangle^2_{L^2} = \frac{1}{n} \mathbb{E}\langle T^{\frac{1}{2}}X, (e_{j}\otimes e_{k})T^{\frac{1}{2}}X\rangle_{L^2}^2.
        \end{split}
    \end{equation*}
\textcolor{black}{Since $\mathbb{E}[T^{\frac{1}{2}}X \otimes T^{\frac{1}{2}}X]= \Lambda$, by Theorem \ref{KLrep}, we have $T^{\frac{1}{2}}X = \sum_i{x_{i}e_{i}}$, where $\mathbb{E}[x_{i}]= 0$ and $\mathbb{E}[x_{i}x_{j}]= \delta_{ij}\tau_{i}$.} Then using similar ideas as in Lemma~\ref{cmdifference}, we get
\textcolor{black}{
\begin{equation*}
    \begin{split}
        \mathbb{E}\langle T^{\frac{1}{2}}X, (e_{j}\otimes e_{k})T^{\frac{1}{2}}X\rangle_{L^2}^2 \lesssim \tau_{j} \tau_{k}
    \end{split}
\end{equation*}
resulting in \begin{equation*}
    \begin{split}        &\|(\Lambda+\lambda I)^{-p}(\Lambda-\hat{\Lambda}_{n})(\Lambda+\lambda I)^{-p}\|_{L^2 \to L^2} \lesssim_{p}  \frac{1}{\sqrt{n}}\sum_{j}\frac{\tau_{j}}{(\tau_{j}+\lambda)^{2p}}\\
        \lesssim{} & \frac{1}{\sqrt{n}} \sum_{j}\frac{j^{-b}}{(j^{-b}+\lambda )^{2p}}
        \lesssim  \frac{1}{\sqrt{n}} \lambda^{-\frac{1+2bp-b}{b}},
    \end{split}
\end{equation*}
}
provided $p \geq \frac{1}{2}$.
\end{proof}

\begin{lemma}
\label{Covariance estimation}
Let $\delta \in (0,1)$ and $\lambda \geq c_{\delta} n^{-\frac{b}{b+1}}$ for some positive constant $c_{\delta}$. Under Assumptions \ref{as:4} and \ref{as:5}, with probability at least $1-\delta$, the following holds:

\begin{enumerate}[label=(\alph*)]
    \item  $\|(\hat{\Lambda}_{n}-\Lambda)(\Lambda+\lambda I)^{-1}\|_{L^2 \to L^2} \leq \frac{1}{2}$;
    \item $\|(\Lambda+\lambda I)^{l}(\hat{\Lambda}_{n}+\lambda I)^{-l}\|_{L^2 \to L^2} \leq 2^{l}, \quad~ \forall l \in [0,1]$;
    \item $\|(\Lambda+\lambda I)^{-l}(\hat{\Lambda}_{n}+\lambda I)^{l}\|_{L^2 \to L^2} \leq (\frac{3}{2})^{l}, \quad~ \forall l \in [0,1]$.
\end{enumerate}

\end{lemma}
\begin{proof}
    We start by considering
\begin{align*}
&\|(\hat{\Lambda}_{n}-\Lambda)(\Lambda+\lambda I)^{-1}\|_{L^2 \to L^2}=
    \|(\Lambda+\lambda I)^{-1}(\hat{\Lambda}_{n}-\Lambda)\|_{L^2 \to L^2}\\
    &=  \sup_{h,g : \|h\|_{L^2},\|g\|_{L^2}=1} |\langle (\Lambda+\lambda I)^{-\frac{1}{2}}g, (\Lambda+\lambda I)^{-\frac{1}{2}}(\hat{\Lambda}_{n}-\Lambda)h\rangle_{L^2}|.
 \end{align*}
    Using the fact that $h = \sum_{j\geq 1} \langle  h,e_{j}\rangle_{L^2}e_{j}$ and $g = \sum_{j\geq 1} \langle  g,e_{j}\rangle_{L^2}e_{j}$, we get
\begin{equation*}
    \begin{split}
        & \langle (\Lambda+\lambda I)^{-\frac{1}{2}}g, (\Lambda+\lambda I)^{-\frac{1}{2}}(\hat{\Lambda}_{n}-\Lambda)h\rangle_{L^2}\\
        &  =  \left\langle \sum_{j\geq 1} \frac{\langle g,e_{j}\rangle_{L^2}e_{j}}{(\tau_{j}+\lambda)^{\frac{1}{2}}}, \sum_{k,l\geq 1}\frac{\langle  h,e_{k}\rangle_{L^2}\langle e_l,(\hat{\Lambda}_{n}-\Lambda)e_{k}\rangle_{L^2}e_l}{(\tau_{l}+\lambda)^{\frac{1}{2}}}\right\rangle_{L^2}\\
         &  =  \sum_{j,k \geq 1} \frac{\langle  g,e_{j}\rangle_{L^2}\langle  h,e_{k}\rangle_{L^2}}{\tau_j+\lambda}
         \langle e_{j}, (\hat{\Lambda}_{n}-\Lambda)e_{k}\rangle_{L^2}.
    \end{split}
\end{equation*}
Applying Cauchy–Schwartz inequality yields
\begin{equation*}
    \begin{split}
        \|(\Lambda+\lambda I)^{-1}(\hat{\Lambda}_{n}-\Lambda)\|_{L^2 \to L^2} \leq \left(\sum_{j,k \geq 1} \frac{1}{(\tau_j+\lambda)^2}
        \langle e_{j}, (\hat{\Lambda}_{n}-\Lambda)e_{k}\rangle_{L^2}^2\right)^\frac{1}{2}.
    \end{split}
\end{equation*}
Now taking expectations on both sides and applying Jensen's Inequality, we have
\begin{equation*}
\mathbb{E}\|(\Lambda+\lambda I)^{-1}(\hat{\Lambda}_{n}-\Lambda)\|_{L^2 \to L^2} \leq \left(\sum_{j,k \geq 1} \frac{1}{(\tau_j+\lambda)^2}
        \mathbb{E}\langle e_{j}, (\hat{\Lambda}_{n}-\Lambda)e_{k}\rangle_{L^2}^2\right)^\frac{1}{2}.
\end{equation*}
From Lemma \ref{ncmdifference}, we can see that
\textcolor{black}{
\begin{equation*}
    \mathbb{E}\langle e_{j}, (\hat{\Lambda}_{n}-\Lambda)e_{k}\rangle_{L^2}^2 \lesssim \frac{1}{n} \tau_{j} \tau_{k}.
\end{equation*}
}
Therefore, an application of Markov's inequality yields
\textcolor{black}{
\begin{equation*}
\begin{split}
    \|(\Lambda+\lambda I)^{-1}(\hat{\Lambda}_{n}-\Lambda)\|_{L^2 \to L^2} \lesssim_{p} & \frac{1}{\sqrt{n}}\left(\sum_{j, k \geq 1} \frac{\tau_{j} \tau_{k}}{(\tau_j+\lambda)^2} \right)^{\frac{1}{2}} \lesssim \frac{1}{\sqrt{n}}\left(\sum_{j \geq 1} \frac{\tau_{j}}{(\tau_j+\lambda)^2} \right)^{\frac{1}{2}} \\
    \lesssim & \frac{1}{\sqrt{n}}\left(\sum_{j \geq 1} \frac{j^{-b}}{(j^{-b}+\lambda)^2} \right)^{\frac{1}{2}}
 \lesssim \frac{1}{\sqrt{n}} \lambda^{-\frac{1}{2b}-\frac{1}{2}}
    \end{split}
\end{equation*}
}
and the result in $(a)$ follows under \textcolor{black}{$\lambda\gtrsim n^{-\frac{b}{b+1}}$}.
As a consequence of part $(a)$, we get
\begin{equation*}
    \begin{split}
        &\|(\Lambda+\lambda I)(\hat{\Lambda}_{n}+\lambda I)^{-1}\|_{L^2 \to L^2} 
         = \|[(\hat{\Lambda}_{n}-\Lambda)(\Lambda+\lambda I)^{-1}+I]^{-1}\|_{L^2 \to L^2} \\
         & \leq  \frac{1}{1-\|(\hat{\Lambda}_{n}-\Lambda)(\Lambda+\lambda I)^{-1}\|_{L^2 \to L^2}}
         \leq  2, \quad \forall  \lambda \gtrsim \textcolor{black}{n^{-\frac{b}{b+1}}},
    \end{split}
\end{equation*}
and part $(b)$ follows from combining $(a)$ with Lemma \ref{cordes} $(\text{Corde's inequality})$. For part $(c)$, note that
\begin{equation*}
    \begin{split}
        \|(\Lambda+\lambda I)^{-1}(\hat{\Lambda}_{n}+\lambda I)\|_{L^2 \to L^2} = & \|I - (\Lambda-\hat{\Lambda}_{n})(\Lambda+\lambda I)^{-1}\|_{L^2 \to L^2}\\
        \leq &  1 +  \|(\hat{\Lambda}_{n}- \Lambda)(\Lambda+\lambda I)^{-1}\|_{L^2 \to L^2}\\
        \leq & \frac{3}{2}, \quad \forall  \textcolor{black}{\lambda \gtrsim n^{-\frac{b}{b+1}}},
    \end{split}
\end{equation*}
and applying Lemma \ref{cordes} $(\text{Corde's inequality})$ yields the result.
\end{proof}

\textcolor{black}{
\begin{remark}
    From Lemma \ref{ncmdifference}, we have $$\|(\Lambda+\lambda I)^{-\frac{1}{2}}(\Lambda-\hat{\Lambda}_{n})(\Lambda+\lambda I)^{-\frac{1}{2}}\|_{L^2 \to L^2} \lesssim_{p} \frac{1}{\sqrt{n}} \lambda^{-\frac{1}{b}}.$$ Since $$\|(\Lambda + \lambda I)^{\frac{1}{2}} (\hat{\Lambda} + \lambda I)^{-\frac{1}{2}}\|_{L^2 \to L^2}^2 \leq \frac{1}{1-\|(\Lambda+\lambda I)^{-\frac{1}{2}}(\Lambda-\hat{\Lambda}_{n})(\Lambda+\lambda I)^{-\frac{1}{2}}\|_{L^2 \to L^2}},$$ we obtain
    \begin{equation*}
        \|(\Lambda + \lambda I)^{\frac{1}{2}} (\hat{\Lambda} + \lambda I)^{-\frac{1}{2}}\|_{L^2 \to L^2} \leq \sqrt{2} \quad \forall \lambda \gtrsim n^{-\frac{b}{2}}.
    \end{equation*}
   From a similar calculation as in part $(c)$ of Lemma \ref{Covariance estimation}, we can observe that
    \begin{equation*}
        \|(\Lambda + \lambda I)^{-\frac{1}{2}} (\hat{\Lambda} + \lambda I)^{\frac{1}{2}}\|_{L^2 \to L^2} \leq \frac{3}{2} \quad \forall \lambda \gtrsim n^{-\frac{b}{2}}.
    \end{equation*}
\end{remark}
}

\begin{lemma}\cite[Lemma 5.1]{cordes1987}
\label{cordes}
Suppose $T_1$ and $T_2$ are two positive bounded linear operators on a separable Hilbert space. Then
$$\|T_1^pT_2^p\| \leq \|T_1T_2\|^p, \text{ when } 0\leq p \leq 1. $$
\end{lemma}
\noindent
\begin{lemma}\cite[Lemma A.6]{balasubramanian2024functional}
\label{supppbound}
    For any $0 < \alpha \leq \beta$,
    $$\sup_{i\in \mathbb{N}}\left[\frac{i^{-\alpha}}{i^{-\beta}+\lambda}\right] \leq \lambda^{\frac{\alpha-\beta}{\beta}}, ~~ \forall ~\lambda >0.$$
\end{lemma}

\begin{lemma}
\label{supbound}
Let $a,m,p,q$ and $l$ be positive numbers. Then for $r = \min \{p,\frac{lq}{m}+a\}$ and $a< p$, we have
    $$\sup_{i \in \mathbb{N}}\left[\frac{i^{-(p-a)m}}{(i^{-q}+\lambda)^l}\right] \leq \lambda^{\frac{(r-a)m-ql}{q}},~ \forall~ \lambda >0.$$
\end{lemma}
\begin{proof} Consider
    \begin{equation*}
    \begin{split}
        \sup_{i \in \mathbb{N}}\left[\frac{i^{-(p-a)m}}{(i^{-q}+\lambda)^l}\right]  = & \left(\sup_{i \in \mathbb{N}}\left[\frac{i^{-(p-a)\frac{m}{l}}}{i^{-q}+\lambda}\right] \right)^l.
    \end{split}
    \end{equation*}
If $p \leq \frac{lq}{m}+a$, i.e., $(p-a)\frac{m}{l} \leq q$, then it follows from Lemma \ref{supppbound} that
\begin{equation*}
    \sup_{i \in \mathbb{N}}\left[\frac{i^{-(p-a)m}}{(i^{-q}+\lambda)^l}\right] \leq \lambda^{\frac{(p-a)m-ql}{q}}.
\end{equation*}
If $p > \frac{lq}{m}+a$, i.e., $(p-a)\frac{m}{l} > q$, then 
\begin{equation*}
\begin{split}
    \sup_{i \in \mathbb{N}}\left[\frac{i^{-(p-a)m}}{(i^{-q}+\lambda)^l}\right] = & \left(\sup_{i \in \mathbb{N}}\left[\frac{i^{-q}}{(i^{-q}+\lambda)}\right]\right)^l \left(\sup_{i \in \mathbb{N}}\left[i^{-\left(\frac{(p-a)m}{l}-q\right)}\right]\right)^l\leq 1,
    \end{split}
\end{equation*}
where the last step follows from Lemma \ref{supppbound} and the fact that $\frac{(p-a)m}{l}-q >0$.
\end{proof}

\begin{lemma}
\label{seriessum}
    For $\alpha >1$, $\beta >1,$ and $q \geq \frac{\alpha}{\beta}$, we have 
    $$\sum_{i\in \mathbb{N}}\frac{i^{-\alpha}}{(i^{-\beta}+\lambda)^q} \lesssim \lambda^{-\frac{1+\beta q -\alpha}{\beta}}, ~~ \forall ~\lambda >0.$$
\end{lemma}
\begin{proof}
Consider
\begin{equation*}
        \sum_{i \in \mathbb{N}} \frac{i^{-\alpha}}{(i^{-\beta}+\lambda)^{q}} = \sum_{i \in \mathbb{N}} \frac{i^{q\beta-\alpha}}{(1+\lambda i^{\beta})^{q}}  \lesssim \int_{0}^{\infty} \frac{x^{q \beta-\alpha}}{(1+\lambda x^{\beta})^{q}}dx
        =  \lambda^{-\frac{1+q \beta -\alpha }{\beta}} \int_{0}^{\infty} \frac{y^{q \beta-\alpha}}{(1+y^{\beta})^{q}} dy.
\end{equation*}
Note that
\begin{equation*}
    \begin{split}
        \int_{0}^{\infty} \frac{y^{q \beta-\alpha}}{(1+y^{\beta})^{q}} dy =& \int_{0}^{\infty} \frac{(y^{\beta})^{q-\frac{\alpha}{\beta}}}{(1+y^{\beta})^{q-\frac{\alpha}{\beta}}(1+y^{\beta})^{\frac{\alpha}{\beta}}}dy 
        \leq  \int_{0}^{\infty} \frac{1}{(1+y^{\beta})^{\frac{\alpha}{\beta}}} dy\\
        = & \int_{0}^{1} \frac{1}{(1+y^{\beta})^{\frac{\alpha}{\beta}}} dy +\int_{1}^{\infty} \frac{1}{(1+y^{\beta})^{\frac{\alpha}{\beta}}} dy. 
    \end{split}
\end{equation*}
For $y\ge 1$, $\frac{1}{1+y^\beta} \le \frac{1}{y^\beta}$ and therefore, $\int_{1}^{\infty} \frac{1}{(1+y^{\beta})^{\frac{\alpha}{\beta}}} dy \le \int^\infty_1 y^{-\alpha}\,dy=\frac{1}{\alpha-1}$. On the other hand, $\int_{0}^{1} \frac{1}{(1+y^{\beta})^{\frac{\alpha}{\beta}}} dy\le 1$. 
The result therefore follows.
\end{proof}

\begin{theorem}[Karhunen-Lo\`{e}ve Expansion]\cite[Corollary 7]{RKHSinprob2004stats}
\label{KLrep}
   \textcolor{black}{ Let $X_{t}$ be a zero-mean square-integrable stochastic process defined over a probability space $(\Omega, \mathcal{F}, P)$ and indexed over a closed and bounded interval $[a, b]$, with continuous covariance function $K_{X}(s, t)$. Then $K_{X}(s, t)$ is a Mercer kernel and letting $e_{k}$ be an orthonormal basis on $L^2([a, b])$ formed by the eigenfunctions of  $T_{K_{X}}: L^2([a,b])\rightarrow L^2([a,b]),\,f\mapsto \int^b_a K_X(\cdot,t)f(t)\,dt$ with respective eigenvalues $\lambda_{k}$, $X_{t}$ admits the following representation
    $$X_{t} = \sum _{k=1}^{\infty }Z_{k}e_{k}(t),$$
where the convergence is in $L^2$, uniform in $t$ and
$$Z_{k}=\int _{a}^{b}X_{t}e_{k}(t)dt.$$
Furthermore, the random variables $Z_k$ have zero mean, are uncorrelated and have variance $\lambda_{k}$, i.e., 
$$\mathbf {E} [Z_{k}]=0,~\forall k\in \mathbb {N} \qquad {\mbox{and}}\qquad \mathbf {E} [Z_{i}Z_{j}]=\delta _{ij}\lambda _{j},~\forall i,j\in \mathbb {N}.$$}
\end{theorem}
\textcolor{black}{\section{Kullback-Leibler divergence formula}\label{subsec:kl}
    Let $P_{\theta}, P_{\theta'}$ be joint probability distribution of $\{(X_{i}, Y_{i}): 1 \leq i \leq n \}$ for $\beta^* = f_{\theta}$ and $\beta^* = f_{\theta'}$ respectively. Let $\epsilon$ follow a normal distribution with 0 mean and $\sigma^2$ variance. Then $P_{\theta}(Y|X) \sim \mathcal{N}(\langle X, f_{\theta}\rangle_{L^2}, \sigma^2)$ and $P_{\theta'}(Y|X) \sim \mathcal{N}(\langle X, f_{\theta'}\rangle_{L^2}, \sigma^2)$. So
    \begin{equation*}
        \begin{split}
\log\left(\frac{P_{\theta}}{P_{\theta'}}\right) = & \log\left(\frac{P_{\theta}(Y|X)}{P_{\theta'}(Y|X)}\right) \\
            = & \frac{1}{2 \sigma^2}\sum_{i=1}^{n} [(Y_{i}-\langle X_{i}, f_{\theta'} \rangle_{L^2})^2 - (Y_{i}-\langle X_{i}, f_{\theta} \rangle_{L^2})^2]\\
            = & \frac{1}{2 \sigma^2}\sum_{i=1}^{n}[\langle X_{i}, f_{\theta}-f_{\theta'} \rangle_{L^2}(2Y_{i}-\langle X_{i}, f_{\theta}+f_{\theta'} \rangle_{L^2})]\\
            = & \frac{1}{\sigma^2} \sum_{i=1}^{n}[Y_{i}\langle X_{i}, f_{\theta}-f_{\theta'} \rangle_{L^2}] - \frac{1}{2 \sigma^2}\sum_{i=1}^{n}[\langle X_{i}, f_{\theta}-f_{\theta'} \rangle_{L^2}\langle X_{i}, f_{\theta}+f_{\theta'} \rangle_{L^2}]
            \end{split}
        \end{equation*}
        \begin{equation*}
        \begin{split}
            =& \frac{1}{\sigma^2} \sum_{i=1}^{n}[(Y_{i}-\langle X_{i}, f_{\theta} \rangle_{L^2})\langle X_{i}, f_{\theta}-f_{\theta'} \rangle_{L^2}] + \frac{1}{2 \sigma^2}\sum_{i=1}^{n}\langle X_{i}, f_{\theta}- f_{\theta'}\rangle_{L^2}^2,
        \end{split}
    \end{equation*}
yielding
\begin{equation*}
    \begin{split}
        \mathcal{K}(P_{\theta}, P_{\theta'}) = & \int\log\left(\frac{P_{\theta}}{P_{\theta'}}\right)dP_{\theta}\\
        = & \frac{n}{2 \sigma^2} \mathbb{E}[\langle X, f_{\theta}- f_{\theta'}\rangle_{L^2}^2] = \frac{n}{2 \sigma^2} \|C^{\frac{1}{2}}(f_{\theta}-f_{\theta'})\|_{L^2}^2.
        \end{split}
\end{equation*}}

\section{Versions of Lemmas \ref{cmdifference} and \ref{ncmdifference} under the boundedness of $X$}\label{equivalentbounds}
In this appendix, we provide the equivalent versions of Lemmas~\ref{cmdifference} and \ref{ncmdifference} under the boundedness assumption on $X$.
\begin{lemma}
\label{cmdifnew}
    Suppose $T$ and $C$ commute and Assumptions \ref{as:1} and  \ref{as:3} hold. Also assume that $\sup_{\omega \in \Omega}\|X(\cdot, \omega)\|_{L^2} < \infty$. Then for all choices of $p$ such that $2p (t+c) \geq 2 \alpha t +c$, we have
    \begin{equation*}
    \|(\Lambda + \lambda I)^{-p}T^{\alpha}(C-\hat{C}_{n})T^{\alpha}(\Lambda + \lambda I)^{-p}\|_{L^2 \to L^2} \lesssim_{p} \frac{\lambda^{-2p}}{n} \lambda^{\frac{\alpha t}{t+c}} + \frac{1}{\sqrt{n}} \lambda^{-\frac{1+4p(t+c)-(4 \alpha t +c)}{2(t+c)}}.
\end{equation*}
\end{lemma}
\begin{proof}
Consider
    \begin{align}\label{Eq:ttt}
         &\|(\Lambda + \lambda I)^{-p}T^{\alpha}(C-\hat{C}_{n})T^{\alpha}(\Lambda + \lambda I)^{-p}\|_{L^2 \to L^2}\nonumber\\
        &  \qquad \qquad \leq \|(\Lambda + \lambda I)^{-p}T^{\alpha}(C-\hat{C}_{n})\|_{L^2 \to L^2}\|T^{\alpha}(\Lambda + \lambda I)^{-p}\|_{L^2 \to L^2}\nonumber\\
        & \qquad \qquad \leq \lambda^{\frac{\alpha t - p(t+c)}{t+c}}  \|(\Lambda + \lambda I)^{-p}T^{\alpha}(C-\hat{C}_{n})\|_{L^2 \to L^2},
    \end{align}
    where last step follows from the fact that
    \begin{equation*}
        \begin{split}
            \|T^{\alpha}(\Lambda + \lambda I)^{-p}\|_{L^2 \to L^2} = \sup_{i}\left|\frac{\mu_{i}^{\alpha}}{(\mu_{i}\xi_{i} +\lambda)^p}\right| \lesssim \sup_{i}\left|\frac{i^{-\frac{\alpha t}{p}}}{i^{-(t+c)}+\lambda}\right|^p \leq \lambda^{\frac{\alpha t - p(t+c)}{t+c}}.
        \end{split}
    \end{equation*}
    Let $\xi = (\Lambda + \lambda I)^{-p}T^{\alpha} X \otimes X$ be a random element in $HS(L^2)$, the Hilbert space of Hilbert–Schmidt operators on $L^2$, with inner product $\langle A, B \rangle_{HS} = \text{Tr}(A^{\top}B)$. Considering the $HS$-norm of $\xi$, we have
\begin{equation*}
    \begin{split}
        \|\xi\|_{HS}^2 = &  \sum_{m \geq 1} \|(\Lambda + \lambda I)^{-p}T^{\alpha} \langle X, \phi_{m}\rangle_{L^2} X\|_{L^2}^2\\
        \leq & \|(\Lambda + \lambda I)^{-p}\|^2_{L^2 \mapsto L^2} \|T^{\alpha}X\|_{L^2}^2 \sum_{m}|\langle X, \phi_{m} \rangle_{L^2}|^2
        \lesssim  \lambda^{-2p}.
    \end{split}
\end{equation*}
Next, using the representation $X = \sum_{k \geq 1} \langle X, \phi_{k}\rangle_{L^2} \phi_{k}$ from Lemma~\ref{KLrep} in 
\begin{equation*}
        \|\xi\|_{HS}^2 =  \sum_{m \geq 1} \|(\Lambda + \lambda I)^{-p}T^{\alpha} \langle X, \phi_{m} \rangle_{L^2} X\|_{L^2}^2,
\end{equation*}
we have, 
\begin{align*}
    \begin{split}
        \|\xi\|_{HS}^2 &=\|(\Lambda + \lambda I)^{-p}T^{\alpha}  X\|_{L^2}^2\sum_{m \geq 1} \langle X, \phi_{m} \rangle^2_{L^2} \\
        &=  \|X\|_{L^2}^2 \left \|\sum_{k\geq 1} \frac{\mu_{k}^{\alpha} \langle X, \phi_{k} \rangle_{L^2} \phi_{k}}{(\mu_{k} \xi_{k} +\lambda)^p} \right\|^2_{L^2}
        \lesssim  \sum_{k \geq 1} \frac{\mu_{k}^{2 \alpha} \langle X, \phi_{k} \rangle^2_{L^2}}{(\mu_{k} \xi_{k} +\lambda)^{2p}}.
    \end{split}
\end{align*}
By taking expectations on both sides, we get
\begin{equation*}
    \begin{split}
        \mathbb{E}[ \|\xi\|_{HS}^2] \lesssim & \sum_{k \geq 1} \frac{\mu_{k}^{2 \alpha} \mathbb{E}\langle X \otimes X \phi_{k}, \phi_{k} \rangle}{(\mu_{k} \xi_{k} +\lambda)^{2p}}\\
        = &  \sum_{k \geq 1} \frac{\mu_{k}^{2 \alpha} \langle C \phi_{k}, \phi_{k} \rangle}{(\mu_{k} \xi_{k} +\lambda)^{2p}} =  \sum_{k \geq 1} \frac{\mu_{k}^{2 \alpha} \xi_{k}}{(\mu_{k} \xi_{k} +\lambda)^{2p}}\\
        \lesssim & \sum_{k \geq 1} \frac{k^{-2\alpha t }k^{-c}}{(k^{-(t+c)}+\lambda)^{2p}} \lesssim \lambda^{-\frac{1+2p(t+c)-(2 \alpha t+ c )}{t+c}}.
    \end{split}
\end{equation*}

Then it follows from \cite[Theorem 6.14]{SVM2008steinwart} that
$$\|(\Lambda + \lambda I)^{-p}T^{\alpha}(C - \hat{C}_{n})\|_{L^2 \mapsto L^2} \lesssim_{p} \frac{\lambda^{-p}}{n} + \frac{1}{\sqrt{n}}\lambda^{-\frac{1+2p(t+c)-(2 \alpha t +c)}{2(t+c)}},$$
using which in \eqref{Eq:ttt} yields the result.
\end{proof}

\vspace{5mm}
\begin{lemma}
\label{ncmdiffnew}
    Suppose Assumptions \ref{as:2} and
    \ref{as:4} hold. Also assume that $\sup_{\omega \in \Omega}\|X(\cdot, \omega)\|_{L^2} < \infty$. Then for $p\geq \frac{1}{2}$, we have
    \begin{equation*}
    \|(\Lambda+\lambda I)^{-p}(\Lambda-\hat{\Lambda}_{n})(\Lambda+\lambda I)^{-p}\|_{L^2 \to L^2} \lesssim_{p} \frac{\lambda^{-2p}}{n} + \frac{1}{\sqrt{n}} \lambda^{-\frac{1+4bp - b}{2b}}.
\end{equation*}
\end{lemma}
\begin{proof}
Consider
\begin{equation*}
    \|(\Lambda+\lambda I)^{-p}(\Lambda-\hat{\Lambda}_{n})(\Lambda+\lambda I)^{-p}\|_{L^2 \to L^2} \leq \frac{1}{\lambda^{p}} \|(\Lambda+\lambda I)^{-p}(\Lambda-\hat{\Lambda}_{n})\|_{L^2 \to L^2}.
\end{equation*}
Let $\xi = (\Lambda+\lambda I)^{-p} T^{\frac{1}{2}}X \otimes T^{\frac{1}{2}}X$ and consider its $HS$-norm,
\begin{equation*}
    \|\xi\|_{HS}^2 = \sum_{m} \|(\Lambda+\lambda I)^{-p} \langle T^{\frac{1}{2}}X , \phi_{m}\rangle_{L^2} T^{\frac{1}{2}}X\|^2_{L^2} \leq \frac{\|T^{\frac{1}{2}}X\|^4}{\lambda^{2p}}.
\end{equation*}
Next, using the representation $T^{\frac{1}{2}}X = \sum_{k \geq 1} \langle T^{\frac{1}{2}}X, e_{k} \rangle e_{k}$ from Lemma \ref{KLrep} in
\begin{equation*}
    \|\xi\|_{HS}^2 = \sum_{m\ge 1} \|(\Lambda+\lambda I)^{-p} \langle T^{\frac{1}{2}}X , \phi_{m}\rangle_{L^2} T^{\frac{1}{2}}X\|^2_{L^2},
\end{equation*}
we have
\begin{equation*}
    \begin{split}
        \|\xi\|_{HS}^2 = \|T^{\frac{1}{2}}X\|^2_{L^2} \left\|\sum_{k\geq 1} \frac{\langle T^{\frac{1}{2}}X, e_{k} \rangle_{L^2} e_{k}}{(\tau_{k} + \lambda)^{p}}\right\|^2 \lesssim \sum_{k \geq 1} \frac{\langle T^{\frac{1}{2}}X, e_{k}\rangle^2_{L^2}}{(\tau_{k} + \lambda)^{2p}}.
    \end{split}
\end{equation*}
By taking expectations on both sides, we get
\begin{equation*}
        \mathbb{E}[\|\xi\|_{HS}^2] \lesssim  \sum_{k \geq 1} \frac{\tau_{k}}{(\tau_{k} + \lambda)^{2p}} \lesssim \sum_{k \geq 1} \frac{k^{-b}}{(k^{-b}+\lambda)^{2p}}
        \lesssim  \lambda^{-\frac{1+2pb-b}{b}}.
\end{equation*}
The result therefore follows from  \cite[Theorem 6.14]{SVM2008steinwart}.
\end{proof}
\end{document}